\documentclass[11]{article}    

\usepackage{moreverb}
\usepackage{amssymb}
\usepackage{amsmath}
\usepackage{graphicx}

\newcommand\BibTeX{{\rmfamily B\kern-.05em \textsc{i\kern-.025em b}\kern-.08em
T\kern-.1667em\lower.7ex\hbox{E}\kern-.125emX}}

\usepackage{color}

\newtheorem{proposition}{\bf Proposition}[section]
\newtheorem{lemma}{\bf Lemma}[section]

\newtheorem{remark}{\bf Remark}[section]

\newtheorem{example}{\bf Example }[section]

\newenvironment{proof}{
\begin{trivlist}
\item[\hspace{\labelsep}{\bf\noindent Proof. }] }{\par\hfill\end{trivlist}
\par}

\date{\empty}

\title{
\huge\bf Continuous-time multi-type Ehrenfest model and related Ornstein-Uhlenbeck diffusion on a star graph\thanks{ 
To appear in {\bf Mathematical Methods in the Applied Sciences}.
}}

\author{
Antonio  {\bf Di Crescenzo}\footnote{Corresponding author \ -- \ Address: 
Dipartimento di Matematica, Universit\`a degli Studi di Salerno, Via Giovanni Paolo II n.\ 132, I-84084 Fisciano (SA), Italy \ -- \ Email: adicrescenzo@unisa.it \ -- \ ORCID: 0000-0003-4751-7341}
 \qquad  
Barbara {\bf Martinucci}\footnote{Address: 
Dipartimento di Matematica, Universit\`a degli Studi di Salerno, Via Giovanni Paolo II n.\ 132, I-84084 Fisciano (SA), Italy \ -- \ Email: bmartinucci@unisa.it \ -- \ ORCID: 0000-0001-8340-4200} 
 \qquad  
Serena {\bf Spina}\footnote{Address: 
Dipartimento di Matematica, Universit\`a degli Studi di Salerno, Via Giovanni Paolo II n.\ 132, I-84084 Fisciano (SA), Italy \ -- \ Email: sspina@unisa.it \ -- \ ORCID: 0000-0001-6408-7596} 
}

\begin{document}

\maketitle

\begin{abstract}
We deal with a continuous-time Ehrenfest model defined over an extended star graph, defined as a 
lattice formed by the integers of $d$ semiaxis joined at the origin. The dynamics on each ray are regulated 
by linear transition rates, whereas the switching among  rays at the origin occurs according to a general 
stochastic matrix. We perform a detailed investigation of the transient and asymptotic behavior of this process. 
We also obtain a diffusive approximation of the considered model, which leads to an 
Ornstein-Uhlenbeck diffusion process over a domain formed by semiaxis joined at the origin, named spider. 
We show that the approximating process possesses a truncated Gaussian stationary density. Finally, the 
goodness of the approximation is discussed through comparison of stationary distributions, means and variances. 

\smallskip\noindent
{\em Keywords}: Branching processes; Diffusion processes;
Ehrenfest model; Ornstein-Uhlenbeck process; Stationary distribution 
\end{abstract}

% ==============================================================
\section{Introduction}\label{Section:Introduction}
% ==============================================================
%
The celebrated Ehrenfest model is a Markov chain over a finite state space, with linearly state-depending transition rates 
and reflecting endpoints, that was suitably proposed to describe the diffusion of gas molecules in a container. It is widely 
studied as a prototype for random motions in physics and in applied sciences, and for modeling random phenomena 
in thermodynamics and chemistry (see, for instance, Balaji et al.\ \cite{Balaji2010} and Flegg et al.\ \cite{Flegg2008}). 
Modified versions of the basic model have been considered such that (i) a general probabilistic rule holds for the system state 
change (cf.\ Hauert et al.\ \cite{Hauert2004}), (ii) the presence of additional large jumps is used to explain certain features 
emerging in finance for returns in stock index prices and exchange rates for currencies (cf.\ Takahashi \cite{Takahashi2004}), 
(iii) catastrophes occurring at constant rate force the system to reset into state 0 (cf.\ Dharmaraja et al.\ \cite{Dharmaraja2015}). 
\par
In this paper we investigate a multi-type extension of the continuous-time Ehrenfest model, and its diffusive 
approximation based on the Ornstein-Uhlenbeck process. 
The state space of the extended model is a finite lattice, say $S$, formed by the integers $0,1,2,\ldots, N$ of $d$ 
lines joined at the origin, thus constituting an extended star graph. The evolution of the stochastic process over each line of 
$\cal S$ evolves as a classical Ehrenfest model, i.e.\ as a continuous-time skip-free Markov chain (as a birth-death process)
with linear decreasing upward transition rate $\lambda (N-k)$ and increasing downward transition rate $\mu (N+k)$ at $k$. 
The state $k=N$ is reflecting since the upward transition rate vanishes therein. Moreover, the transitions from the 
state 0 to each of the $d$ lines are governed by rates depending on the elements of a stochastic matrix. The transitions 
of the process from a line to another one correspond to the type changes of the considered multi-type Ehrenfest model. 
The case $d=2$ corresponds to the one-dimensional Ehrenfest model. Our analysis, based on the 
probability-generating-function approach, allows to determine the explicit expression of the transient probabilities 
(cumulative on the rays) when $\lambda =\mu$, and the asymptotic probabilities for any choice of parameters 
$\lambda$ and $\mu$.  In particular, when $\lambda =\mu$, the asymptotic distribution is strictly related to the 
binomial distribution with parameters $(2N, \frac{1}{2})$. 
\par
A similar process describing the dynamics of a multi-type birth-death-immigration process has been analyzed in 
Di Crescenzo et al.\ \cite{DiCrescenzo2016}, where the transitions on the states of a star graph with various semiaxis 
are regulated by linear increasing transition rates. This process was also studied under certain limit conditions that lead to 
a diffusion process on the star graph with linear drift and infinitesimal variance on each ray. In the realm of mathematical 
biology, other investigations devoted to birth-death processes on graphs are due to Allen et al.\ \cite{Allen2020}, 
Kaveh et al.\ \cite{Kaveh2015}, and Sui et al.\ \cite{Sui2015}, for instance. Furthermore, the analysis of birth-death 
processes on networks and lattice structures viewed as graphs have been performed to model evolutionary systems 
also by means of the mean-field methods (cf.\ Granovsky and Zeifman \cite{Granovsky1997}, and Peliti \cite{Peliti1985}). 
\par
The difficulties related to the analysis of discrete evolution models on star graph stimulated several authors to consider 
alternative models consisting in diffusion processes on the state space formed by semiaxis joined at the origin 
(also known as spider). In this framework, we recall the contribution by Benichou and Desbois \cite{Benichou}, 
where a Brownian particle diffusing along the links of a general graph is considered and relevant quantities  are computed 
for different kinds of graphs, such as for star graphs. Other investigations concerning the dynamics of the Brownian 
motion on the spider are due to Cs\'aki et al.\ \cite{Csaki2016} and Kostrykin et al.\ \cite{Kostrykin2012}, 
also with care to the possible boundary conditions at the vertex in view of important applications. 
Furthermore, see Dassios and Zhang \cite{Dassios2020} for the analysis of the reflected Brownian motion with drift on 
a finite collection of rays, in view of possible applications in risk theory finalized to price the Parisian type options. 
In addition, Papanicolaou et al.\ \cite{PaPaLe12} also pointed out that diffusion processes of this kind can 
be applied to spread of toxic particles in a system of channels or vessels, or to propagation of information in networks. 
In this framework, we recall that one of the first contributions on diffusion processes on graphs was given by 
Freidlin and Wentzell \cite{FrWe93}. Occupation time functionals for birth-death processes and diffusion processes 
on graphs were studied by Weber \cite{We01}.  
%P. Exner, J. P. Keating, P. Kuchment, T. Sunada, and A. Teplyaev (eds.), Analysis on Graphs and Its Applications, Proc. Symp. Pure Math., vol. 77, Providence, American Math. Soc., 2010.)
%For instance, in Dharmaraja \cite{Dharmaraja2015} {\em et al.} a Markov chain with catastrophes on a star graph with two semiaxes explains the second law of thermodynamics on the macroscopic scale. 
\par
Along the line of the above mentioned researches, after investigating the transient and the asymptotic behavior 
we employ a scaling procedure that leads to a diffusive approximation of the considered model. The resulting 
process is an Ornstein-Uhlenbeck diffusion on the spider with special reflecting-type conditions on the vertex 
at the origin. In the papers by Cs\'aki et al.\ \cite{Csaki2016}, and Dassios and Zhang \cite{Dassios2020},
the switching of the Brownian motion between the semiaxis is regulated by independent general distributions, 
whereas in the contribution by  Papanicolaou et al.\ \cite{PaPaLe12} it follows a uniform distribution over the rays. 
In the present paper we provide the explicit expression of the stationary probabilities for the diffusive approximation in 
the cases such that when the diffusive particle reaches the vertex then the choice of the next line occurs 
\\
1. uniformly to any of the $d$ lines,
\\
2.  uniformly to any of the $d-1$ lines different from the originating one,
\\
3.  toward the next line, from $l$ to $l+1$, and from $d$ to $1$, thus visiting cyclically any line,
\\
4. toward the next line, from $l$ to $l+1$, until it reaches the last line, i.e.\ line $d$,
\\
5. toward one of the adjacent lines, according to a random walk scheme.
\par
It is worth mentioning that the  Ornstein-Uhlenbeck process, obtained here through a diffusive approximation, 
has been largely investigated for its important applications in several fields, in particular in the context of neuronal activity. 
Ricciardi and Sacerdote \cite{Ri79} provided one of the first contributions in this area, by studying mean and variance of the 
first-passage time through a constant boundary. We  recall also Lansky et al.\ \cite{Lansky2007} for the analysis of an optimum 
signal in the related neuronal model, and Buonocore et al.\ \cite{BuCaNoPi2015} for applications in neuronal models with 
periodic input signals through an Ornstein-Uhlenbeck process in the presence of a reflecting boundary. 
The membrane potential is also modeled by a non homogeneous Ornstein-Uhlenbeck process with jumps in 
Giorno and Spina \cite{Giorno2014}, where the effect of random refractoriness is also considered.
See also Giorno et al.\ \cite{Giorno2012} for some quantitative informations on the reflected Ornstein-Uhlenbeck 
process subject to catastrophes, originating from a heavy-traffic approximation to a queueing system. 
\par
{\em Plan of the paper}: In Section \ref{Section:model} we provide a thorough description of the stochastic model and 
the differential-difference equations for the transient probabilities. 
We also describe some possible fields of application of the considered model. 
Section \ref{Section:GenFunct} is devoted to the analysis of the stochastic process, 
with special attention to the determination of the probability generating functions, 
which allow to obtain a closed-form expression of the probabilities in the special case when $\lambda=\mu$. 
Comparisons between exact probabilities and their estimates based on simulation are also provided.  
Various asymptotic results are then investigated in Section \ref{section:Asymptotic} as time tends to infinity, 
including the asymptotic probability generating function and the corresponding stationary probabilities, with mean, 
variance and coefficient of variation. We also investigate the (Shannon) entropy of the system in the stationary 
phase, and show its maximum over the ratio $\lambda/\mu$ of rates, which depends on the number $N$. 
Section \ref{section:approximation} is concerning the diffusion approximation that leads to a diffusion 
process on the spider through a suitable scaling procedure. We determine the partial differential equation 
for the transient probability density of the process, with the reflecting/switching condition at the vertex of the spider. 
The equations of the cumulative density on the rays of the spider correspond to those of the Ornstein-Uhlenbeck 
process in the presence of a reflecting boundary at 0. Thus, we obtain the joint asymptotic probability distribution 
of the process, which is formed by two independent laws: (i) the density of the location on the ray of the spider, 
which has a truncated Gaussian form, and (ii) the distribution of the occupied ray, which depends strictly 
on the probabilities that govern the switching mechanism between the rays.  Some possible choices of the 
switching probabilities are studied in order to come to a complete description of the  asymptotic distribution 
of the diffusion process. Some comparisons between the distributions of the discrete model and the 
diffusive approximating process are provided to illustrate the goodness of the approximation. 
Finally, concluding remarks on possible future developments are given in Section 6. 
\par
Throughout the paper, $\mathbb N$ denotes the set of positive integers, and  $\mathbb N_0=\{0\}\cup\mathbb N$. 
% ==============================================================
\section{The multi-type Ehrenfest model}\label{Section:model}
% ==============================================================
%Let $(\Omega, {\cal F}, {\mathbb P}, \{{\cal F}_t\}_{t\geq 0})$ be a a filtered probability space 
%satisfying the usual conditions. 
%
We consider a system that can accommodate at most $N$ particles, with $N\in \mathbb N$, and such that $d$ 
types of particles are allowed, for $d\in \mathbb N$. The set of possible types is denoted by $D:=\{1,2,\ldots, d\}$. 
Moreover, the particles accommodated simultaneously in the system must be of the same type. 
The particle dynamics is regulated by the following assumptions, where $h>0$ is sufficiently small: 
\begin{description}
\item{{\em (a)} \ } 
If the system at time $t$ containes $k=1,2,\ldots,N$ particles of type $j\in D$, then 
during the time interval $(t,t+h]$ either one particle leaves the system with probability $\mu (N+k) h+o(h)$, 
or a new particle of the same type joins the system with probability $\lambda (N-k) h+o(h)$, 
or the particle number is unchanged with probability $1-[\mu (N+k)+\lambda (N-k)] h+o(h)$. 
\item{{\em (b)} \ } 
If the system is empty at time $t$, then during the time interval $(t,t+h]$ either the system is occupied by 
a particle of type $j \in D$, with probability $c_{l,j}\lambda N h+o(h)$, assuming that the last particle in the 
system was of type $l \in D$, or the system remains empty with probability $1-\lambda N h+o(h)$. 
\end{description}
From the above assumptions we have that $\lambda$ and $\mu$ are positive parameters that regulate the joining 
and leaving intensities of the particles, respectively. Moreover, assumption {\em (a)} implies that the arrivals 
of new particles are inhibited if the system contains $N$ particles. The coefficients $c_{l,j}$ actually 
form the discrete probability distribution that regulates the switching mechanism for the particle types, 
that acts when the system empties. We have 
\begin{equation}
c_{l,j}\geq 0,\qquad \sum_{j\in D}  c_{l,j}=1,\qquad \forall \,l,j \in D,
\label{c_lj}
\end{equation}
so that $C:=\{c_{l,j}\}_{l,j \in D}$ is a stochastic matrix. 
\par
Let us now introduce the continuous-time Markov chain $\{({\cal N}(t),{\cal L}(t)), t\geq 0\}$ that describes the system dynamics, 
such that, at time $t$, ${\cal N}(t)=k$ gives the number of particles in the system, and ${\cal L}(t)=j$ gives the type of such particles. 
The state space of the process is the set $S_0=\{(0,0)\}\cup (\textbf{N} \times D)$, with $\textbf{N}:=\{1,2,\ldots,N\}$, 
consisting of the integers of $d$ segments $S_1,S_2,\ldots, S_d$ $(d\in\mathbb{N})$ with a common extreme $(0,0)$ 
(see Figure \ref{FigSystem}). We denote $S=S_0\setminus\{(0,0)\}$ and, for simplicity, we write $0$ instead of $(0,0)$. 
%
% ============================= FIGURE =============================
\begin{figure}[t]
\begin{center}
\begin{picture}(270,230)
\put(50,50){\line(1,1){150}}
\put(50,50){\line(2,1){150}}
\put(50,50){\line(6,1){150}}
\put(50,50){\line(2,-1){150}}
\put(50,50){\circle*{3}}
\put(100,100){\circle*{3}}
\put(100,75){\circle*{3}}
\put(100,58.5){\circle*{3}}
\put(100,25){\circle*{3}}
\put(150,150){\circle*{3}}
\put(150,100){\circle*{3}}
\put(150,66.5){\circle*{3}}
\put(150,0){\circle*{3}}
\put(200,200){\circle*{3}}
\put(200,125){\circle*{3}}
\put(200,75){\circle*{3}}
\put(200,-25){\circle*{3}}
\put(40,51){\makebox(20,15)[t]{\small $(0,0)$}}
\put(91,100){\makebox(20,15)[t]{\small $(1,1)$}}
\put(91,72){\makebox(20,15)[t]{\small $(1,2)$}}
\put(91,53){\makebox(20,15)[t]{\small $(1,3)$}}
\put(91,20){\makebox(20,15)[t]{\small $(1,d)$}}
\put(141,150){\makebox(20,15)[t]{\small $(2,1)$}}
\put(141,97){\makebox(20,15)[t]{\small $(2,2)$}}
\put(141,63){\makebox(20,15)[t]{\small $(2,3)$}}
\put(141,-3){\makebox(20,15)[t]{\small $(2,d)$}}
\put(191,196){\makebox(20,15)[t]{\small $(N,1)$}}
\put(191,120){\makebox(20,15)[t]{\small $(N,2)$}}
\put(191,70){\makebox(20,15)[t]{\small $(N,3)$}}
\put(191,-28){\makebox(20,15)[t]{\small $(N,d)$}}
\put(210,189){\makebox(20,15)[t]{  $S_1$}}
\put(210,114){\makebox(20,15)[t]{  $S_2$}}
\put(210,63){\makebox(20,15)[t]{  $S_3$}}
\put(210,-38){\makebox(20,15)[t]{  $S_d$}}
\put(165,175){\makebox(20,15)[t]{\small $\ldots$}}
\put(165,110){\makebox(20,15)[t]{\small $\ldots$}}
\put(165,68){\makebox(20,15)[t]{\small $\ldots$}}
\put(165,-16){\makebox(20,15)[t]{\small $\ldots$}}

\put(100,40){\line(0,1){2}}
\put(100,45){\line(0,1){2}}
\put(100,50){\line(0,1){2}}
\put(150,26){\line(0,1){3}}
\put(150,38){\line(0,1){3}}
\put(150,50){\line(0,1){3}}
\put(200,5){\line(0,1){4}}
\put(200,27.5){\line(0,1){4}}
\put(200,50){\line(0,1){4}}
\end{picture}
\vspace{1cm}
\end{center}
\caption{Schematic representation of the state space $S_0$.}
\label{FigSystem}
\end{figure}
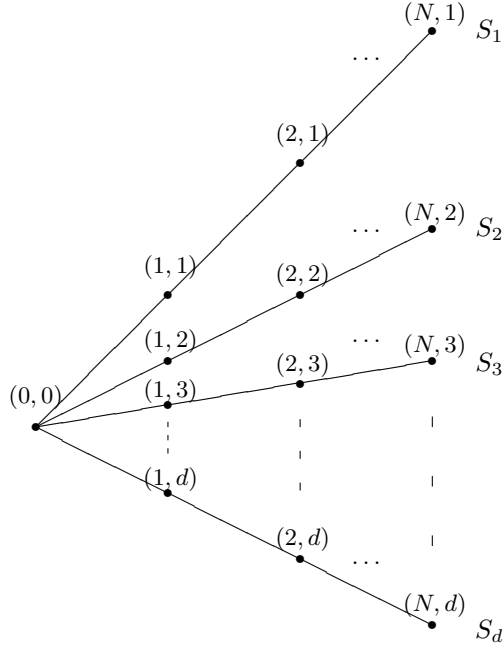
\par
Formally, the system dynamics is regulated by the transition rates 
$$
 q({\bf \alpha};{\bf \beta})=\lim_{h\rightarrow 0^+} \frac{1}{h} \,
 {\mathbb P}\left[({\cal N}(t+h),{\cal L}(t+h))={\bf \beta}\,\vert\,({\cal N}(t),{\cal L}(t))={\bf \alpha}\right],
 \qquad {\bf \alpha}\in S,\; {\bf \beta}\in S_0,
$$
$$
q(0;1,j;l)=\lim_{h\rightarrow 0^+} \frac{1}{h} \,
 {\mathbb P}\left[({\cal N}(t+h),{\cal L}(t+h))=(1,j)\,\vert\,({\cal N}(t),{\cal L}(t))=0,{\cal J}(t)=l\right],
 \qquad j,\,l \in D,
$$
where ${\cal J}(t)$ is the last state visited by the Markov chain before arriving in $0$. 
According to the assumptions {\em (a)} and {\em (b)}, the following relations hold, for $l,j\in D$, 
\begin{equation}
\begin{array}{l}
 q(k,j;k-1,j)=\mu (N+k),\qquad  k\in {\bf N},     \\
 q(k,j;k+1,j)=\lambda(N-k),\qquad  k\in {\bf N},      \\
 q(0;1,j;l)=c_{l,j} \lambda N,
\end{array} 
\label{eq:tassi}
\end{equation}
where $\lambda,\mu>0$, and $c_{l,j}$ satisfy the conditions (\ref{c_lj}). 
Moreover, for $i,j,k,r\in {\bf N}$ one has 
$$
\begin{array}{l}
 q(k,j;r,j)=0\qquad \hbox{if $|k-r|>1$}, \\ 
 q(k,i;r,j)=0\qquad \hbox{if $i\neq j$}, \\
 q(0,0;r,j)=q(r,j;0,0)=0\qquad \hbox{if $r\neq 1$}.
\end{array} 
$$
Note that $c_{l,j}\lambda N $ represents the intensity of the arrival of a new particle of type $j$, 
given that the system is empty and the last previous particle in the system was of type $l$. 
We remark that the considered Markov chain is a skip-free process and that $(0,0)$ is a non-absorbing state. 
Moreover, the given process is bounded, 
and hence uniquely determined by the transition rates (cf.\ Chen et al.\ \cite{ChPoZhCa2005}). 
\par
Let us now introduce the  transition probabilities  of the process $\{({\cal N}(t),{\cal L}(t)), t\geq 0\}$. Assuming 
that the initial condition is given by $({\cal N}(0),{\cal L}(0))=0,{\cal J}(0)=l_0$, with $l_0 \in D$, we consider  
\begin{equation}
\begin{split}
  & p(0,l, \cdot ):={\mathbb P}\{({\cal N}(\cdot),{\cal L}(\cdot))=0,{\cal J}(\cdot)=l\,|\, ({\cal N}(0),{\cal L}(0))=0,{\cal J}(0)=l_0\}, 
  \qquad l\in D, 
  \\
  & 
  p(k,j, \cdot ):={\mathbb P}\{({\cal N}(\cdot),{\cal L}(\cdot))=(k, j)\,|\, ({\cal N}(0),{\cal L}(0))=0,{\cal J}(0)=l_0\},
  \qquad  k\in {\bf N}, \;\; j\in D,
\end{split}
\label{eq:probpkjt}
\end{equation}
with initial conditions expressed as 
\begin{equation}
p(0,l,0)=\delta_{l,l_0},
 \label{probiniz1}
\end{equation}
where $\delta_{l,l_0}$ is the Kronecker delta, and
\begin{equation}
p(k,j,0)=0,\qquad k,j\in {\bf N}.
 \label{probiniz2}
\end{equation}
We are now able to provide the Kolmogorov forward equations governing the transition probabilities (\ref{eq:probpkjt}). 
Recalling the rates (\ref{eq:tassi}), the following system of differential-difference equations holds, for  $j\in D$,  and $t>0$:
\begin{eqnarray}
\label{eq:system}
&&  \hspace{-0.8cm}
{d \over d t}\;p(0,j, t)=\mu (N+1)\,p(1,j,t)-\lambda N\,p(0,j,t),\nonumber
\\
&&  \hspace{-0.8cm}
{d \over d t}\;p(1,j, t)=\mu (N+2)\,p(2,j,t)+\sum_{l \in D}c_{l,j} \lambda N p(0,l,t)-[\lambda (N-1)+\mu (N+1)]\,p(1,l,t),\nonumber
\\
&&  \hspace{-0.8cm}
{d \over d t}\;p(k,j, t)=\mu (N+k+1)\,p(k+1,j,t)+\lambda (N-k+1) p(k-1,j,t)
\\
&& \hspace{1.5cm}-[\lambda (N-k)+\mu (N+k)]\,p(k,j,t),
\hspace{3cm}  k \in {\bf N}\setminus\left\{1,N\right\}\nonumber
\\
&&  \hspace{-0.8cm}
{d \over d t}\;p(N,j, t)=\lambda\,p(N-1,j, t)-\mu2N\,p(N,j, t).
\nonumber
\end{eqnarray}
Moreover, we can express the marginal probabilities for the number of particles in the system in terms of probabilities (\ref{eq:probpkjt}) 
as follows:
\begin{equation}
 p(0,\cdot):={\mathbb P}\{({\cal N}(\cdot),{\cal L}(\cdot))=0\,|\, ({\cal N}(0),{\cal L}(0))=0,{\cal J}(0)=l_0\}  
 = \sum_{l\in D} p(0,l,\cdot)
 \label{p0t}
\end{equation}
and
\begin{equation}
p(k,\cdot):= {\mathbb P}\{ {\cal N}(\cdot) =k\,|\, ({\cal N}(0),{\cal L}(0))=0,{\cal J}(0)=l_0\}
= \sum_{j\in D} p(k,j,\cdot),\qquad k\in {\bf N}. 
 \label{pkt}
\end{equation}
Taking into account the conditions (\ref{c_lj}), from the system (\ref{eq:system}) it follows that the probabilities (\ref{p0t}) and (\ref{pkt}) 
satisfy the following Kolmogorov forward equations, for  $t>0$:
\begin{eqnarray}
&&  \hspace{-0.8cm}
{d \over d t}\;p(0, t)=\mu (N+1)\,p(1,t)-\lambda N\,p(0,t),\nonumber
\\
&&  \hspace{-0.8cm}
{d \over d t}\;p(1,t)=\mu (N+2)\,p(2,t)+\lambda N p(0,t)-[\lambda (N-1)+\mu (N+1)]\,p(1,t),\nonumber
\\
&&  \hspace{-0.8cm}
{d \over d t}\;p(k, t)=\mu (N+k+1)\,p(k+1,t)+\lambda (N-k+1) p(k-1,t)\nonumber
\\
&& \hspace{1.5cm}-[\lambda (N-k)+\mu (N+k)]\,p(k,t),
\qquad\qquad k \in {\bf N}\setminus\left\{1,N\right\}\nonumber
\\
&&  \hspace{-0.8cm}
{d \over d t}\;p(N, t)=\lambda\,p(N-1,t)-\mu 2N\,p(N, t).
\nonumber
\end{eqnarray}
Due to (\ref{probiniz1}) and (\ref{probiniz2}), the related initial conditions are given by 
\begin{equation}
 p(0,0)=1, \qquad p(k,0)=0,\quad k\in {\bf N}.
 \label{pkt0}
\end{equation}
\par 
We pinpoint that the present model deserves interest in several contexts. 
For instance, the review of Crawford and Suchard \cite{CrSu2012} points out how 
various kinds of birth-death processes can be applied in ecology, genetics, and evolution. 
Moreover, the paper by Giorno et al.\ \cite{GiornoNN1985} shows that a process 
with linear decreasing birth rate and linear increasing death rate can be used to describe 
the number of customers in a finite-capacity queue. In this setting, the process 
with rates (\ref{eq:tassi}) can also be viewed as a model for the evolution of a multi-type queueing system,
where the following rules hold:
\\
- new customers are discouraged from joining long queues, 
\\
- the system can accomodate at most $N$ customers, 
\\
- the server adapts the service rate to the number of customers, 
\\ 
- there are $d$ types of customers, 
\\
- all customers in the system belong to the same type,
\\
- the jockeying mechanism governed by the stochastic matrix $C$ allows to switch possibly from a type to another type 
of customers when the system is empty. 
% ==============================================================
\section{Analysis of the model}\label{Section:GenFunct}
% ==============================================================
In this section we use the generating function-based approach to investigate the transient dynamics of the considered system. 
To this aim, let us consider the probability generating function 
\begin{equation}
F(z,t):=\mathbb E\left[z^{{\cal N}(t)}\,|\, ({\cal N}(0),{\cal L}(0))=0,{\cal J}(0)=l_0\right]
=p(0,t)+\sum_{k\in {\bf N}} z^k p(k,t),
\qquad z\in [0,1],\quad t\geq 0,
\label{FPgrande}
\end{equation}
where the state probabilities $p(0,t)$ and $p(k,t)$ have been defined in Eqs. (\ref{p0t}) and  (\ref{pkt}), respectively. 
By virtue of (\ref{pkt0}),  the following initial condition holds:
\begin{equation}
F(z,0)=1,\qquad z\in [0,1].
\label{initialconditions}
\end{equation}
Moreover, from (\ref{FPgrande}) one has the boundary conditions   
\begin{equation}
F(1,t)=1,\qquad F(0,t)=p(0,t),\qquad t\geq 0.
\label{boundconditions}
\end{equation}
\begin{proposition}
\label{prop_F}
The generating function (\ref{FPgrande}) satisfies the following partial differential equation
for $z\in [0,1]$ and $t\geq 0$:
\begin{equation}
{\partial \over \partial t}\!F(z, t)
=(1-z)\left[ 
-\mu \frac{N}{z} p(0,t)+\frac{N}{z} (\mu-\lambda z) F(z,t)
+(\mu+\lambda z) {\partial \over \partial z}\!F(z, t)\right].
\label{eq:diffF}
\end{equation}
\end{proposition}
\begin{proof}
Recalling Eqs.\ (\ref{p0t}) and (\ref{pkt}), the probability generating function (\ref{FPgrande}) can be expressed 
in terms of (\ref{eq:probpkjt}) as
\begin{equation}
F(z,t)=p(0,t)+\sum_{k\in {\bf N}} z^k \sum_{j\in D} p(k,j,t)=\sum_{j\in D}\left[ p(0,j,t)+ G_j(z,t) \right],
\label{eq:defF2}
\end{equation}
where we have set
$$
 G_j(z,t):=\sum_{k\in {\bf N}} z^k p(k,j,t),
 \qquad z\in [0,1],\quad t\geq 0.
$$
From the system (\ref{eq:system}), for every $j\in D$, the probability generating function $G_j(z,t)$ 
satisfies the following differential equation:
\begin{eqnarray*}
&& \hspace*{-1cm}
{\partial \over \partial t} G_j(z, t)
=(1-z)(\mu+\lambda z){\partial \over \partial z}G_j(z, t)+\frac{N}{z} (1-z) (\mu-\lambda z) G_j(z, t)
\\
&& \hspace*{0.6cm}
-(N+1) \mu p(1,j,t)+ z \sum_{l \in D} c_{l,j} \,\lambda N p(0,l,t).
\end{eqnarray*}
Hence, the equation (\ref{eq:diffF}) follows making use of (\ref{eq:system}), (\ref{eq:defF2}) and condition (\ref{c_lj}).
\end{proof}
Hereafter, the result given in Proposition \ref{prop_F} is used to obtain an integral form of $F(z,t)$.
%%%%%%%%%%%%%%%%%%%%%%%%%%
% per qualunque lambda e mu
%%%%%%%%%%%%%%%%%
\begin{proposition}\label{prop:Fzt}
For all $\lambda,\,\mu>0$, Eq.\ (\ref{eq:diffF}), with conditions (\ref{initialconditions})  and  (\ref{boundconditions}),
admits of the following solution for $z\in [0,1]$ and $t\geq 0$:
\begin{eqnarray}\label{solFgenerale_lambda_neq_mu}
F(z,t) \!\!\!\! &=& \!\!\!\!  \left[\frac{\left(\mu(z-1)+(z \lambda+\mu)\,e^{t(\lambda+\mu)}\right)
 \left(\lambda(1-z)+(z \lambda+\mu)\,e^{t(\lambda+\mu)}\right)}{(\lambda+\mu)^2\, e^{2t(\lambda+\mu)}\,z }\right]^N 
 -\frac{\mu N(1-z)}{z^N \,(\lambda+\mu)^{2N-1}} 
\nonumber\\ 
&& \!\!\!\!  \times  \int_0^t  p(0,y)\,e^{-2N (t-y)(\lambda+\mu)} %\nonumber\\
%&&\times 
\left[(z \lambda+\mu)\,e^{(t-y)(\lambda+\mu)}-\lambda(z-1)\right]^N 
\nonumber\\ 
&& \!\!\!\!  \times 
\left[(z \lambda+\mu)\,e^{(t-y)(\lambda+\mu)}+\mu(z-1)\right]^{N-1} dy.
\end{eqnarray}
%
%and $\alpha_k$ are the roots of the polynomial defined in (\ref{pol_P}).
%\label{propFzt}
\end{proposition}
\begin{proof}
By adopting the method of characteristics, Eq.\ (\ref{eq:diffF}) gives the following characteristic equations for the original system
\begin{equation}
 {\displaystyle{d z \over d s}=(\mu+\lambda z)(z-1)}, \qquad 
 {\displaystyle{d t \over d s}=1}, \qquad 
 {\displaystyle{d F \over d s}=-\frac{\mu N(1-z)}{z}p(0,t)-\frac{N}{z}(1-z)(\lambda z-\mu) F(z,t)}.
\label{eq:characteristic_lambda_mu}
\end{equation}
From Eqs.\ (\ref{eq:characteristic_lambda_mu}), along the characteristic curves
\begin{equation}
z=\frac{\mu +\lambda \tau+\mu(\tau-1)e^{s(\lambda+\mu)}}{\mu +\lambda \tau-\lambda(\tau-1)e^{s(\lambda+\mu)}},\qquad t=s,\qquad \tau\in {\mathbb R},
\label{characteristic_lambda_mu}
\end{equation}
the partial differential equation (\ref{eq:diffF}) yields
$$
\displaystyle \frac{{\rm d} F}{{\rm d}s}
+N(\lambda+\mu)\left[2-\frac{\mu+\lambda \tau}{\mu+\lambda \tau+\lambda (1-\tau)e^{s(\lambda+\mu)}}-\frac{\mu+\lambda \tau}{\mu+\lambda \tau-\mu (1-\tau)e^{s(\lambda+\mu)}}\right]F -\frac{N \mu(\lambda+\mu)(\tau-1)e^{s(\lambda+\mu)}}{\mu+\lambda \tau+\mu(\tau-1)e^{s(\lambda+\mu)}}\,p(0,s)=0.
$$
By solving this linear first order differential equation, by taking into account conditions (\ref{boundconditions}), we obtain
\begin{eqnarray}
F(s)&=&\frac{(\mu+\lambda)^{2N} \,\tau^N}{\left[\mu+\lambda \tau-\lambda (\tau-1)e^{y(\lambda+\mu)}\right]^N\left[\mu+\lambda \tau+\mu (\tau-1)e^{y(\lambda+\mu)}\right]^{N}}
\nonumber \\
&+&\frac{N \,\mu\,(\lambda+\mu)\,(\tau-1)}{\left[\mu+\lambda \tau-\lambda (\tau-1)e^{y(\lambda+\mu)}\right]^N\left[\mu+\lambda \tau+\mu (\tau-1)e^{y(\lambda+\mu)}\right]^{N}}
\nonumber \\
&\times&\int_0^s p(0,y)e^{y(\lambda+\mu)}\left[\mu+\lambda \tau-\lambda (\tau-1)e^{y(\lambda+\mu)}\right]^N\left[\mu+\lambda \tau+\mu (\tau-1)e^{y(\lambda+\mu)}\right]^{N-1}dy.
\label{eq:Fsproof}
\end{eqnarray}
From (\ref{characteristic_lambda_mu}) one has:
$$
\tau=\frac{\mu(z-1)+(z \lambda+\mu)e^{t(\lambda+\mu)}}{\lambda(1-z)+(z \lambda+\mu)e^{t(\lambda+\mu)}},\qquad s=t;
$$
so, by substituting in (\ref{eq:Fsproof}), 
after some calculations and due to (\ref{initialconditions}) 
we obtain the solution (\ref{solFgenerale_lambda_neq_mu}).
\end{proof}
The integral form of $F(z,t)$ obtained in Proposition \ref{prop:Fzt} is expressed in terms of $p(0,t)$.  
Hence, determining the latter function is a relevant problem. In the following proposition 
we obtain its Laplace transform 
$$
 H(\eta):={\cal L}_\eta[p(0,t)]= 
 \int_0^{\infty} e^{- \eta t} p(0,t)\,dt, \qquad \eta \geq 0  
$$
in terms of the Gauss hypergeometric function
\begin{equation}
 {}_{2}F_{1}(a,b;c;z)=\sum_{n=0}^{+\infty} \frac{(a)_n(b)_n}{(c)_n}\,\frac{z^n}{n!}. 
 \label{hypergeometricf}
\end{equation}
%
%%%%%%%%%%%%%%%%%%%%%
%Risultato su p(0,t)
%%%%%%%%%%%%%%%%%
\begin{proposition}\label{prop:Heta}
For all $\lambda,\,\mu>0$, the Laplace transform of $p(0,t)$ is given by 
\begin{eqnarray}\label{p_Ltransf}
H(\eta)=&&
\left[\mu\,{}_{2}F_{1}\left(-N,\frac{\eta}{\lambda+\mu};1+N+\frac{\eta}{\lambda+\mu};-\frac{\lambda}{\mu}\right)\right]\nonumber\\
&&\times \left[\eta\,\mu \,{}_{2}F_{1}\left(1-N,1+\frac{\eta}{\lambda+\mu};1+N+\frac{\eta}{\lambda+\mu};-\frac{\lambda}{\mu}\right)\right.\nonumber\\
&&\left.+\,\frac{\eta\,\lambda(\eta+\lambda+\mu)}{(\lambda+\mu)(N+1)+\eta}\,{}_{2}F_{1}\left(1-N,2+\frac{\eta}{\lambda+\mu};2+N+\frac{\eta}{\lambda+\mu};-\frac{\lambda}{\mu}\right)\right]^{-1}, 
\qquad \eta \geq 0.
\end{eqnarray}
Moreover, if $\lambda=\mu$ then 
\begin{equation}\label{p_Ltransf_lambda_eq_mu}
H(\eta)=\frac{2}{\eta}\, \frac{\Gamma\left(1+\frac{\eta}{4 \mu}\right)\,\Gamma\left(N+\frac{1}{2}+\frac{\eta}{4 \mu}\right)}{\Gamma\left(1+\frac{\eta}{4 \mu}\right)\,\Gamma\left(N+\frac{1}{2}+\frac{\eta}{4 \mu}\right)+\Gamma\left(N+1+\frac{\eta}{4 \mu}\right)\,\Gamma\left(\frac{1}{2}+\frac{\eta}{4 \mu}\right)}, 
\qquad \eta \geq 0.
\end{equation}
\end{proposition}
\begin{proof}
By requiring that $\lim_{z\to 0^+} z^N F(z,t)=0$, from (\ref{solFgenerale_lambda_neq_mu}) we obtain, for all $\lambda,\,\mu>0$, 
\begin{eqnarray*}
&&\hspace{-4cm} 
 \left[\frac{\mu \left(e^{t(\lambda+\mu)}-1\right)\left(\mu\,e^{t(\lambda+\mu)}+\lambda\right)}{(\lambda+\mu)^2\, e^{2t(\lambda+\mu)} }\right]^N 
 -\frac{\mu^N N}{(\lambda+\mu)^{2N-1}} \int_0^t  p(0,y)e^{-2N (t-y)(\lambda+\mu)} 
\nonumber \\
&& \times
 \left[\mu\,e^{(t-y)(\lambda+\mu)}+\lambda\right]^N\left[e^{(t-y)(\lambda+\mu)}-1\right]^{N-1} dy=0,
\end{eqnarray*}
so that 
$$
 \left[\left(1-e^{-t(\lambda+\mu)}\right)\left(\mu+\lambda\,e^{-t(\lambda+\mu)}\right)\right]^N
 =N(\lambda+\mu)\, \int_0^t  p(0,y)\frac{\left[\left(1-e^{-(t-y)(\lambda+\mu)}\right)\left(\mu+\lambda\,e^{-(t-y)(\lambda+\mu)}\right)\right]^{N}}{e^{(t-y)(\lambda+\mu)}-1} \,dy.
$$
Applying the Laplace transform ${\cal L}_\eta$ on both sides, one has
\begin{eqnarray}
&&\hspace{-2cm} 
\frac{\mu^N \,N\,}{\eta}\frac{\Gamma(N)\,\Gamma\left(1+\frac{\eta}{\lambda+\mu}\right)}{\,\Gamma\left(N+1+\frac{\eta}{\lambda+\mu}\right)}{}_{2}F_{1}\left(-N,\frac{\eta}{\lambda+\mu};1+N+\frac{\eta}{\lambda+\mu};-\frac{\lambda}{\mu}\right)
\nonumber\\
&&=N\,(\lambda+\mu)\,H(\eta)\,{\cal L}_\eta\left[e^{-t(\lambda+\mu)}\left(\mu+\lambda e^{-t(\lambda+\mu)}\right)\left[\left(1-e^{-(t-y)(\lambda+\mu)}\right)\left(\mu+\lambda\,e^{-(t-y)(\lambda+\mu)}\right)\right]^{N-1}\right]
\nonumber \\
&&=N\,(\lambda+\mu)\,H(\eta)\,  
\Big\{
\mu \,{\cal L}_{\eta+\lambda+\mu}\left[\left[\left(1-e^{-(t-y)(\lambda+\mu)}\right)\left(\mu+\lambda\,e^{-(t-y)(\lambda+\mu)}\right)\right]^{N-1}\right]\nonumber\\
&&+\lambda\,{\cal L}_{\eta+2(\lambda+\mu)}\left[\left[\left(1-e^{-(t-y)(\lambda+\mu)}\right)\left(\mu+\lambda\,e^{-(t-y)(\lambda+\mu)}\right)\right]^{N-1}\right]\Big\},
\label{lapl_tranf_rel}
\end{eqnarray}
where $H(\eta)$ denotes the Laplace transform of $p(0,t)$, and 
(cf.\  Eq.\ (28) of Prudnikov et al.\ \cite{PrudnikovVol4})     
\begin{eqnarray}
&&\hspace{-2cm} 
{\cal L}_\rho\left[\left[\left(1-e^{-(t-y)(\lambda+\mu)}\right)\left(\mu+\lambda\,e^{-(t-y)(\lambda+\mu)}\right)\right]^{N-1}\right]\nonumber\\
&&=\frac{\mu^{N-1}}{\rho}\Gamma(N)\frac{\Gamma\left(1+\frac{\rho}{\lambda+\mu}\right)}{\Gamma\left(N+\frac{\rho}{\lambda+\mu}\right)}{}_{2}F_{1}\left(1-N,\frac{\rho}{\lambda+\mu};N+\frac{\rho}{\lambda+\mu};-\frac{\lambda}{\mu}\right).
\label{part_lapl_transf}
\end{eqnarray}
Hence, from (\ref{lapl_tranf_rel}) and (\ref{part_lapl_transf}) we obtain the expression given in (\ref{p_Ltransf}) 
for $\lambda\neq \mu$.
Moreover, if $\lambda=\mu$, then making use of (see Eq.\ (15.1.21) of Abramowitz and Stegun\cite{Abram1994})
$$
 {}_{2}F_{1}\left(a,b;a-b+1;-1\right)
 =\frac{2^{-a}\sqrt{\pi}\,\Gamma\left(a-b+1\right)}{\Gamma\left(\frac{a+1}{2}\right)\,\Gamma\left(\frac{a}{2}-b+1\right)}
$$
the expression given in (\ref{p_Ltransf_lambda_eq_mu}) thus follows from (\ref{p_Ltransf}).
\end{proof}
%
%per lambda uguale a mu si puÃ² andare avanti e calcolare le probabilitÃ 
Aiming to obtain the inverse Laplace transform of $H(\eta)$,  we first provide the following lemma, whose proof is given in Appendix A. 
%%%%lemma
%
\begin{lemma}
The $(N+1)$-degree polynomial  
\begin{equation}
\label{pol_P}
P(x)=x \left[\prod_{r=0}^{N-1} \left(x+2 \mu (2r+1)\right)+\prod_{r=0}^{N-1} \left(x+2 \mu (2r+2)\right)\right]
\end{equation}
has one root equal to $0$ and $N$ distinct negative roots.
\label{lemma_pol}
\end{lemma}
%
%Some instances of $P(x)$ are listed hereafter:
%\\
%$N=1$: $P(x)=2x(x+3\mu)$,
%\\
%$N=2$: $P(x)=2x(x^2+10 x \mu+22\mu^2)$,
%\\
%$N=3$: $P(x)=2x(x^3+21 x^2\mu +134 x \mu^2+252\mu^3)$,
%\\
%$N=4$: $P(x)=2x(x^4+36 x^3\mu +452 x^2\mu^2 +2304 x \mu^3+3912 \mu^4)$. 
\par
Hereafter we obtain the expression of the probability (\ref{p0t}) by inverting the Laplace transform   $H(\eta)$ when $\lambda=\mu$. 
We set $\rho(0):=\displaystyle\lim_{t\rightarrow +\infty} p(0,t)$, so that we shall express $p(0,t)$ as  the sum of a time-dependent 
term and the asymptotic value $\rho(0)$. 
%%%%%%%% 
\begin{proposition}\label{prop:p(0,t)}
If $\lambda=\mu$, for all $t\geq 0$ one has  
\begin{equation}
p(0,t)=\rho(0)+ 2 \sum_{k=2}^{N+1} \frac{Q(\alpha_k)}
{\beta_k} { e}^{\alpha_k t},
\label{p0t_lambda_eq_mu}
\end{equation}
with 
\begin{equation}
 \rho(0)  = \frac{2 {2N \choose N}}{{2N \choose N}+4^N}
 =  \frac{2}{1+\displaystyle\frac{\sqrt{\pi} \,N!}{\Gamma(N+1/2)}}, 
\label{eq:rho0}
\end{equation}
and 
\begin{equation}\label{eq:defbetak}
 \beta_k= \displaystyle{\lim_{\eta \rightarrow \alpha_k} \frac{P(\eta)}{\eta-\alpha_k}}
 =  \prod_{\stackrel{s=1}{s\neq k}}^{N+1}  (\alpha_k-\alpha_s),
 \qquad k=1,2,\ldots, N+1,
\end{equation}
where $0=\alpha_1>\alpha_2>\ldots >\alpha_{N+1}$ are the roots of the polynomial (\ref{pol_P}), and 
\begin{equation}
\label{pol_Q_P}
Q(x)=\prod_{r=0}^{N-1}\left[x+2\mu (2r+1)\right]. 
\end{equation}
\end{proposition}
\begin{proof}
Expanding the gamma functions in the right-hand-side of (\ref{p_Ltransf_lambda_eq_mu}), one has 
$$
H(\eta)=2 \,\frac{ Q(\eta)}{P(\eta)},
$$
for $Q$ and $P$ given in (\ref{pol_Q_P}) and (\ref{pol_P}), respectively. 
The roots of the $N$-degree polynomial  defined in (\ref{pol_Q_P})  are all distinct and negative, given by  
$-2\mu$, $-6 \mu$, $\ldots$, $-2(2N-1) \mu $.  Hence, by taking the inverse Laplace transform 
and making use of Eq.\ 2.1.4.7 of Prudnikov et al.\ \cite{Prudnikov5} we obtain 
$$
 p(0,t)=2 \sum_{k=1}^{N+1} \frac{Q(\alpha_k)} { \beta_k}\,e^{\alpha_k  t},
 \qquad t\geq 0,
$$
where $0=\alpha_1>\alpha_2>\ldots >\alpha_{N+1}$ are the roots of the polynomial $P$, due to Lemma \ref{lemma_pol}, 
and where $\beta_k$ is defined in (\ref{eq:defbetak}). 
Finally, after straightforward calculations one obtains the expression (\ref{p0t_lambda_eq_mu}). 
\end{proof}
The knowledge of $p(0,t)$ when $\lambda=\mu$, obtained in Proposition \ref{prop:p(0,t)}, allows to determine the expression of 
the probabilities (\ref{pkt}) in terms of the polynomials (\ref{pol_P}) and (\ref{pol_Q_P}), and of the 
hypergeometric function (\ref{hypergeometricf}). 
\begin{proposition}\label{prop:prt}
If $\lambda=\mu$, for all $t\geq 0$ one has  
\begin{eqnarray}
p(r,t) \!\!\!\! &=& \!\!\!\!  \frac{1}{4^N}{2N \choose N+r}\sum_{l=0}^N {N \choose l} \left(-e^{-4\mu t}\right)^l {}_{2}F_{1}\left(-2l,-N+r,-2N,2\right) 
 + \frac{\mu N}{2^{2N-2}}(-1)^{N-r}\sum_{j=0}^{N-1}{N-1 \choose j}(-1)^{N-1-j} 
\nonumber \\
&\times& \!\!\!\!  \left\{{2N-1 \choose N+r}{}_{2}F_{1}\left(-2j,-N+r+1,-2N+1,2\right)-{2N-1 \choose N+r-1}{}_{2}F_{1}\left(-2j,-N+r,-2N+1,2\right)\right\}\nonumber\\
&\times& \!\!\!\!  \left\{\sum_{k=1}^{N+1}\frac{R(\alpha_k)e^{-\left|\alpha_k\right|t}}{\left|\alpha_k\right|-2\mu(2N-1-2j)}-e^{-2 \mu t (2N-1-2j)}\sum_{k=1}^{N+1}\frac{R(\alpha_k)}{\left|\alpha_k\right|-2\mu(2N-1-2j)}\right\}\nonumber\\
&+& \!\!\!\!  \frac{\mu N}{2^{2N-2}}(-1)^{N-r}{2N \choose N+r}\sum_{j=0}^{N-1}{N-1 \choose j}(-1)^{N-1-j}{}_{2}F_{1}\left(-2j,-N+r,-2N,2\right)\nonumber\\
&\times& \!\!\!\!  \left\{\sum_{k=1}^{N+1}\frac{R(\alpha_k)e^{-\left|\alpha_k\right|t}}{\left|\alpha_k\right|-4\mu(N-j)}-e^{-4 \mu t (N-j)}\sum_{k=1}^{N+1}\frac{R(\alpha_k)}{\left|\alpha_k\right|-4\mu(N-j)}\right\}, 
\qquad\qquad r=1,2,\ldots,N. 
\label{pkt_lambda_eq_mu}
\end{eqnarray}
where 
$$
 R(\alpha_k)=\frac {Q(\alpha_k)} {\beta_k},
\qquad k=1,2,\ldots, N+1,
$$
with $\beta_k$ defined in (\ref{eq:defbetak}),
and where $0=\alpha_1>\alpha_2>\ldots >\alpha_{N+1}$ are the roots of the polynomial (\ref{pol_P}). 
\end{proposition}
%%%%%%%%%%%%%%%%%%%%
\begin{proof}
For $\lambda=\mu$, making use of  (\ref{p0t_lambda_eq_mu}) in the right-hand-side of 
Eq.\ (\ref{solFgenerale_lambda_neq_mu})  we have   
\begin{eqnarray*}
F(z,t) \!\!\!\!  &=&  \!\!\!\! 
\left[\frac{(z+1)^2- e^{- 4 \mu t} (1-z)^2}{4 z }\right]^N
+\frac{\mu N}{(2 z)^N} \frac{(1+z)(z-1)^{2 N-1} }{2^{N-2}} \sum_{j=0}^{N-1} {N-1 \choose j} (-1)^{N-1-j}
\left[\frac{z+1}{z-1} \right]^{2 j}
\nonumber
\\
&\times&  \!\!\!\! 
 \sum_{k=1}^{N+1} R(\alpha_k) \frac{1}{|\alpha_k|-2 \mu (2 N-1-2j)}
[{\rm e}^{-2 \mu (2 N-1-2 j) t}-e^{-|\alpha_k|\, t} ]-\frac{\mu N}{z^N 2^{2 N-2}} \sum_{j=0}^{N-1} {N-1 \choose j} (-1)^{N-1-j}
\nonumber
\\
&\times&   \!\!\!\! 
 (z-1)^{2 N-2 j} (z+1)^{2 j}
 \sum_{k=1}^{N+1} R(\alpha_k) \frac{1}{|\alpha_k|-4 \mu (N-j)}[{\rm e}^{-4 \mu (N-j) t}-e^{-|\alpha_k|\, t} ].
\end{eqnarray*}
Hence, by employing series expansion techniques one obtains Eq.\ (\ref{pkt_lambda_eq_mu}).
\end{proof}
%
%%%%%%%%%%%%%%%%%%%%%%%%%%%%%% \rho(k)
\begin{figure}[t]
\centering
\includegraphics[width=7.5cm]{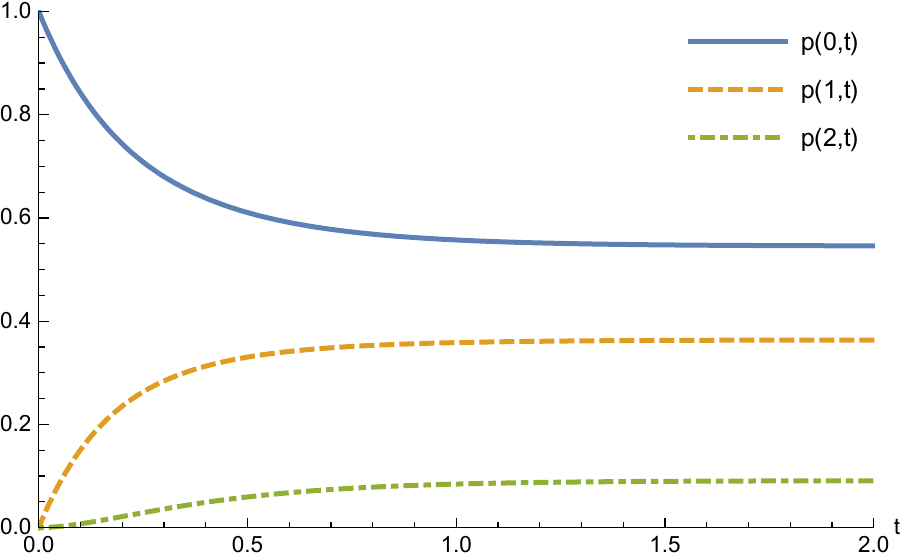}	
$\;\;$
\includegraphics[width=7.5cm]{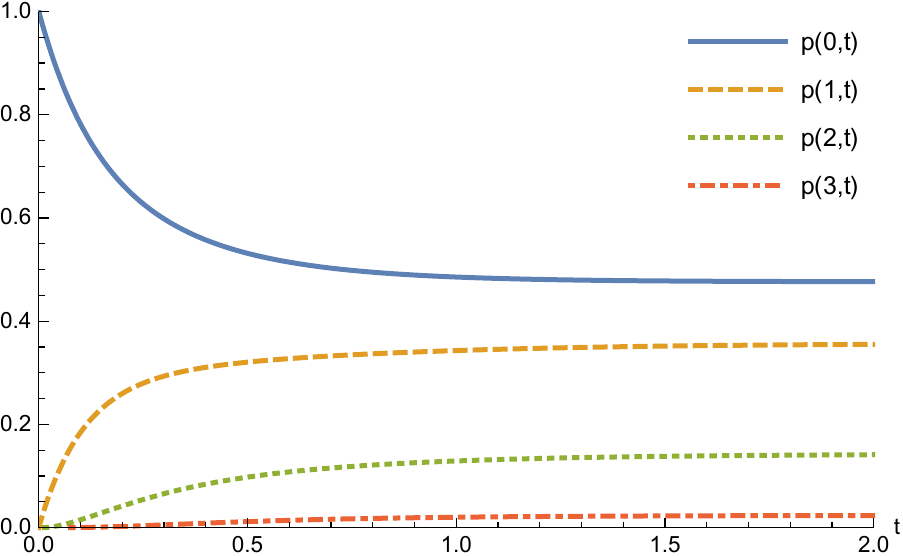}	
\caption{
The probabilities $p(r,t)$, given in (\ref{pkt_lambda_eq_mu}), are plotted for $N =2$ (left) and  $N = 3$ ( right), 
with $\lambda=\mu=1$.}
\label{probtransienti}
\end{figure}
%
%%%%%%%%%%%%%%%
\par
Figure \ref{probtransienti} shows the transient probabilities obtained in Proposition \ref{prop:prt} for two choices  of $N$. 
Unfortunately, for $\lambda \neq \mu$ the expression of $p(r,t)$ is very hard to be computed. However, in this case 
we adopt a Monte Carlo simulation approach to obtain estimates of the probabilities defined in (\ref{p0t}) and (\ref{pkt}). 
Some plots of estimates of such probabilities based on simulation and the corresponding exact values, when available,  
are provided in Figures \ref{fig:3tre}, \ref{fig:4quattro}, \ref{fig:5cinque} and \ref{fig:6sei}. 
In all cases, the estimates provide a quite good correspondence with the exact probabilities. 
%%%%%%%%%%%%%%%%%%%%%%%%%%%%%% nuove figure
\begin{figure}[t]
\centering
\includegraphics[width=8.3cm]{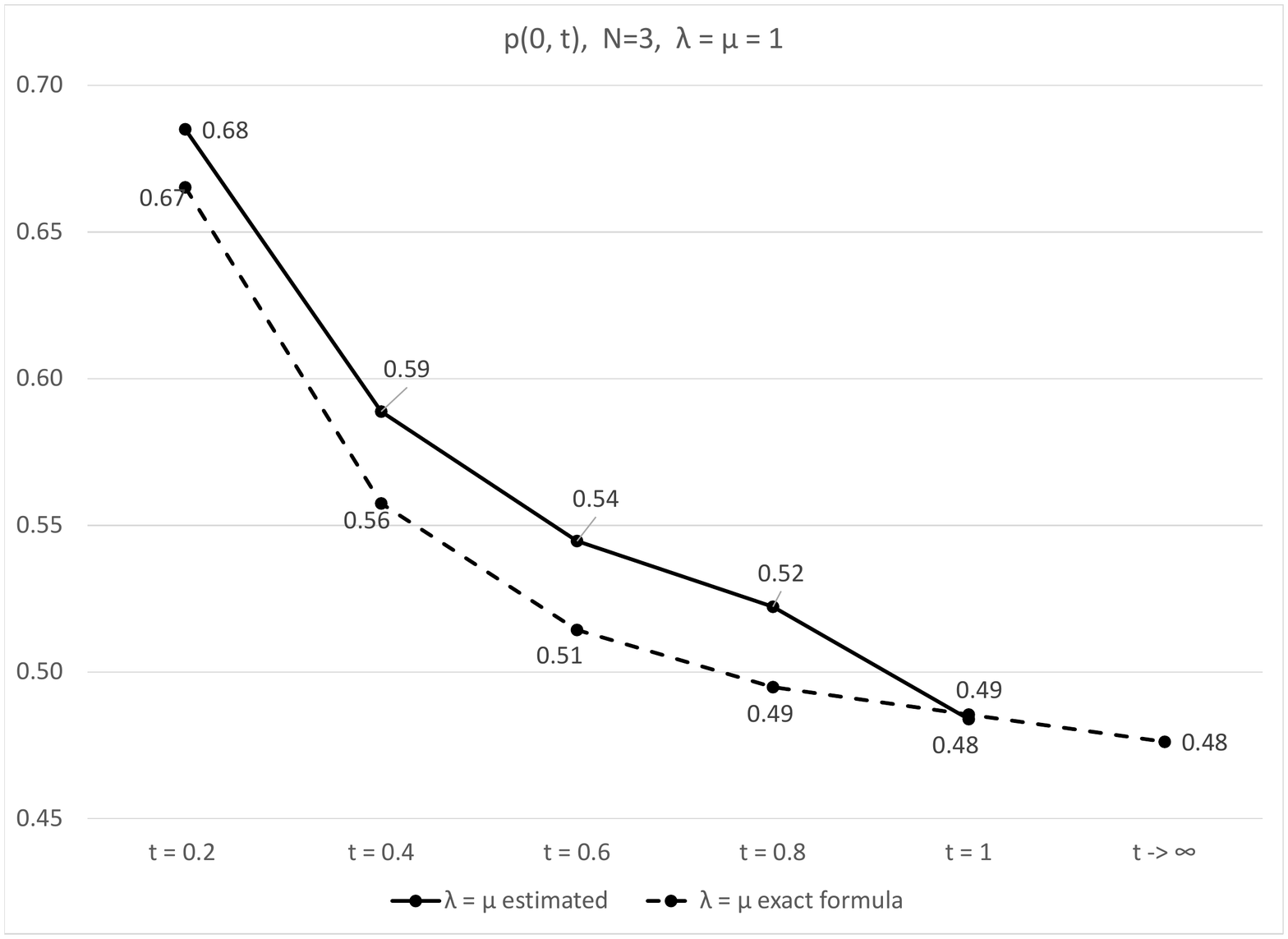}	
\includegraphics[width=8.3cm]{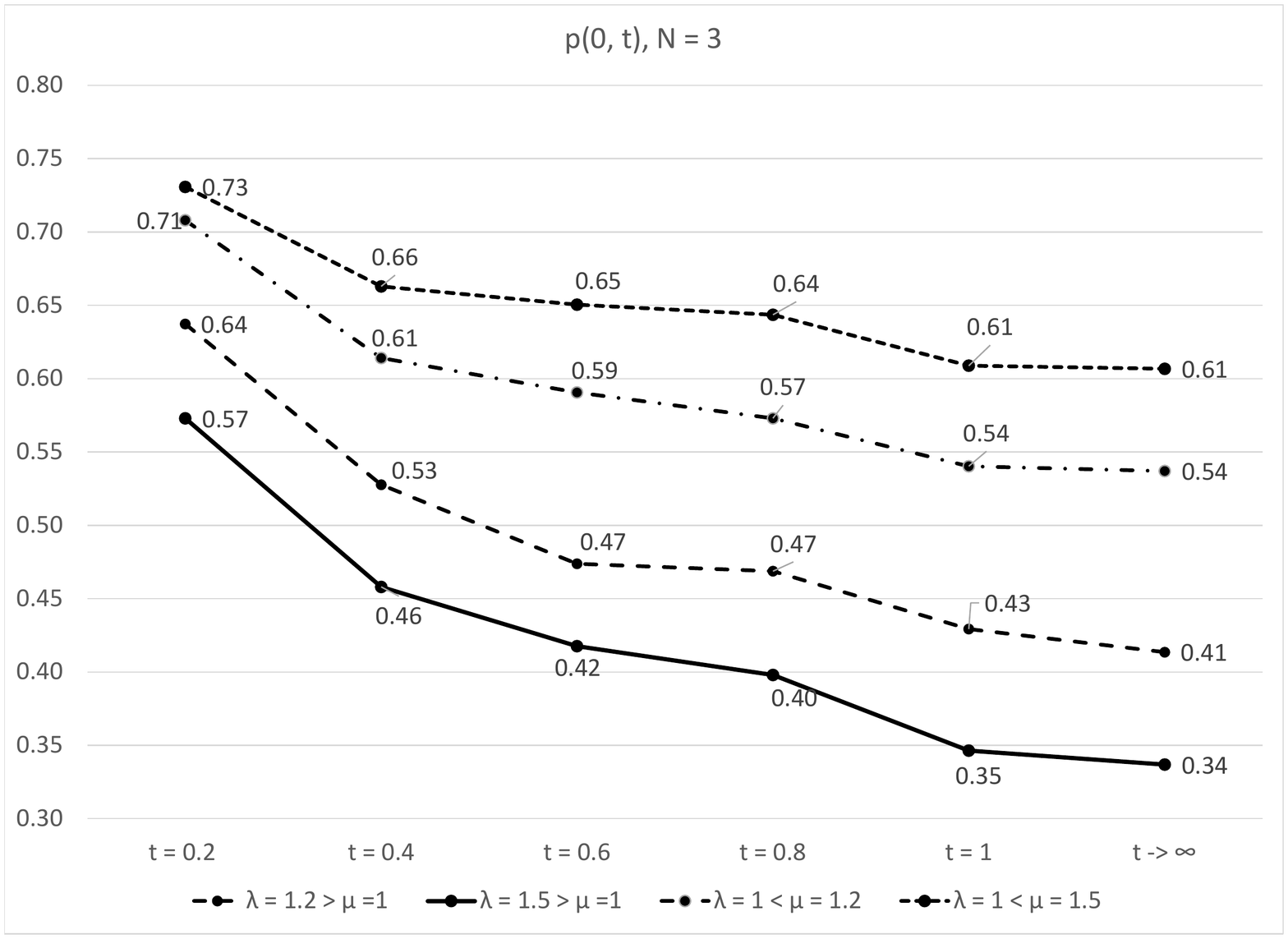}	
\caption{On the left:  the probability $p(0,t)$ (dashed line) given in (\ref{p0t_lambda_eq_mu}) compared with its estimation (continuous line) performed via $10^4$ Monte Carlo simulations, for $\lambda=\mu=1$ and $N =3$. On the right: the estimates of $p(0,t)$ for various choices of 
 $\lambda$, $\beta$ and $t$, with $N=3$.}
\label{fig:3tre}
\end{figure}
%%%%%%%%%%%%%%%%%%%%%%%%%%%%%%
%%%%%%%%%%%%%%%%%%%%%%%%%%%%%% nuove figure
\begin{figure}[t]
\centering
\includegraphics[width=8.3cm]{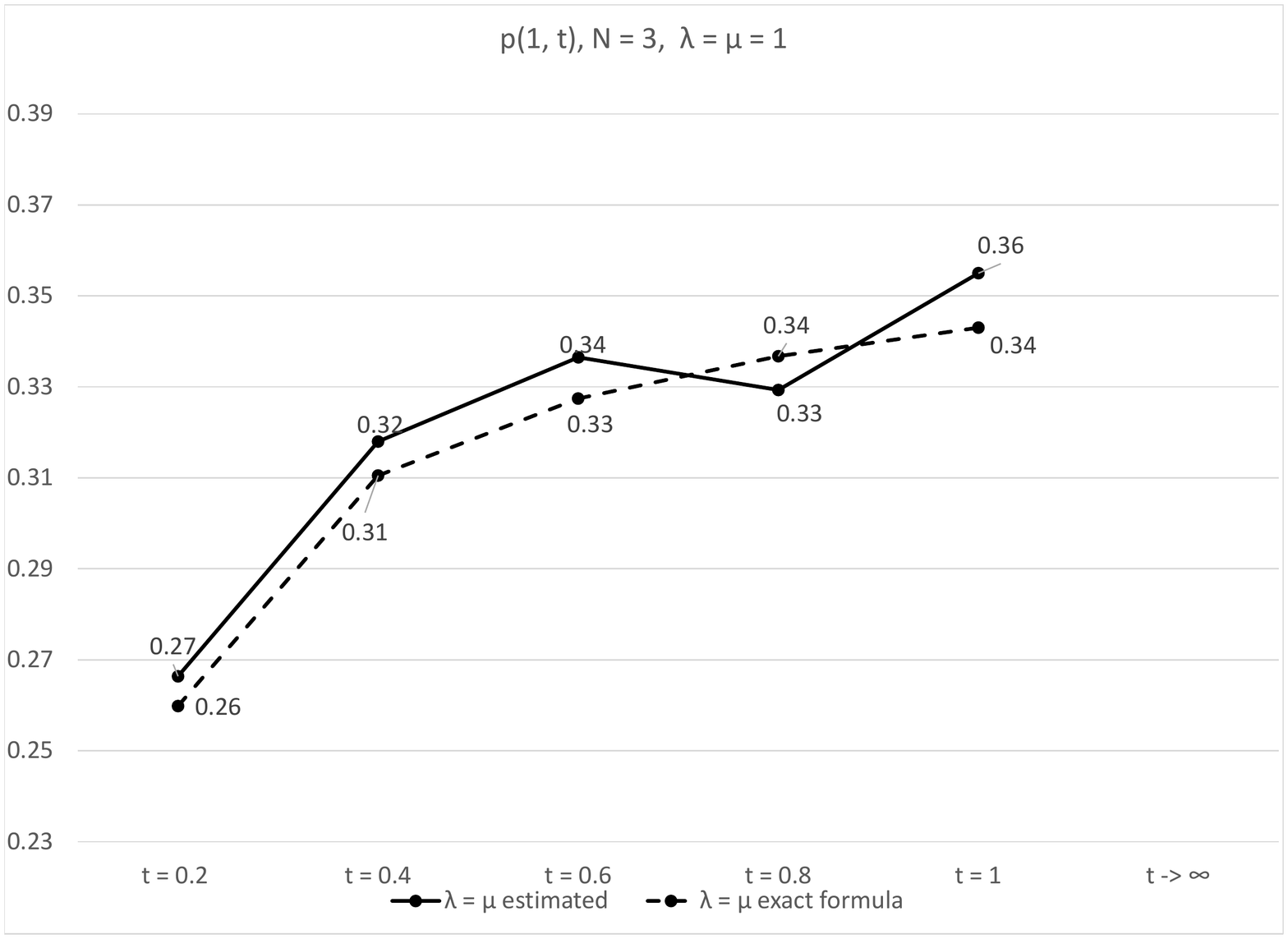}	
\includegraphics[width=8.3cm]{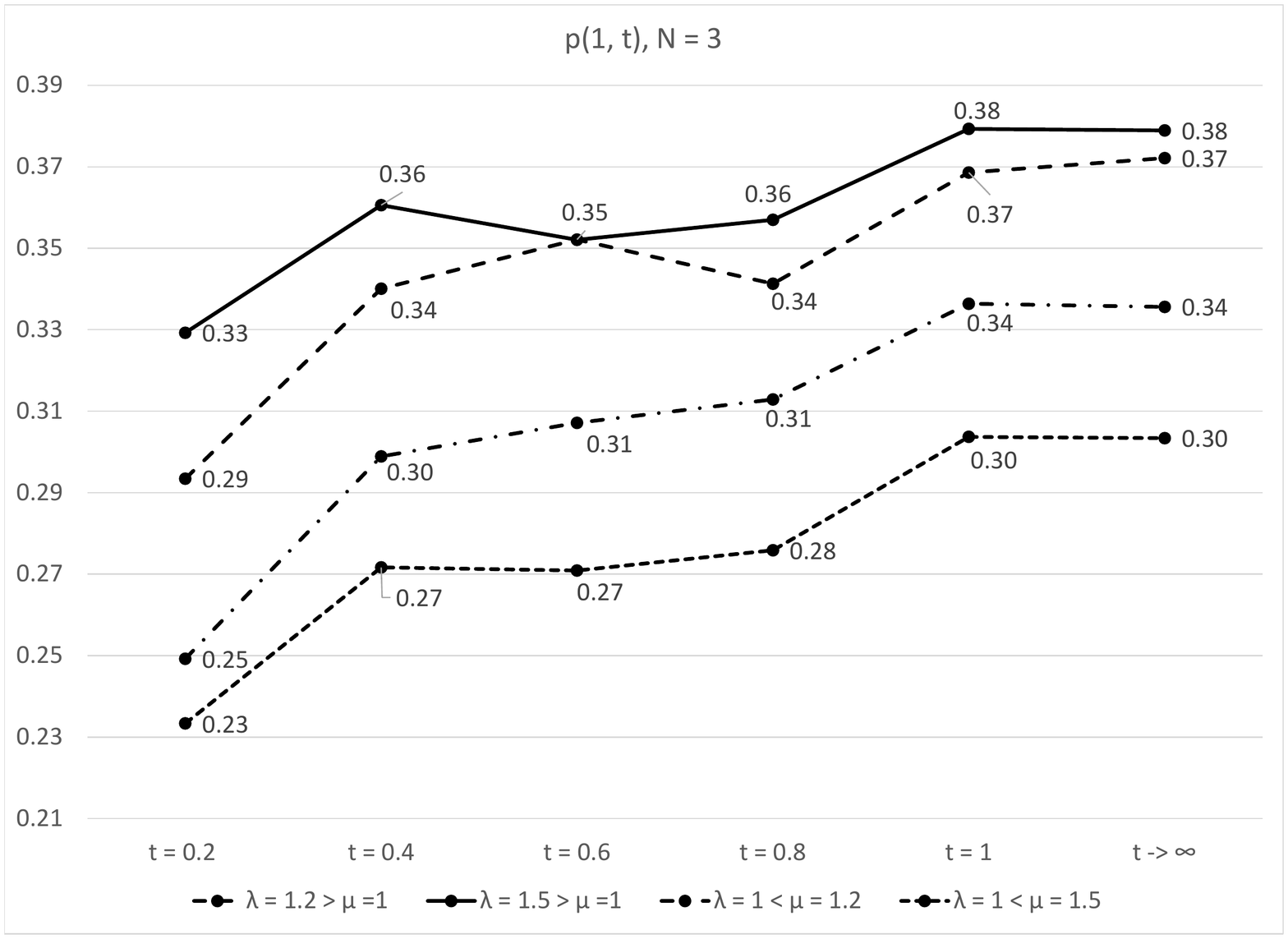}	
\caption{Same as Figure 3, for $p(1,t)$.}
\label{fig:4quattro}
\end{figure}
%%%%%%%%%%%%%%%%%%%%%%%%%%%%%%
%%%%%%%%%%%%%%%%%%%%%%%%%%%%%% nuove figure
\begin{figure}[t]
\centering
\includegraphics[width=8.3cm]{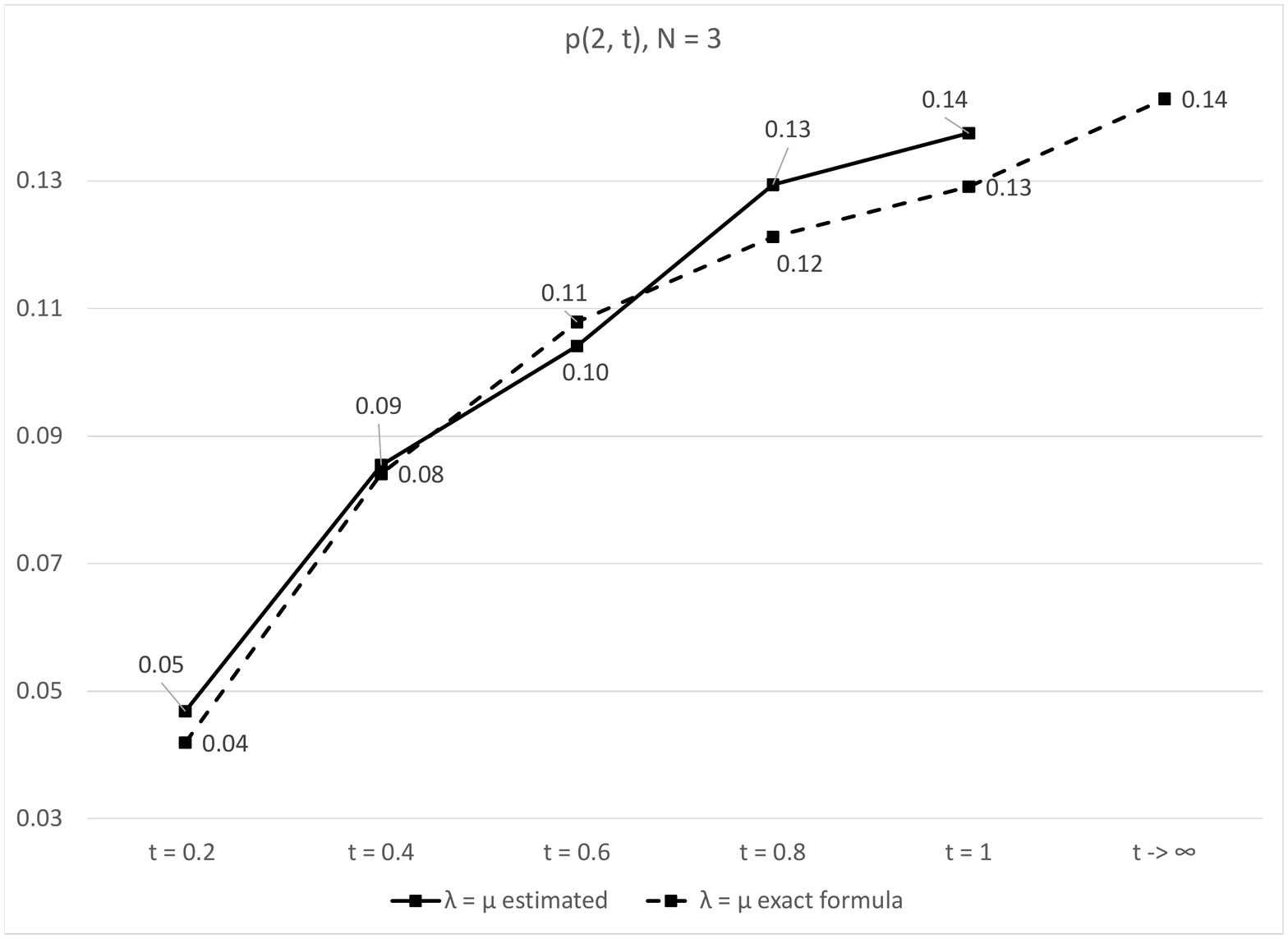}	
\includegraphics[width=8.3cm]{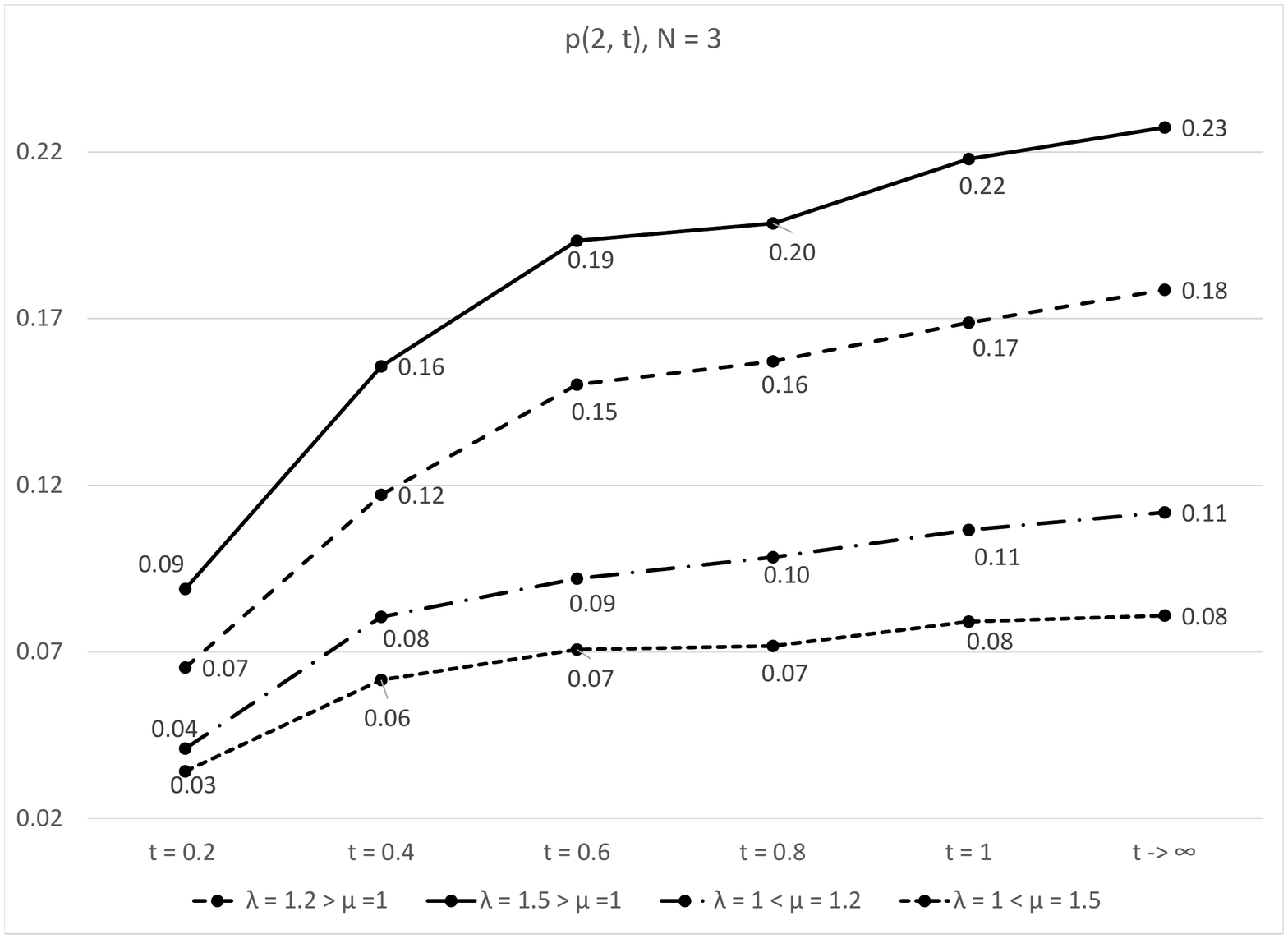}	
\caption{Same as Figure 3, for $p(2,t)$.}
\label{fig:5cinque}
\end{figure}
%%%%%%%%%%%%%%%%%%%%%%%%%%%%%%
%%%%%%%%%%%%%%%%%%%%%%%%%%%%%% nuove figure
\begin{figure}[t]
\centering
\includegraphics[width=8.3cm]{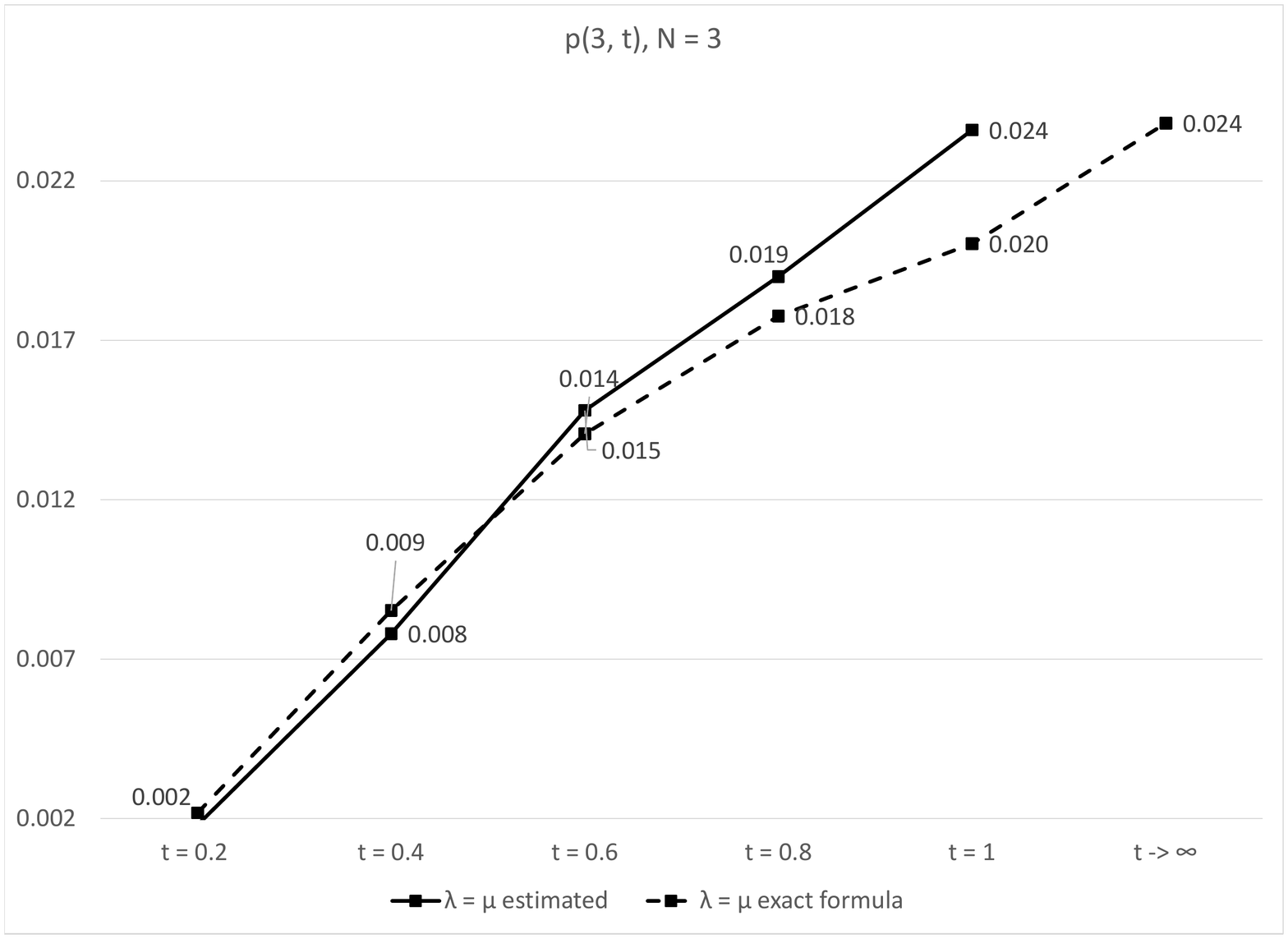}	
\includegraphics[width=8.3cm]{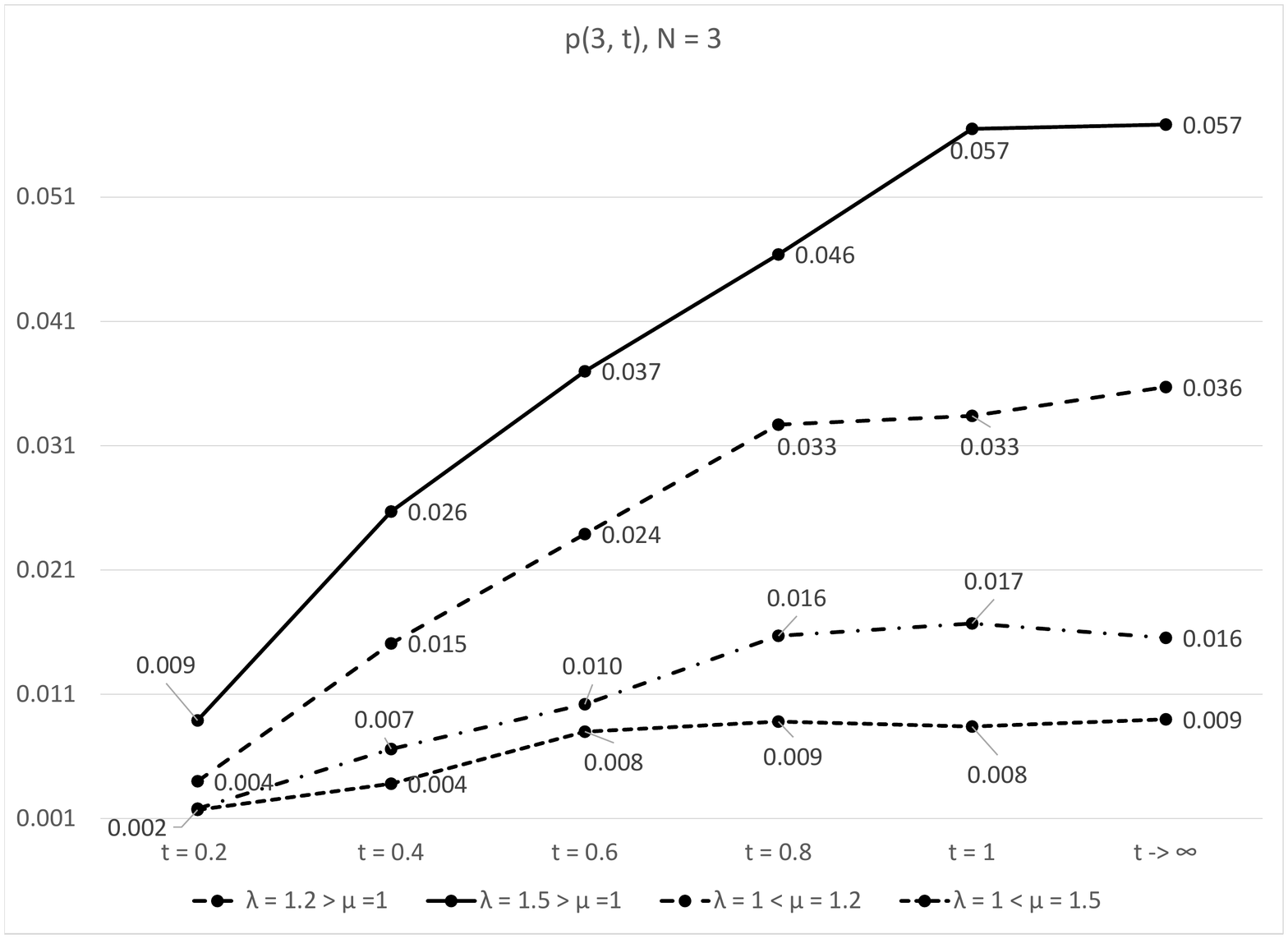}	
\caption{Same as Figure 3, for $p(3,t)$.}
\label{fig:6sei}
\end{figure}
%%%%%%%%%%%%%%%%%%%%%%%%%%%%%%
\par
The following section will be devoted to determine the asymptotic probability law of the process under investigation in the limit 
as $t\to \infty$. 
Note that some values of $\lim_{t\to \infty} p(r,t)$, shown in the Figures \ref{fig:3tre}, \ref{fig:4quattro}, 
\ref{fig:5cinque} and \ref{fig:6sei}, have been evaluated by means of Eq.\ (\ref{plimit_lambda_neq_mu_new}) below. 
% 
% ==============================================================
\section{Asymptotic results}\label{section:Asymptotic}
% ==============================================================
%According to the one-dimensional case, ${\cal N}$ admits a stationary distribution 
%if and only if $\lambda<\mu$. 
A typical problem of interest in the analysis of stochastic systems is the determination of the existence of a steady-state 
behavior when $t$ tends to $+\infty$. For instance, it is well known that the asymptotic distribution of the classical 
continuous-time Ehrenfest model is of binomial type (see, e.g.\ Section 2.1 of Dharmaraja et al.\  \cite{Dharmaraja2015}). 
Aiming to analyze the steady state of the present multi-type extension of the model, now we 
introduce the stationary probabilities 
\begin{equation}
 \rho(k):={\mathbb P}\left({\cal N}=k\right)=\lim_{t\rightarrow +\infty} p(k,t), \qquad k=0,1,\ldots,N,
 \label{eq:rhok}
\end{equation}
where ${\cal N}$ denotes the discrete random variable describing the stationary state of the system, 
with $p(0,t)$ and $p(k,t)$  defined respectively in (\ref{p0t}) and (\ref{pkt}).
The corresponding asymptotic probability generating function is given by 
$$
F(z):= \mathbb E\left[z^{\cal N}\right] = \lim_{t\to +\infty} F(z,t) = \rho(0)+\sum_{k\in {\bf N}} z^k \rho(k),
\qquad z\in [0,1],
$$
where $F(z,t)$ is defined in (\ref{FPgrande}). In the following proposition we obtain the explicit 
expression of $F(z)$, given in terms of the hypergeometric function (\ref{hypergeometricf}). 
We shall see that it depends on the rates $\lambda$ and $\mu$ 
only through their ratio. Hence, now we set 
\begin{equation}
 \varrho=\frac{\lambda}{\mu}.
 \label{eq:defrho}
\end{equation}
\begin{proposition}\label{asymptotic}
The  probability generating function  of ${\cal N}$, for $z\in [0,1]$ results:
\begin{eqnarray}
F(z) &=& 
\frac{(1+\varrho z)^{2N}}{z^N(1+\varrho)^{2N}}\left[1+ g(\varrho,N) 
 \sum_{j=0}^{N-1} {N \choose j+1} \left(\frac{ z-1}{1+ \varrho z}\right)^{j+1}
{}_{2}F_{1}\left(-N,j+1,j+2,\frac{\varrho( z-1)}{1+ \varrho z}\right)\right], 
\label{asymp_Fz_expr}
\end{eqnarray}
where $\varrho$ is defined in (\ref{eq:defrho}), and 
\begin{equation}
g(\varrho,N) := \frac{1}{{}_{2}F_{1}\left(-N,1,1+N,- \varrho\right)},  
\label{g_1}
\end{equation}
\end{proposition}
\begin{proof}
The proof is given in Appendix A.
\end{proof}
Note that, due to (\ref{asymp_Fz_expr}), it is not hard to see that $F(1)= 1$. 
We are now able to obtain the steady-state distribution of the multi-type extension of the 
continuous-time Ehrenfest model. 
\begin{proposition}\label{stationaryprob}
The stationary probabilities defined in (\ref{eq:rhok}) are  given by 
\begin{equation}
 \rho(k)
 = \frac{g(\varrho,N)}{ {2N \choose N}} \varrho^k {2N \choose N+k},  
 \qquad k=0,1,\ldots,N, 
\label{plimit_lambda_neq_mu_new}
\end{equation}
where the function $g$ has been introduced in (\ref{g_1}), and $\varrho$ is defined in (\ref{eq:defrho}). 
\end{proposition}
\begin{proof}
The proof is given in Appendix A.
\end{proof}
\par
The stationary probabilities given in Proposition \ref{stationaryprob} are plotted in Figure \ref{rho_k} 
for various choices of $N$ and $\varrho$. 
%
%%%%%%%%%%%%%%%%%%%%%%%%%%%%%% \rho(k)
\begin{figure}[t]
\centering
\includegraphics[width=7cm]{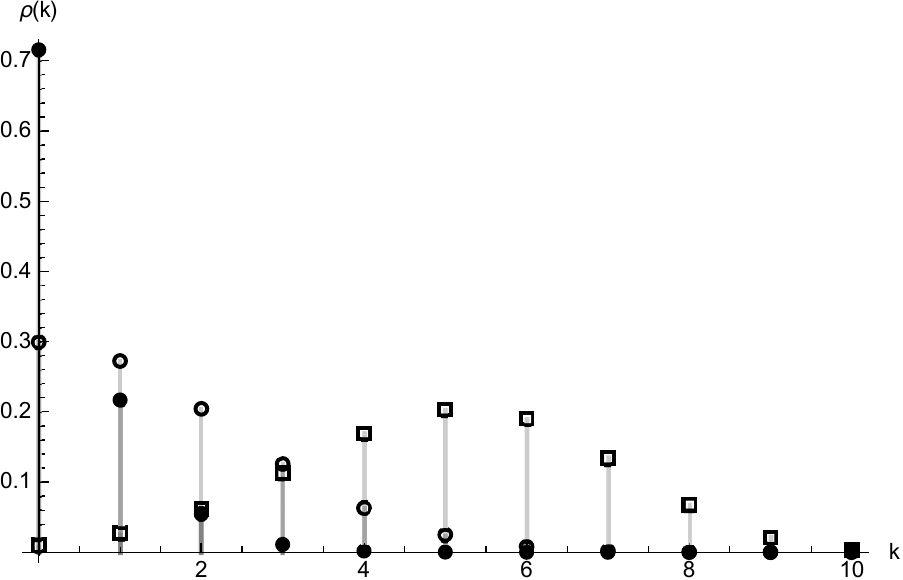}	
\includegraphics[width=7cm]{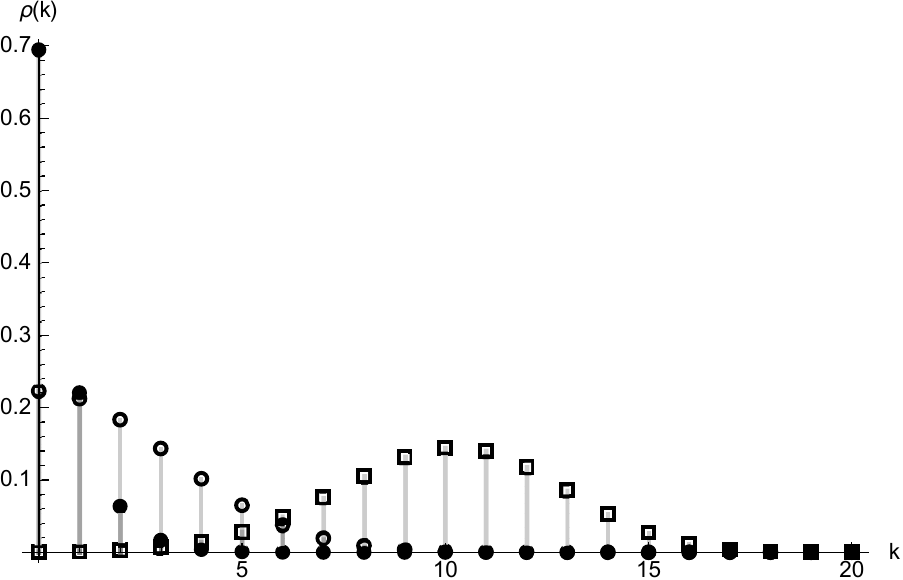}	
\caption{
The stationary probabilities $\rho(k)$ given in (\ref{plimit_lambda_neq_mu_new}) %and (\ref{plimit_lambda_eq_mu}) 
are plotted for $N=10$ on the left, $N=20$ on the right, and for $\varrho=1$ (empty circle), $\varrho=1/3$ (full circle), 
$\varrho=3$ (square).}
\label{rho_k}
\end{figure}
%
%%%%%%%%%%%%%%%
%\begin{remark}
%\blue{It is easy to observe that the function $g(N)$ introduced in (\ref{g_l_ug_m}) is equal to the stationary probability $\rho(0)$ expressed in (\ref{plimit_lambda_eq_mu}) for $\lambda=\mu$. This happens also for the function $g(\lambda,\mu,N)$ given in (\ref{g_1}), that is $g(\lambda,\mu,N)=\rho(0)$ with $\rho(0)$ expressed in (\ref{plimit_lambda_neq_mu}). DA DIMOSTARE (NUMERICAMENTE SI TROVA)
%}
%\end{remark}
%
\par
A relevant role is played by the stationary probability $\rho(0)$, which is the probability that the system 
is asymptotically empty. The case $\varrho=1$, i.e.\ $\lambda=\mu$, has been already considered 
in Proposition \ref{prop:p(0,t)}, where it is shown that $p(0,t)$ tends to $\rho(0)$ exponentially. 
\begin{remark}
We note that, from the formula (15.1.23) of 
Abramowitz and Stegun\cite{Abram1994} and properties of the Gamma function,  Eq.\ (\ref{g_1}) gives
\begin{equation}
 g(1,N)=\displaystyle\frac{2{2N \choose N} }{{2N \choose N}+4^N}. 
\label{eq:g1N}
\end{equation}
Hence, if $\varrho=1$ then the stationary probabilities (\ref{plimit_lambda_neq_mu_new}) can be written as  
\begin{equation}
 %\rho(0)=\frac{2 {2N \choose N}}{{2N \choose N}+4^N},\qquad 
 \rho(k)=\mathbb P\left({\cal N}=k\right)
 =\frac{2 {2N \choose N+k}}{{2N \choose N}+4^N}, 
 \qquad k=0,1,\ldots,N.
 \label{plimit_lambda_eq_mu}
\end{equation}
In this case, given a random variable ${\cal B}\sim {\rm Bin}\left(2N,\,\frac{1}{2}\right)$, from (\ref{plimit_lambda_eq_mu})  
the following decomposition holds, for $k=0,1,\ldots,N$,
$$
\rho(k) =c \left[ \mathbb P\left({\cal B}=N-k\right)+ \mathbb P\left({\cal B}=N+k\right)\right],\qquad 
\hbox{with \ } c=\frac{4^N}{4^N+{2N \choose N}}. 
$$
\end{remark}
\begin{remark}
Note that, when $\varrho=1$, we can verify explicitly that $\sum_{k=0}^N \rho(k)=1$. Indeed, noting that 
$$
 2^{2N}={2N \choose N}+2\sum_{k=1}^{N} {2N \choose N+k},
$$
from (\ref{plimit_lambda_eq_mu}) we have 
$$
 \sum_{k=0}^N \rho(k)=\frac{2}{{2N \choose N}+4^N}\sum_{k=0}^{N} {2N \choose N+k}
 =\frac{1}{ {2N \choose N}+4^N}\left(2{2N \choose N}+4^N-{2N \choose N}\right)=1.
$$
\end{remark}
%
%Note that the expression of $\rho(0)$ given in (\ref{p0limit_lambda_eq_mu}) can be obtained also from (\ref{p0t_lambda_eq_mu}) as limit for $t\rightarrow \infty$.
%
\begin{remark}\label{remark:3}
It is worth mentioning that, in the case $\lambda=\mu$, we can disclose the explicit relationship between the stationary probabilities $\rho(k)$ given in (\ref{plimit_lambda_neq_mu_new}) and the stationary probabilities of the classical Ehrenfest model. 
Indeed, it is well known  that, for $\lambda=\mu$, the stationary probabilities of the Ehrenfest model are given by 
(see, for instance, Eq.\ (16) of Dharmaraja et al.\ \cite{Dharmaraja2015})
$$
 \widetilde{q}_k:={2 N\choose N-k} 2^{-2 N},\qquad k\in\{-N,-N+1,\ldots,-1,0,1,\ldots, N\}.
$$
In order to compare the probabilities $\tilde{q}_k$ with $\rho(k)$, 
we first determine a suitable normalization constant $c(\lambda, N)$ such that
$$
 c(\lambda, N)\cdot \left(\sum_{j=-N}^0 \widetilde{q}_j+\sum_{j=0}^N \widetilde{q}_j\right)=1,
$$
and thus
$$
 c(\lambda, N)
 =4^{-N}\left\{2 {2 N\choose N}+\left[{2 N\choose N-1} +{2 N\choose N+1}\right]  {}_{2}F_{1}(1,1-N;N+2;-1) \right\}.
$$
Hence, the following identity holds
$$
 \rho(k)=\frac{\widetilde{q}_k+\widetilde{q}_{-k}}{c(\lambda, N)},
 \qquad k=0,1,\ldots,N.
$$
Note that the special role of the state 0 in the multi-type Ehrenfest model yields $\rho(0)= 2 \widetilde{q}_0/c(\lambda, N)$. 
\end{remark}
\par
Now we provide the asymptotic mean, variance and coefficient of variation of ${\cal N}$. 
\begin{proposition}\label{prop:asymptoticmvcv}
The asymptotic mean,  the asymptotic variance and the the asymptotic
coefficient of variation  of  ${\cal N}$ are given respectively by:
\begin{eqnarray*}%\label{asym_mean_var}
&&
E\left[{\cal N}\right]
=\frac{N}{1+\varrho}\left[\varrho-1 +g(\varrho,N)\right],
\\
&&
Var\left[{\cal N}\right]
=\frac{N}{(1+\varrho)^2}\left[\varrho(2-g(\varrho,N)-g(\varrho,N)\,N)
+N (1-g(\varrho,N))g(\varrho,N) \right],
\\
&&
CV\left[{\cal N}\right]
=\sqrt{\frac{\varrho (2-g(\varrho,N))}{N\left[\varrho-1+ g(\varrho,N)\right]^2}
-\frac{  g(\varrho,N)}{\varrho-1 +  g(\varrho,N)}},
\end{eqnarray*}
%
%\label{CV_l_neq_m}
%
where the function $g$ is provided in (\ref{g_1}). 
\end{proposition}
\begin{proof}
The given results follow from the probability generating function given in (\ref{asymp_Fz_expr}). 
%Result (\ref{asym_mean_var}) follows making use of the probability generating function obtained in Proposition (\ref{stationaryprob}) given in (\ref{asymp_Fz_expr}). For $\lambda=\mu$ we simply obtain (\ref{asym_mean_var_l_eq_m}) by recalling (\ref{g_l_ug_m}).
\end{proof}
%%%%%%%%%%%%%%%%%%%%%%
%

%%%%%%%%%%%%%%%%%%%%%%%%
%%
In Figure \ref{mean_var_cv} the stationary mean, variance and coefficient of variation given in 
Proposition \ref{prop:asymptoticmvcv} are plotted for $N=20$, and for different choices of $\varrho$. 
%%%%%%%%%%%%%%%%%%%%%%%%%%%%%%  
\begin{figure}[t]
\centering
\includegraphics[width=5.5cm]{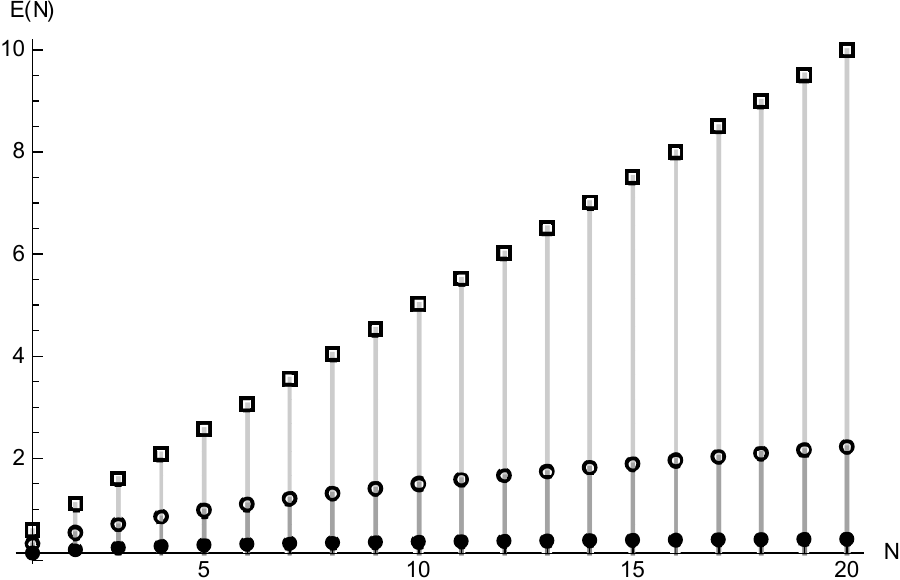}	
\includegraphics[width=5.5cm]{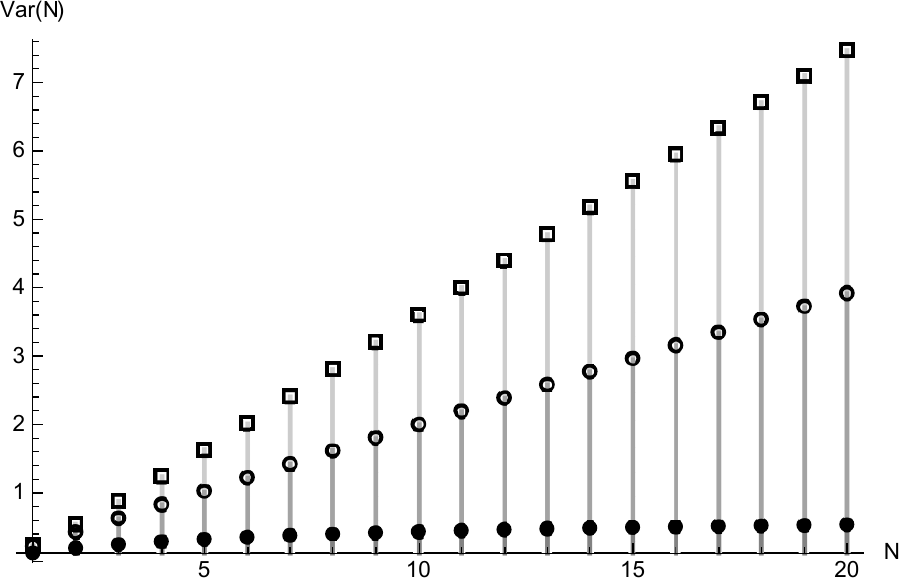}	
\includegraphics[width=5.5cm]{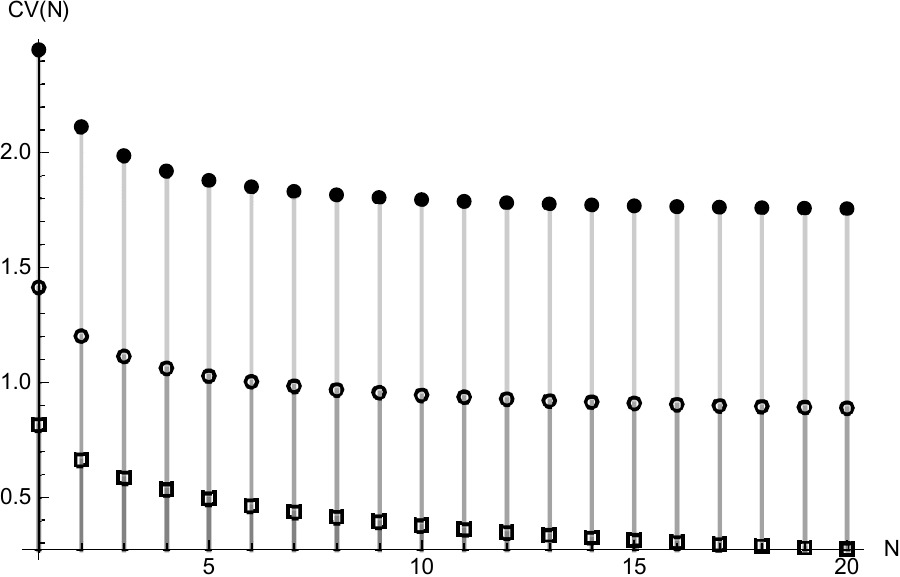}	
\caption{The stationary mean, variance and coefficient of variation are plotted for $N=20$ and for $\varrho=1$ (empty circle), $\varrho=1/3$ (full circle), $\varrho=3$ (square).}
\label{mean_var_cv}
\end{figure}
%
%%%%%%%%%%%%%%%%%%%%%%%%%%%%%
\begin{remark}\label{rem:EVCV}
If $\varrho=1$, i.e.\ $\lambda=\mu$, making use of (\ref{eq:g1N}) we can see that the quantities provided in   
Proposition \ref{prop:asymptoticmvcv} become respectively 
\begin{eqnarray*}%\label{asym_mean_var_l_eq_m}
&&
E\left[{\cal N}\right]=\frac{N {2N \choose N}}{{2N \choose N}+4^N},
\\
&&
Var\left[{\cal N}\right]
=\frac{N}{2}\left[1-\frac{{2N \choose N}}{{2N \choose N}+4^N}\left(1+\frac{2N{2N \choose N}}{{2N \choose N}+4^N}\right)\right],
\\
&&
CV\left[{\cal N}\right]
=\sqrt{\frac{2^{2N-1}}{N{2N \choose N}}\left(1+\frac{2^{2N}}{{2N \choose N}}\right)-1}.
\end{eqnarray*}
\end{remark}
%% \label{CV_l_eq_m}
\par
In order to investigate the behaviour of the mean, the variance and the coefficient of variation of $\cal N$ 
when $N$ is large, let us now discuss the behavior of $g(\varrho,N)$ for $N$ large. 
In spite of the difficulty in managing the Gauss hypergeometric function in the denominator of 
Equation (\ref{g_1}), in the following Lemma we disclose an useful asymptotic result, whose proof is given in Appendix A. 
\begin{lemma}\label{lemma_appr_g}
If $\varrho<1$, then  for $N$ large the function $g(\varrho,N)$ defined in (\ref{g_1})  can be approximated as
\begin{equation}
g(\varrho,N)\approx 
\frac{2^{3-2N}(2N)!\sqrt{\pi}N^{\frac{5}{2}}(\varrho-1)^3
\left(\log\left[\frac{(\varrho+1)^2}{4\varrho}\right]\right)^{\frac{5}{2}}}{(N!)^2 \left[3(\varrho-1)^3
+N  \left(\log\left[\frac{(\varrho+1)^2}{4\varrho}\right]\right)^{\frac{5}{2}}
\left[(3\varrho+1)^2-8N(\varrho-1)^2\right]\right]}.
\label{approx_g_1}
\end{equation}
\end{lemma}
%
%%%%%%%%%%%%%%%%%%%%%%%
%
%%%%%%%%%%%%%%%%%%
\begin{proposition}%
The asymptotic mean, the asymptotic variance and the asymptotic coefficient of variation given in 
Proposition \ref{prop:asymptoticmvcv}, for $N\to +\infty$ admit the following behaviour:
\begin{itemize}
\item if $\varrho > 1$, then both $E\left[{\cal N}\right]$ and $Var\left[{\cal N}\right]$ tend to $+\infty$, 
whereas $CV\left[{\cal N}\right]$ tends to $0$;
\item if $\varrho=1$, then both $E\left[{\cal N}\right]$ and $Var\left[{\cal N}\right]$ tend to $+\infty$, 
whereas $CV\left[{\cal N}\right]$ tends to $\sqrt{\frac{\pi}{2}-1}$;
\item if $\varrho < 1$, then following limits hold: 
\begin{equation}
\label{mean_as_la<mu_Nlarge}
\lim_{N\to \infty}E\left[{\cal N}\right]= \frac{\varrho}{1-\varrho},
\end{equation}
\begin{equation}
\label{var_as_la<mu_Nlarge}
\lim_{N\to \infty}Var\left[{\cal N}\right]
= \frac{3(1-\varrho)^3}{8 (1+ \varrho)^2\left(\log\left[\frac{(1+\varrho)^2}{4\varrho}\right]\right)^{5/2}}
 -\frac{145\varrho^4+492\varrho^3 +374\varrho^2 +12\varrho+1}{128(1-\varrho^2)^2},
\end{equation}
\begin{equation}
\label{cv_as_la<mu_Nlarge}
\lim_{N\to \infty}CV\left[{\cal N}\right]
= \frac{1}{8\varrho(1+\varrho)}
\sqrt{\frac{  24(1-\varrho)^5}{\left(\log\left[\frac{(1+\varrho)^2}{4\varrho}\right]\right)^{5/2}}
-\frac{  145\varrho^4+492 \varrho^3 +374\varrho^2 +12\varrho +1}{2}}.
\end{equation}
\end{itemize}
\end{proposition}
\begin{proof}
If $\varrho > 1$, the function $g$ defined in (\ref{g_1}) is a divergent series as $N\to \infty$. 
Hence, in this case the mean and variance given in Proposition \ref{prop:asymptoticmvcv} diverge, 
whereas the corresponding coefficient of variation tends to $0$. 
\par
When $\varrho=1$, the mean and variance given in Proposition \ref{prop:asymptoticmvcv} 
diverge by comparing infinities. For the related coefficient of variation, making use of  
${2N \choose N}=\frac{4^N}{\Gamma(N+1) \sqrt{\pi}}\Gamma\left(N+\frac{1}{2}\right)$,  it results
$$
CV\left[{\cal N}\right]=\sqrt{\frac{\Gamma(N+1) \sqrt{\pi}}{2N\Gamma\left(N+ {1}/{2}\right)}+\left(\frac{\Gamma(N+1)}{N^{-1/2}\Gamma\left(N+ {1}/{2}\right)}\right)^2\frac{\pi}{2}-1}
\qquad \stackrel{N\rightarrow \infty}{\rightarrow}\qquad \sqrt{\frac{\pi}{2}-1},
$$
since $\frac{N! \sqrt{\pi}}{2N\Gamma\left(N+\frac{1}{2}\right)}$ goes to zero by comparing infinities, and $\frac{\Gamma(N+1)}{N^{-1/2}\Gamma\left(N+\frac{1}{2}\right)}$ tends to $1$ due to formula (6.1.46) of Abramowitz and Stegun \cite{Abram1994}. 
\par
For $\varrho < 1$, by substituting (\ref{approx_g_1}) in the asymptotic mean and variance given in 
Proposition \ref{prop:asymptoticmvcv}, as $N\to \infty$, one obtains the results (\ref{mean_as_la<mu_Nlarge}) and 
(\ref{var_as_la<mu_Nlarge}), and thus the limit (\ref{cv_as_la<mu_Nlarge}). 
\end{proof}
\par
See also the  details provided in Eq.\ (\ref{eq:limEVN}) below for the case $\varrho=1$. 
\par
In order to appreciate the goodness of the numerical approximation provided for $g(\varrho,N)$ 
in Lemma \ref{lemma_appr_g}, in Table 1 we compare the exact stationary probabilities given 
in Proposition \ref{stationaryprob} with the corresponding quantities approximated by means 
of (\ref{approx_g_1}). The considered cases include three choices of $\varrho < 1$, 
and confirm that the approximation is satisfactory when $N$ is large. 
\begin{table}[t] 
\begin{center}
{\footnotesize
\begin{tabular}{|r|  | l|l|  | l|l| |  l|l|}
\hline 
{$N=100$}             & $\varrho =0.25$ &                  &$\varrho =0.5$ &                   &$\varrho =0.75$ &                         \\
 $k$   & $\rho(k)$       &$\tilde{\rho(k)}$ & $\rho(k)$      &$\tilde{\rho(k)}$& $\rho(k)$     & $\tilde{\rho(k)}$       \\
\hline
0          &0.754044         & 0.75298          &0.513742        &0.512301        &0.288403        &0.264746\\
10         &2.6543$\cdot 10^{-7}$   &2.65056$\cdot 10^{-7}$  &0.000185182     &0.000184663     &0.00599468      &0.00550296\\
20         &1.24762$\cdot 10^{-14}$ &1.24586$\cdot 10^{-14}$  &8.91315$\cdot 10^{-9}$ &8.88815$\cdot 10^{-9}$ &0.0000166384    &0.0000152736\\
30         &7.3537$\cdot 10^{-23}$  & 7.34332$\cdot 10^{-23}$ &5.37966$\cdot 10^{-14}$ &5.36457$\cdot 10^{-14}$ &5.7909$\cdot 10^{-9}$  &5.3159$\cdot 10^{-9}$ \\
40         &4.84975$\cdot 10^{-32}$ & 4.84291$\cdot 10^{-32}$ &3.63302$\cdot 10^{-20}$ &3.62283$\cdot 10^{-20}$ &2.25513$\cdot 10^{-13}$ &2.07015$\cdot 10^{-13}$ \\
50         &2.98151$\cdot 10^{-42}$ & 2.9773$\cdot 10^{-42}$  &2.2871$\cdot 10^{-27}$ &2.28068$\cdot 10^{-27}$ &8.18656$\cdot 10^{-19}$ &7.51505$\cdot 10^{-19}$ \\
60         &1.28441$\cdot 10^{-53}$ & 1.2826$\cdot 10^{-53}$  &1.00891$\cdot 10^{-35}$ &1.00608$\cdot 10^{-35}$ &2.08248$\cdot 10^{-25}$ &1.91166$\cdot 10^{-25}$ \\
70         &2.44772$\cdot 10^{-66}$ & 2.44427$\cdot 10^{-66}$ &1.96884$\cdot 10^{-45}$ &1.96332$\cdot 10^{-45}$ &2.34344$\cdot 10^{-33}$ &2.15121$\cdot 10^{-33}$ \\
80         &9.19408$\cdot 10^{-81}$ & 9.18111$\cdot 10^{-81}$ &7.57281$\cdot 10^{-57}$ &7.55157$\cdot 10^{-57}$ &5.19771$\cdot 10^{-43}$ &4.77136$\cdot 10^{-43}$ \\
90         &1.21998$\cdot 10^{-97}$ & 1.21826$\cdot 10^{-97}$ &1.02896$\cdot 10^{-70}$ &1.02608$\cdot 10^{-70}$ &4.07256$\cdot 10^{-55}$ &3.7385$\cdot 10^{-55}$ \\
100        &5.18222$\cdot 10^{-120}$ & 5.17491$\cdot 10^{-120}$ &4.47573$\cdot 10^{-90}$ &4.46318$\cdot 10^{-90}$ &1.02151$\cdot 10^{-72}$ &9.37724$\cdot 10^{-73}$ \\
\hline
{$N=500$}   &      &  &      & &     &        \\
 $k$   &   $\rho(k)$         &$\tilde{\rho(k)}$ & $\rho(k)$      &$\tilde{\rho(k)}$& $\rho(k)$     & $\tilde{\rho(k)}$       \\
\hline
0   &0.75082             &0.750828            & 0.502942         &0.502939          &0.2596           &0.2593\\
50  &3.97755$\cdot 10^{-33}$   &3.97755$\cdot 10^{-33}$    & 2.99981$\cdot 10^{-18}$ &2.99979$\cdot 10^{-18}$  &9.87303$\cdot 10^{-10}$ &9.86145$\cdot 10^{-10}$ \\
100 &8.58336$\cdot 10^{-70}$    &8.58336$\cdot 10^{-70}$    & 7.28844$\cdot 10^{-40}$ &7.28844$\cdot 10^{-40}$  &1.52952$\cdot 10^{-22}$ &1.52772$\cdot 10^{-22}$ \\
150 &5.48926$\cdot 10^{-111}$   &5.48926$\cdot 10^{-111}$   &5.24796$\cdot 10^{-66}$  &5.24793$\cdot 10^{-66}$  &7.02221$\cdot 10^{-40}$ &7.01398$\cdot 10^{-40}$ \\
200 &5.83949$\cdot 10^{-157}$   &5.83949$\cdot 10^{-157}$   &6.28567$\cdot 10^{-97}$  &6.28563$\cdot 10^{-97}$  &5.36288$\cdot 10^{-62}$ &5.35659$\cdot 10^{-62}$ \\
250 &4.09278$\cdot 10^{-208}$   &4.09278$\cdot 10^{-208}$   &4.96015$\cdot 10^{-133}$ &4.96012$\cdot 10^{-133}$ &2.69839$\cdot 10^{-89}$ &2.69522$\cdot 10^{-89}$ \\
300 &4.42983$\cdot 10^{-265}$   &4.42983$\cdot 10^{-265}$   &6.04455$\cdot 10^{-175}$ &6.04451$\cdot 10^{-175}$ &2.0967$\cdot 10^{-122}$ &2.09424$\cdot 10^{-122}$ \\
350 &7.0932955$\cdot 10^{-329}$ &7.0932949$\cdot 10^{-329}$ &1.08974$\cdot 10^{-223}$ &1.08974$\cdot 10^{-223}$ &2.41023$\cdot 10^{-162}$ &2.40741$\cdot 10^{-162}$ \\
400 &2.65999780$\cdot 10^{-401}$ &2.65999756$\cdot 10^{-401}$ &4.60105$\cdot 10^{-281}$ &4.60102$\cdot 10^{-281}$ &6.48866$\cdot 10^{-211}$ &6.48105$\cdot 10^{-211}$ \\
450 &3.10906376$\cdot 10^{-486}$ &3.10906348$\cdot 10^{-486}$ &6.054876$\cdot 10^{-351}$ &6.054838$\cdot 10^{-351}$ &5.4446$\cdot 10^{-272}$ &5.4382$\cdot 10^{-272}$ \\
500 &2.5924940$\cdot 10^{-601}$ &2.5924937$\cdot 10^{-601}$ &5.68451$\cdot 10^{-451}$ &5.68447$\cdot 10^{-451}$ &3.2592$\cdot 10^{-363}$ & 3.2554$\cdot 10^{-363}$ \\
\hline
{$N=1000$}     &        & &        & &       &         \\
 $k$   &   $\rho(k)$         &$\tilde{\rho(k)}$ & $\rho(k)$      &$\tilde{\rho(k)}$& $\rho(k)$     & $\tilde{\rho(k)}$       \\
\hline
0   &0.750415              &0.750415              &0.501485           &0.501485           & 0.255005        & 0.254963 \\
100 &2.09542$\cdot 10^{-65}$      &2.09542$\cdot 10^{-65}$      &1.77512$\cdot 10^{-35}$  &1.77512$\cdot 10^{-35}$  &3.66982$\cdot 10^{-18}$ & 3.66922$\cdot 10^{-18}$ \\
200 &9.60904$\cdot 10^{-139}$    &9.60904$\cdot 10^{-139}$     &1.0319$\cdot 10^{-78}$    &1.03189$\cdot 10^{-78}$  &8.67317$\cdot 10^{-44}$ & 8.67175$\cdot 10^{-44}$ \\
300 &3.8264$\cdot 10^{-221}$      &3.8264$\cdot 10^{-221}$      &5.2089$\cdot 10^{-131}$   &5.20889$\cdot 10^{-131}$  &1.77997$\cdot 10^{-78}$ & 1.77968$\cdot 10^{-78}$ \\
400 &4.1605580$\cdot 10^{-313}$  &4.1605579$\cdot 10^{-313}$   &7.1797$\cdot 10^{-193}$   &7.17969$\cdot 10^{-193}$  &9.97471$\cdot 10^{-123}$ & 9.97308$\cdot 10^{-123}$ \\
500 &1.931348071$\cdot 10^{-415}$ &1.931348049$\cdot 10^{-415}$  &4.22488$\cdot 10^{-265}$  &4.22488$\cdot 10^{-265}$  &2.38635$\cdot 10^{-177}$ & 2.38596$\cdot 10^{-177}$ \\
600 &2.090294344$\cdot 10^{-529}$ &2.090294320$\cdot 10^{-529}$ &5.7964397$\cdot 10^{-349}$ &5.7964350$\cdot 10^{-349}$ &1.33109$\cdot 10^{-243}$ & 1.33087$\cdot 10^{-243}$ \\
700 &4.78531360$\cdot 10^{-657}$  &4.78531355$\cdot 10^{-657}$  &1.6821466$\cdot 10^{-446}$ &1.6821452$\cdot 10^{-446}$ &1.57049$\cdot 10^{-323}$ & 1.57023$\cdot 10^{-323}$ \\
800 &5.65614342$\cdot 10^{-802}$  &5.65614336$\cdot 10^{-802}$  &2.5204229$\cdot 10^{-561}$ &2.5204209$\cdot 10^{-561}$ &9.5668$\cdot 10^{-421}$ & 9.5653$\cdot 10^{-421}$ \\
900 &5.62059564$\cdot 10^{-972}$  &5.62059558$\cdot 10^{-972}$  &3.1749356$\cdot 10^{-701}$ &3.1749330$\cdot 10^{-701}$ &4.8995$\cdot 10^{-543}$ & 4.8987$\cdot 10^{-543}$ \\
1000&3.191157888$\cdot 10^{-1203}$ &3.191157852$\cdot 10^{-1203}$ &2.2850749$\cdot 10^{-902}$ &2.2850730$\cdot 10^{-902}$ &1.43367$\cdot 10^{-726}$ & 1.43343$\cdot 10^{-726}$ \\
\hline
\end{tabular}
}
\caption{Comparisons between the closed form of $\rho(k)$ given in (\ref{plimit_lambda_neq_mu_new}), 
with its approximation obtained by means of (\ref{approx_g_1}), 
with three choices of $N$,  three choices of $\varrho$, and various choices of $k$.} 
\end{center}
\end{table}
\par
We conclude this section by investigating the (Shannon) entropy of the system in the steady state, i.e.
$$
 H({\cal N})=\mathbb E[- \log\rho( {\cal N})]
 =-\sum_{k=0}^N \rho(k) \ln \rho(k),
$$
where $\rho(k)$ is given in (\ref{plimit_lambda_neq_mu_new}). As well known, it is a measure of the amount 
of information provided by ${\cal N}$.  Figure \ref{fig:Entropy}  presents the plot of $H({\cal N})$ 
as a function of $\varrho=\lambda/\mu$, for some choices of $N$. It is clear that $H({\cal N})$  is increasing in $N$. 
Moreover, we see that $H({\cal N})$ is unimodal in $\varrho$. The maxima 
$m={\rm argmax}_{\varrho>0} H({\cal N})$ are reported  in Table 2, where it is shown that $m$ is not monotonic in $N$. 
The considered cases show that the entropy of the system in the steady state reaches the maximum when 
$\lambda$ is close to the double of $\mu$, depending on $N$. 

%%%%%%%%%%%%%%%%%%%%%%%%%%%%%% Fig-entropia
\begin{figure}[t]
\centering
\includegraphics[width=8.cm]{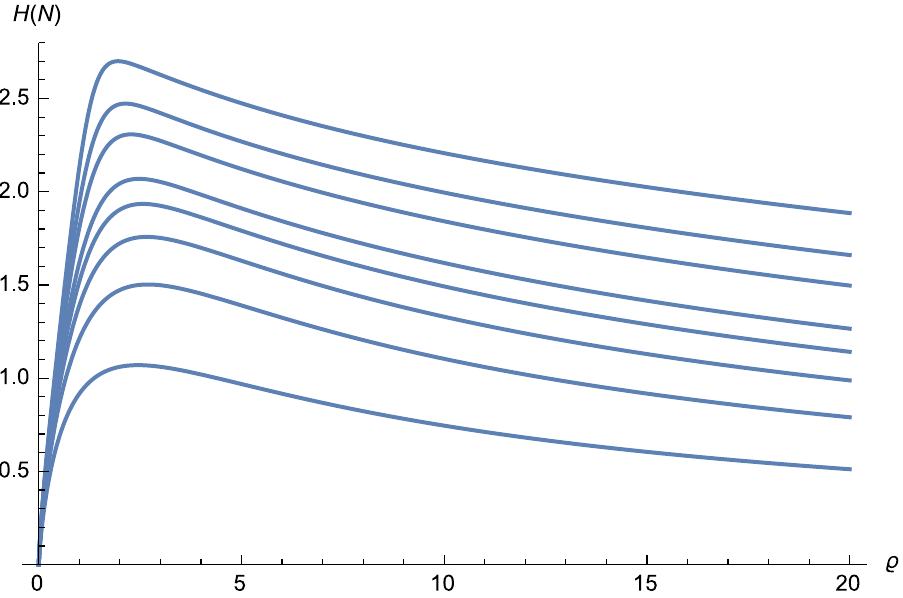}	
\caption{Entropy $H({\cal N})$ as a function of $\rho:=\lambda/\mu$, for $N= 2$, $4$, $6$, $8$, $10$, $15$, $20$, $30$ 
(from bottom to top).}
\label{fig:Entropy}
\end{figure}
%%%%%%%%%%%%%%%%%%%%%%%%%%%%%%
\begin{table}[t] 
\begin{center}
\begin{tabular}{|r|r|r|r| }
\hline 
$N$ & $m$ & $N$ & $m$   \\
\hline
2 & 2.45 & 10 & 2.47    \\
4 & 2.69 & 15  & 2.28   \\ 
6 & 2.66 & 20  & 2.14   \\  
8 & 2.57 & 30  & 1.95   \\
\hline
\end{tabular}
\caption{ Maximum of the entropy $H({\cal N})$ for the same cases shown in Figure \ref{fig:Entropy}.} 
\end{center}
\end{table}
%
% ==============================================================
\section{The diffusion approximation}\label{section:approximation}
% ==============================================================
Diffusion processes are largely adopted in the literature to model the dynamics of randomly fluctuating 
systems, and for the mean-field description of interacting particle systems and multi-agents modeling. 
In particular, the Ornstein-Uhlenbeck process is often used as it provides a fruitful compromise between the 
need to describe the dynamics of phenomena subject to fluctuations in the presence of an equilibrium point and the 
opportunity to have closed-form expressions of interest in applications, such as transition density and first-passage-time 
density through the equilibrium point. For instance, the recent papers by Ascione et al.\ \cite{Ascione}, 
Hongler and Filliger \cite{Hongler2019} and Ratanov \cite{Ratanov} deal with suitable generalizations 
of the Ornstein-Uhlenbeck process. In various contexts, such as queueing and mathematical neurobiology, 
generalized Ornstein-Uhlenbeck processes arise trough a scaling of continuous-time processes on a discrete state space. 
\par
Along this line, in this section we construct a diffusion approximation for the process $\{({\cal N}(t),{\cal L}(t)), t\geq 0\}$ 
that leads to an Ornstein-Uhlenbeck process on the spider. Before adopting a scaling procedure, we perform a different 
parameterization of the model studied in Section \ref{Section:model} by setting 
\begin{equation}
 \lambda=\frac{\alpha}{2}+\frac{\gamma}{2}\epsilon,
 \qquad
 \mu=\frac{\alpha}{2}-\frac{\gamma}{2}\epsilon,
 \qquad \hbox{for \ }\alpha>0, \ \epsilon>0, \ |\gamma|<\frac{\alpha}{\epsilon}.
 \label{eq:parametri}
\end{equation}
Note that $\epsilon$ plays a crucial role in the approximating procedure indicated below, where $\epsilon\to 0^+$.
\par
For all $t>0$, consider the position ${\cal N}^*_{\epsilon}(t)={\cal N}(t)\,\epsilon$, so that 
$\{({\cal N}^*_{\epsilon}(t), {\cal L}(t));\;t\geq 0\}$ 
is a continuous-time stochastic process having state space
$S_{0,\epsilon}^*=\{{ 0}\}\cup\left(\textbf{N}_{\epsilon}\times D\right)$, 
where $\textbf{N}_{\epsilon}=\{\epsilon,2 \epsilon,\ldots,N \epsilon\}$. 
Let $l_0 \in D$; recalling (\ref{eq:probpkjt}), the transient probabilities of the scaled process, 
for $\epsilon>0$, $t\geq 0$, $k\in \textbf{N}$, and $j,l \in D$, are given by 
\begin{equation}
\begin{split}
 & p^*_{\epsilon}(0,l,t)  
 := {\mathbb P}\left\{({\cal N}^*_{\epsilon}(t),{\cal L}(t))={ 0}, {\cal J}(t)=l \,|\,({\cal N}^*_{\epsilon}(0),{\cal L}(0))=0,{\cal J}(0)=l_0\right\}, 
 \\
 & p^*_{\epsilon}(k,j,t)  
 :=  {\mathbb P}\left\{({\cal N}^*_{\epsilon}(t),{\cal L}(t))=(k \epsilon,j)\,|\,({\cal N}^*_{\epsilon}(0),{\cal L}(0))=0,{\cal J}(0)=l_0\right\}.  
\end{split}
\label{eq:prob*}
\end{equation}
Since ${\cal N}^*_{\epsilon}(t)={\cal N}(t)\,\epsilon$, we have  
$p^*_{\epsilon}(0,l,t)=p(0,l,t)$ and $p^*_{\epsilon}(k,j,t)=p(k,j,t)$. 
In the limit as $\epsilon\to 0^+$, the scaled process is shown to converge weakly to a diffusion process 
${\cal X}:=\{(X(t),{\cal L}(t));\;t\geq 0\}$, whose state space is the spider, i.e.\ the star graph 
$S_{\cal X}:=\{0\}\cup\left(\mathbb{R}^+\times D\right)$. When $\epsilon$ tends to $0$, then the 
probabilities $p^*_{\epsilon}(0,l,t)$ and $p^*_{\epsilon}(k,j,t)$ given in (\ref{eq:prob*}) correspond respectively to 
\begin{equation*}
\begin{split}
 & {\mathbb P}\{0\leq X(t)<\epsilon, {\cal L}(t)=0, {\cal J}(t)=l \,|\,(X(0),{\cal L}(0))=0,{\cal J}(0)=l_0\}
 =:f(0,l,t)\,\epsilon+o(\epsilon), \\
 &  {\mathbb P}\{x\leq X(t)<x+\epsilon, {\cal L}(t)=j\,|\,(X(0),{\cal L}(0))=0,{\cal J}(0)=l_0\}
 =:f(x,j,t)\,\epsilon+o(\epsilon),   
\end{split}
\end{equation*}
for $t\geq 0$, $x=k\epsilon \in \mathbb{R}^+$, and $j,l \in D$. Hence, $f(0,l,t)$  and  $f(x,j,t)$ denote the probability density 
of the process ${\cal X}$ at time $t$ in the state 0 and in the state $x$ along the ray $S_j$, respectively. 
Moreover, the initial conditions (\ref{probiniz1}) and (\ref{probiniz2}) thus correspond   to 
$$
 f(x,j,0)=\delta(x)\delta_{j, l_0}, \qquad x\in \{0\}\cup \mathbb{R}^+, \;\; j\in D,
$$
where $\delta(x)$ is the delta-Dirac function. 
\par
We are now able to obtain the equations satisfied by the probability density of the diffusion process ${\cal X}$. 
\begin{proposition}\label{prop:diffj}
Under the limit conditions 
\begin{equation}
 \epsilon\rightarrow0^+, \qquad 
 N\rightarrow +\infty, \qquad 
 N \epsilon=N_\epsilon \rightarrow+\infty, \qquad 
 N \epsilon^2=N_\epsilon \epsilon \rightarrow \nu>0, 
 \label{eq:limcond}
\end{equation}
for $x\in \mathbb{R}^+$, $t> 0$ and $j\in D$, the density $f(x,j,t)$ satisfies the 
following partial differential equation:
\begin{equation}
 {\partial\over\partial t}\;f(x,j,t)=-{\partial\over\partial x}\;
 \Bigl\{[-\alpha(x-\beta)]\,
 f(x,j,t)\Bigr\}
 +{1\over 2}\,\sigma^2{\partial^2\over\partial x^2}f(x,j,t),
 \label{eq:equdiff}
\end{equation}
with boundary conditions
\begin{equation}
  \sum_{l \in D} \left\{\left.\alpha \beta f(0,l,t)-\frac{\sigma^2}{2}\frac{\partial}{\partial x}f(x,l,t)\right|_{x=0} \right\}=0,
  \label{eq:rifless}
\end{equation}
\begin{equation}
f(0,j,t)=\sum_{l \in D} c_{l,j} f(0,l,t),
\qquad j\in D,
\label{cond_approx_diff1}
\end{equation}
\begin{equation}
\lim_{x \rightarrow +\infty} f(x,j,t)=0,\qquad \forall j \in D,
\label{cond_approx_diff2}
\end{equation}
where, for $\nu>0$, 
\begin{equation}
 \sigma^2=\alpha \nu >0, \qquad \beta=\frac{\gamma \nu}{\alpha} \in \mathbb R.
  \label{eq:parsigbet}
\end{equation}
\end{proposition}
\begin{proof}
Since $p(k,j,t)=p^*_{\epsilon}(k,j,t)\approx f(k\epsilon,j,t)\,\epsilon$, for $\epsilon$ close to $0$,   
in analogy with the second equations of system (\ref{eq:system}), for $x=k\epsilon$ with $k=2,3\ldots N-1$, $j\in D$ and $t\geq 0$ we have
\begin{eqnarray}
&&  \hspace{-0.8cm}
%%%%%%%%%%%%%%%%%%%%%%%%%1
{\partial \over \partial t}\;\sum_{l\in D}f(0,j, t) \cdot \epsilon=\mu \frac{N_\epsilon+\epsilon}{\epsilon}\sum_{l\in D} f(\epsilon,l,t)\cdot \epsilon-\lambda \frac{N_\epsilon}{\epsilon}\,\sum_{l\in D}f(0,l,t) \cdot \epsilon,\label{eq:system_diff_approx1}
\\
&&  \hspace{-0.8cm}
%%%%%%%%%%%%%%%%%%%%%%%%%%%%%2
{\partial \over \partial t}\;f(\epsilon,j, t)\cdot \epsilon =\mu \frac{N_\epsilon+2 \epsilon}{\epsilon}\,f(2\epsilon,j,t)\cdot \epsilon+\sum_{l \in D}c_{l,j}\, \lambda \frac{N_\epsilon}{\epsilon} f(0,l,t)\cdot \epsilon \nonumber \\
 &&  \hspace{1.4cm}  - \frac{(\lambda+\mu)N_\epsilon-(\lambda-\mu)\epsilon}{\epsilon}f(x,j,t)\cdot \epsilon,\label{eq:system_diff_approx2}
\\
&&  \hspace{-0.8cm}
%%%%%%%%%%%%%%%%%%%%%%%%%%%%%%%%%3
\frac{\partial}{\partial t}f(x,j,t)\cdot \epsilon=\mu \frac{N_\epsilon+x+\epsilon}{\epsilon}f(x+\epsilon,j,t)\cdot \epsilon+\lambda \frac{N_\epsilon-x+\epsilon}{\epsilon}f(x-\epsilon,j,t)\cdot \epsilon
 \nonumber\\
 &&  \hspace{1.4cm}  - \frac{(\lambda+\mu)N_\epsilon-(\lambda-\mu)x}{\epsilon}f(x,j,t)\cdot \epsilon,\label{eq:system_diff_approx3}
\\
&&  \hspace{-0.8cm}
%%%%%%%%%%%%%%%%%%%%%%%%%%%%%%%%%%54
{\partial \over \partial t}\;f(N_\epsilon,j, t)\cdot \epsilon=\lambda\,f(N_\epsilon-\epsilon,j, t)\cdot \epsilon-\mu \frac{2N_\epsilon}{\epsilon}\,f(N_\epsilon,j, t) \cdot \epsilon.\label{eq:system_diff_approx4}
\end{eqnarray}
%%%%%%%%%%%%%%%555
where $N_\epsilon=N\epsilon$. 
Expanding $f$ as Taylor series, from equation (\ref{eq:system_diff_approx3}) we obtain 
\begin{eqnarray}
&& \hspace{-1cm}
\frac{\partial}{\partial t}f(x,j,t)=(\mu+\lambda)f(x,j,t)+\left[(\mu-\lambda)N_\epsilon+(\mu+\lambda)x+(\mu-\lambda)\epsilon\right]\frac{\partial}{\partial x}f(x,j,t)
 \nonumber \\
 &&  \hspace{1.4cm}  + \frac{\epsilon}{2}\left[(\mu+\lambda)N_\epsilon+(\mu-\lambda)x+(\mu+\lambda)\epsilon\right]\frac{\partial^2}{\partial x^2}f(x,j,t)+o(\epsilon^2).
 %+\left(2{\tilde \mu\over \epsilon} + \tilde \beta\right)k\right]\Delta t\right\}
 %+o(\Delta t),
 %\qquad  k\in\mathbb{N}^+.
\nonumber
\end{eqnarray}
Due to (\ref{eq:parametri}) one has $\lambda-\mu=\gamma \epsilon$ and $\lambda+\mu=\alpha$, so that 
\begin{eqnarray}
&& \hspace{-1cm}
\frac{\partial}{\partial t}f(x,j,t)=\alpha f(x,j,t)+\left(-\gamma \epsilon N_\epsilon+\alpha x-\gamma \epsilon^2\right)
\frac{\partial}{\partial x}f(x,j,t)
 \nonumber \\
 &&  \hspace{1.4cm}  + \frac{\epsilon}{2}\left(\alpha N_\epsilon-\gamma \epsilon x+\alpha \epsilon\right)
 \frac{\partial^2}{\partial x^2}f(x,j,t)+o(\epsilon^2).
 %+\left(2{\tilde \mu\over \epsilon} + \tilde \beta\right)k\right]\Delta t\right\}
 %+o(\Delta t),
 %\qquad  k\in\mathbb{N}^+.
\nonumber
\end{eqnarray}
Making use of the limit conditions (\ref{eq:limcond}), as $\epsilon\rightarrow0^+$ we get
$$
\frac{\partial}{\partial t}f(x,j,t)=\alpha f(x,j,t)+\left(-\gamma \nu+\alpha x\right)
\frac{\partial}{\partial x}f(x,j,t)+\frac{\alpha \nu}{2} \frac{\partial^2}{\partial x^2}f(x,j,t),
$$
that coincides with (\ref{eq:equdiff}) thanks to positions (\ref{eq:parsigbet}). 
Similarly, Eq.\ (\ref{eq:system_diff_approx1}) yields
$$
 {\partial \over \partial t} \sum_{l\in D}f(0,j, t) \, \epsilon
 =\sum_{l\in D}\left\{\left[N_\epsilon (\mu-\lambda)+\mu \epsilon\right]f(0,l,t)+\left(\mu \epsilon^2+\mu N_\epsilon \epsilon\right) \frac{\partial}{\partial x}f(x,l,t)\Big |_{x=0}\right\},
$$
and thus for $\epsilon\rightarrow 0^+$, we come to condition (\ref{eq:rifless}). 
Finally, following an analogous procedure, from (\ref{eq:system_diff_approx2}) and (\ref{eq:system_diff_approx4}) we 
obtain the relations (\ref{cond_approx_diff1}) and (\ref{cond_approx_diff2}), respectively.
\end{proof}
\par
From Proposition \ref{prop:diffj}, it is clear that the considered scaling procedure leads to a diffusion process 
that follows Ornstein-Uhlenbeck dynamics along the semi-infinite rays of the star graph. 
The corresponding drift and infinitesimal variance are given respectively by 
\begin{equation}
A_1(x)=-\alpha(x-\beta), \qquad A_2(x)=\sigma^2,  
\qquad x\in\mathbb R^+,
\label{drift_var}
\end{equation}
with $\alpha >0$, $\beta \in \mathbb R$ and $\sigma>0$. 
We point out that  (\ref{eq:rifless}) represents the reflection condition in the state $0$. Moreover, recalling that 
$C=(c_{l,j})_{l,j \in D}$ is a stochastic matrix, the relation (\ref{cond_approx_diff1}) expresses the switching 
mechanism in the origin of the state space. Finally, (\ref{cond_approx_diff2}) is a regularity condition on the endpoint $+\infty$.
\begin{remark}
Equation (\ref{cond_approx_diff1}) is equivalent to  
\begin{equation}
 \sum_{l \neq j}c_{l,j}  f(0,l,t)=\sum_{l \neq j} c_{j,l} f(0,j,t),
 \qquad   \hbox{$\forall \;t>0$ and  $j\in D$.}
  \label{eq:relrifless}
\end{equation}
This relation expresses a conservation of probability in the state 0. Namely, the left-hand-side of (\ref{eq:relrifless}) 
expresses the intensity that the process enters the line $S_j$ at time $t$ arriving form any different line, 
whereas the right-hand-side of (\ref{eq:relrifless}) gives the intensity that the process exits from the line 
$S_j$ at time $t$ moving toward any different line, so that Eq.\ (\ref{eq:relrifless}) provides an identity 
between the entrance and exit probability current for the line $S_j$ trough the state 0. 
\end{remark}
\par
Let us now introduce the density
\begin{equation}
 h(x,t):=\sum_{j=1}^d f(x,j,t),\qquad x\in \mathbb{R}^+,\quad t\geq 0.
  \label{eq:definhxt}
\end{equation}
\begin{proposition}
For $x\in \mathbb{R}^+$ and $t\geq 0$, the transition density (\ref{eq:definhxt}) satisfies the following
differential equation:
\begin{equation}
 {\partial\over\partial t}\;h(x,t)=-{\partial\over\partial x}\;
 \Bigl\{-\alpha(x-\beta)\,
 h(x,t)\Bigr\}
 +{1\over 2}\,\sigma^2\,{\partial^2\over\partial x^2}h(x,t),
 \label{eq:equdiffsomma}
\end{equation}
with conditions
\begin{equation}
  \left.\alpha \beta h(0,t)-\frac{\sigma^2}{2}\frac{\partial}{\partial x}h(x,t)\right|_{x=0}=0,\qquad \lim_{x\rightarrow +\infty} h(x,t)=0.
  \label{eq:equdiffbound}
\end{equation}
%
%and Dirac-delta initial condition
%%
%\begin{equation}
 %\lim_{t\to 0^+}h(x,t)=\delta(x).
  %\label{eq:initcond}
%\end{equation}
%
\end{proposition}
\begin{proof}
The proof of Eqs.\ (\ref{eq:equdiffsomma}) and (\ref{eq:equdiffbound})
follows immediately from Proposition \ref{prop:diffj}, and recalling position (\ref{eq:definhxt}).
\end{proof}
Note that Eq.\ $(\ref{eq:equdiffsomma})$ is the Fokker-Planck equation for a Ornstein-Uhlenbeck diffusion 
process on $\mathbb{R}^+$ with drift and infinitesimal variance given in (\ref{drift_var}), where (\ref{eq:equdiffbound}) 
gives the reflection condition at the  regular endpoint $x=0$  and the regularity condition for 
the nonattracting-natural  endpoint $x=+\infty$. 
We remark that in general there is no explicit form for the corresponding transition density. However, 
if $\beta=0$ then the transition density can be expressed as a combination of two transition densities of the 
unrestricted process (for details see, for instance, Appendix A of Giorno et al.\ \cite{Giorno2012}). 
% =====================
\subsection{Asymptotic behavior} 
% =====================
\par
In order to investigate the steady state of the approximating diffusion process, we denote by 
$(X,{\cal L})$ the two-dimensional random variable describing the asymptotic behavior of ${\cal X}$. 
The support of $(X,{\cal L})$ is the spider, i.e.\ $\mathbb R_+\cup\{0\}\times D$. 
Hereafter we determine the probability law of $(X,{\cal L})$. Specifically, we show that $X$ and ${\cal L}$ 
are independent, where $X$ has a truncated normal distribution and ${\cal L}$ is distributed as the stationary 
distribution of the Markov chain characterized by the transition matrix $C$ treated in 
(\ref{c_lj}) and (\ref{eq:tassi}). To this aim, the (sub)density related to the $j$-th ray of the spider is denoted as 
\begin{equation}
w(x,j):=\lim_{t \rightarrow +\infty} f(x,j,t),\qquad x \in \mathbb{R}^+\cup\{0\},\; j \in D.
\label{asymp_dens_j}
\end{equation}
Moreover,  the probability density function of $X$ is  
\begin{equation}
 w(x)=\sum_{j\in D} w(x,j), \qquad x \in \mathbb{R}^+\cup\{0\},
\label{defasymp_dens_j}
\end{equation}
whereas ${\bf \pi}=(\pi_1, \ldots, \pi_d)$ is the vector of the 
stationary probabilities of the Markov chain having transition matrix $C$. 
\begin{proposition}
For all $\alpha >0$, $\beta \in \mathbb R$ and $\sigma>0$, the asymptotic density (\ref{asymp_dens_j}) 
satisfies
$$
w(x,j)=w(x)\,\pi_j, \qquad 
 \forall x \in \mathbb{R}^+\cup\{0\}, \;\; j \in D,
$$
with
\begin{equation}
 w(x)=\frac{1}{Q}\, \exp\left\{- \frac{2\,\alpha\,x}{\sigma^2}\left(\frac{x}{2}-\beta\right)\right\},
 \qquad x \in \mathbb{R}^+\cup\{0\},
\label{express_asymp_dens}
\end{equation}
where $Q$ is the normalizing constant given by  
\begin{equation}
Q=\frac{\sigma \sqrt{\pi}}{2\sqrt{\alpha}}
\left(1+{\rm Erf}\left(\frac{\sqrt{\alpha}}{\sigma}\,\beta\right)\right) 
\exp\left\{ \frac{\alpha\,\beta^2}{\sigma^2}\right\},
\label{Q_asymp}
\end{equation}
and ${\rm Erf}(x)=\frac{2}{\sqrt{\pi}}\int_0^x e^{-t^2}dt$ is the error function.
\end{proposition}
\begin{proof}
As $t \rightarrow +\infty$, Eq.\ (\ref{eq:equdiff})  becomes 
%
%\begin{equation}
$$
0=-{\partial\over\partial x}\;
 \Bigl\{-\alpha(x-\beta)\,
 w(x,j)\Bigr\}
 +{1\over 2}\,\sigma^2\,{\partial^2\over\partial x^2}w(x,j), 
 %\label{eq:equdiff_asymp}
%\end{equation}
$$
whose solution for $x\in \mathbb{R}^+$ and $j\in D$ is given by 
\begin{equation}
w(x,j)=w(0,j)\,\exp\left\{- \frac{2\,\alpha\,x}{\sigma^2}\left(\frac{x}{2}- \beta\right)\right\}.
\label{sol_w1}
\end{equation}
By letting $t\to + \infty$ in (\ref{cond_approx_diff1}), due to (\ref{asymp_dens_j}) one has
\begin{equation}
 w(0,j)=\sum_{l \in D} c_{l,j}\,w(0,l),\qquad j\in D.
 \label{eq:woj}
\end{equation}
Hence, 
$$
w(0,j)=\frac{\pi_j}{Q},\qquad j\in D,
$$
where  $\left(w(0,j); j\in D\right)\equiv  {\bf \pi}$  is the vector of the stationary probabilities of the Markov chain 
characterized by the transition matrix $C$ treated in  (\ref{c_lj}) and (\ref{eq:tassi}). From (\ref{sol_w1}) and (\ref{defasymp_dens_j})
one thus obtains $w(x,j)=w(x)\,\pi_j$, with $w(x)$ given in (\ref{express_asymp_dens}).   Finally,
by integrating on $x$ and summing on all $j \in D$, from  (\ref{express_asymp_dens}) one has (\ref{Q_asymp}) 
after a straightforward calculation.
\end{proof}
\begin{remark}
From Eq.\  (\ref{eq:woj}) it is not hard to see that (cf.\   (\ref{eq:relrifless}))
$$
 \sum_{l \neq j}c_{l,j}  w(0,l)=\sum_{l \neq j} c_{j,l} w(0,j),
 \qquad   \hbox{$\forall \; j\in D$.}
$$
\end{remark}
\begin{remark}
The asymptotic density (\ref{express_asymp_dens}) is unimodal, with mode in the equilibrium point $x=\beta$. 
When $\lambda=\mu$, from (\ref{eq:parametri}) and (\ref{eq:parsigbet}) we have  $\beta=0$. In this case, 
we can compare the density $w(x)$ with the asymptotic density of the Ehrenfest model. 
Indeed, it can be easily proven that 
$$
 w(x)=2\, \widetilde{W}(x),\qquad x \in \mathbb{R}^+\cup\{0\},
$$
where (see, for instance, Eq.\ (31) of Dharmaraja et al.\ \cite{Dharmaraja2015}) $\widetilde{W}(x)$ is the 
steady-state density of the diffusion approximation of the discrete-time Ehrenfest model. Clearly, this result 
is in agreement with the comparison given in Remark \ref{remark:3} for the discrete models.
\end{remark}
\par
Making use of Eqs.\ (\ref{express_asymp_dens}) and (\ref{Q_asymp}), we are now able to 
recover the asymptotic mean and variance of $X$.
\begin{proposition}
If $\beta \neq 0$, the mean and the variance of $X$ are given respectively by 
$$ 
 \mathbb E\left [X\right]
 =\beta \left(1 +\frac{1}{\sqrt{\pi}}\, \frac{\sigma} {\sqrt{\alpha}\,\beta}
 \frac{\exp\left\{- \frac{\alpha\,\beta^2}{\sigma^2}\right\}}
 {1+{\rm Erf}\left(\frac{\sqrt{\alpha}\,\beta}{\sigma}\right)}\right),
$$
and 
$$ 
 Var\left [X\right]
 =\frac{\sigma^2}{2 \alpha}
 \left( 1-  \frac{2 }{\pi}
 \frac{\exp\left\{- 2\frac{ \alpha\,\beta^2}{\sigma^2}\right\}}
 {\left(1+{\rm Erf}\left(\frac{\sqrt{\alpha}\,\beta}{\sigma}\right)\right)^2}
 - \frac{2 \beta \sqrt{\alpha}}{\sqrt{\pi} \sigma } \frac{ \exp\left\{- \frac{\alpha\,\beta^2}{\sigma^2}\right\}}
  {\left(1+{\rm Erf}\left(\frac{\sqrt{\alpha}\,\beta}{\sigma}\right)\right)}
 \right),
$$
whereas if $\beta = 0$ then 
\begin{equation}
 \mathbb E\left[X\right] =\frac{\sigma}{\sqrt{\pi\,\alpha }}, 
 \qquad 
 Var\left[X\right] =\frac{\sigma^2}{2\alpha}\left(1-\frac{2}{ \pi}\right).
 \label{eq:meanvarb0}
\end{equation}
\end{proposition}
\par
With reference to the parameters  $\alpha >0$, $\beta \in \mathbb R$ and $\sigma>0$, we point out the following. 
\\
(i) The mean of $X$ is increasing in $\beta$, with  $\mathbb E \left[X\right]\to 0$ for $\beta\to -\infty$, and 
$\frac{\mathbb E \left[X\right]}{\beta}\to 1$ for $\beta\to \infty$.  
Furthermore, $\mathbb E\left[X\right]$ is decreasing with respect to $\alpha/\sigma^2$, such that 
$\mathbb E\left[X\right]\to \infty$ if $\alpha/\sigma^2\to 0^+$. Moreover, 
if $\alpha/\sigma^2\to \infty$ then $\mathbb E\left[X\right]\to \beta$ if $\beta>0$, and 
$\mathbb E\left[X\right]\to 0$ if $\beta\leq 0$.
\\
(ii) The variance of $X$ is   increasing in $\beta$, with $Var\left [X\right]\to 0$  for $\beta\to - \infty$, and  
$Var\left [X\right]\to \frac{\sigma^2}{2 \alpha}$  for $\beta\to \infty$. 
Moreover, $Var\left [X\right]$ is decreasing with respect to $\alpha/\sigma^2$, such that 
$Var\left [X\right]\to \infty$ when $\alpha/\sigma^2\to 0^+$, and 
$Var\left [X\right]\to 0$ when $\alpha/\sigma^2\to \infty$, for all $\beta\in \mathbb R$. 
\begin{example}\rm
Recalling (\ref{c_lj}) and (\ref{eq:tassi}), let us now consider some examples of the matrix $C$,  
which regulates the switching mechanism for the particle types, 
and the corresponding  vector $\vec \pi=(\pi_1, \ldots, \pi_d)$ of the stationary probabilities.
\begin{enumerate}
\item %example 1
The transitions from line $l$ to line $j$ occur uniformly:
$$
 c_{l,j}=\frac{1}{d},\qquad \forall \, l,j \in D.
$$
\item %example 2
The transitions occur uniformly on any line different from the previous one:  
%\begin{equation}
$$
c_{l,j}=\left\{
 \begin{array}{ll}
 \displaystyle \frac{1}{d-1}, & \quad l\neq j,\\[2mm]
 0, & \quad \hbox{otherwise} 
\end{array}
\right.
\qquad (\forall\, l,j \in D).
%\label{ex5}
%\end{equation}
$$
\item %example 3
The transitions occur cyclically clockwise:
$$
c_{l,j}=\left\{
 \begin{array}{ll}
 1, & \quad j=l+1, \\[1mm]
 0, & \quad \hbox{otherwise} 
\end{array}
\right.
\quad (l=1,2,\ldots,d-1), 
%\label{ex2_1}
\qquad 
c_{d,j}=\left\{
 \begin{array}{ll}
 1, & \quad j=1,  \\[1mm]
 0, & \quad \hbox{otherwise}. 
\end{array}
\right.
%\label{ex2_2}
%\end{equation}
$$
Under the assumptions of the first three cases, one obtains the stationary uniform distribution
$\vec \pi=\left(\frac{1}{d},\,\frac{1}{d}, \ldots,\frac{1}{d}\right)$.
\item %example 4
The transitions occur  sequentially, until line $d$ is reached:   
$$
 c_{l,j}=\left\{
 \begin{array}{ll}
 1, & \quad j=l+1, \\[1mm]
 0, & \quad\hbox{otherwise} 
\end{array}
\right.
\quad 
(l=1,2,\ldots,d-1),
\qquad 
c_{d,j}=\left\{
 \begin{array}{ll}
 1, & \quad j=d,\\[1mm]
 0, & \quad \hbox{otherwise} .
\end{array}
\right.
%\label{ex1_2}
%\end{equation}
$$
In this case, since $d$  is an absorbing line, the  stationary probability vector is 
$\vec \pi=(0,\,0, \ldots,0,\,1)$.
\item  %example 5
The transitions occur on adjacent lines, according to a random-walk scheme: 
%
%\begin{equation}
$$
c_{1,j}=\left\{
 \begin{array}{ll}
 1, & \; j=2, \\[1mm]
 0, & \; \hbox{otherwise} 
\end{array}
\right.
%\label{ex3_1}
%\end{equation}
\quad  
c_{l,j}=\left\{
 \begin{array}{ll}
 1-p, & \; j=l-1,\\[1mm]
p, & \; j=l+1, \\[1mm]
 0, & \; \hbox{otherwise} 
\end{array}
\right.
\;\;
(l=2,3,\ldots,d-1), 
\quad
c_{d,j}=\left\{
 \begin{array}{ll}
 1, & \; j=d-1,\\[1mm]
 0, & \; \hbox{otherwise}. 
\end{array}
\right.
%\label{ex3_3}
%\end{equation}
$$
If $p\neq \frac{1}{2}$, then the stationary vector $\vec \pi$ has components
$$
 \pi_1=\left(\frac{1}{p}-1\right)^{d-2} \pi_d,
 \quad 
 \pi_j=\frac{1}{1-p}\left(\frac{1}{p}-1\right)^{d-j}  \pi_d,\quad j=2,3,\ldots,d-1,
 \quad 
 \pi_d=\frac{(p-1)(2p-1)}{2p\left[p-1+\left(\frac{1}{p}-1\right)^{d}p\right]};
$$
%
%\begin{eqnarray}
%&&\pi_1=\left(\frac{1}{p}-1\right)^{d-2}\, \pi_d,\nonumber\\
%&&\pi_j=\frac{1}{1-p}\left(\frac{1}{p}-1\right)^{d-j}\, \pi_d,\qquad j=2,3,\ldots,d-1,\nonumber\\
%&&\pi_d=\frac{(p-1)(2p-1)}{2p\left[p-1+\left(\frac{1}{p}-1\right)^{d}p\right]};\nonumber
%\end{eqnarray}
%
on the other hand, if $p= \frac{1}{2}$, then the components of $\vec \pi$ are 
$$
 \pi_1=\pi_d=\frac{1}{2(d-1)},\qquad 
 \pi_j=\frac{1}{d-1},\quad j=2,3,\ldots,d-1.
$$
%
%\begin{eqnarray}
%&&\pi_1=\pi_d=\frac{1}{2(d-1)},\nonumber\\
%&&\pi_j=\frac{1}{d-1},\qquad j=2,3,\ldots,d-1.\nonumber\\
%\end{eqnarray}
%
\end{enumerate}
\end{example}
%
% ==============================================================
\subsection{Some comparisons}\label{section:comparisons}
% ==============================================================
Let us now discuss the goodness of the continuous approximation derived so far. 
Since the approximation is performed under the limit conditions (\ref{eq:limcond}), 
we expect that it improves as $\epsilon$ tends to 0 and as $N$ grows larger. 
\par
We first assess the correspondence between the stationary distributions of the Ehrenfest model 
and its continuous approximation. Hence, we refer to the stationary probabilities $\rho(k)$ introduced 
in (\ref{eq:rhok}) and to the probability density function $w(x)$ specified in (\ref{defasymp_dens_j}). 
By considering the case $\varrho =1$, i.e.\ $\lambda=\mu$ and thus $\beta=0$, due to the Stirling approximation one has 
\begin{equation}
  \frac{  {2N \choose N+k}}{{2N \choose N}+4^N}  \sim \frac{1}{\sqrt{\pi N}}
  \qquad \hbox{as $N\to \infty$,}
  \label{eq:stirl}
\end{equation}
and thus Eq.\ (\ref{plimit_lambda_eq_mu}) yields  
$$
 \rho(k) \sim \frac{2}{\sqrt{\pi N}}\qquad \hbox{as $N\to \infty$,}
$$
whereas Eq.\ (\ref{express_asymp_dens}) becomes   
$$
 w(k\epsilon)\,\epsilon  
 =\frac{2}{\sqrt{\pi N}}\,\exp{\left\{-\frac{k^2}{N}\right\}},
$$
so that we finally obtain, for any $\epsilon>0$ and $k\in \mathbb{N}_0$,
$$
 \rho(k) \sim  w(k\epsilon)\, \epsilon \qquad \hbox{as $N\to \infty$.}
$$
This confirms the agreement between the stationary distributions of the considered processes. 
See also Table 3, where the quantities of interest are shown for some choices of the parameters, 
together with the relative difference 
$$
 \Delta(k):=\frac{w(k\epsilon)\,\epsilon-\rho(k)}{\rho(k)},
$$
and according to the limiting procedure considered above. 
Again, the given values confirm that the approximation improves as $N\to \infty$.   
\begin{table}[t] 
\begin{center}
{\footnotesize
\begin{tabular}{|r|lll|lll|lll|}
\hline 
{} &  $N=5000$, & $\sigma^2=100$  & {}  & $N=10000$,  & $\sigma^2=200$ & {} &  $N=15000$,   & $\sigma^2=300$ & {}  \\
\hline 
$k$ & $w(k\epsilon)\,\epsilon$ & $\rho(k)$ & $\Delta(k)$ & 
$w(k\epsilon)\,\epsilon$ & $\rho(k)$& $\Delta(k)$ & 
$w(k\epsilon)\,\epsilon$ & $\rho(k)$& $\Delta(k)$  \\
\hline
$0$ & 0.0159577 &0.015831 & 0.00800385     & 0.0112838 &0.0112203 & 0.0056544    & 0.00921318 &0.00917085 & 0.00461492 \\
$1$ & 0.0159545 & 0.0158278  & 0.00800383     & 0.0112827 & 0.0112192  & 0.00565439    & 0.00921256 & 0.00917024  & 0.00461492 \\
$2$ & 0.0159449 &  0.0158183 &0.00800377     & 0.0112793 &  0.0112159  &0.00565438      & 0.00921072 & 0.00916841 &0.00461491 \\
$3$ & 0.015929& 0.0158025 &0.00800366     & 0.0112736& 0.0112103 &0.00565435        & 0.00920765 & 0.00916535 &0.0046149  \\
$4$ & 0.0159067& 0.0157804  &0.00800352     & 0.0112658& 0.0112024  &0.00565432      & 0.00920336 & 0.00916108 &0.0046149 \\
$5$ & 0.0158781 & 0.015752  & 0.00800334    & 0.0112556 & 0.0111923  & 0.00565427      & 0.00919783 & 0.00915558 &0.00461487 \\
$10$ & 0.0156417& 0.0155175 & 0.00800184    & 0.0111715& 0.0111087  & 0.00565389     & 0.009151196 & 0.00910992 &0.0046147 \\
$20$ & 0.0147308& 0.014614 &  0.007996     & 0.0108413& 0.0107804 &  0.00565241      & 0.00897074 & 0.00892954 &0.00461404  \\
$30$ & 0.013329 & 0.0132234 & 0.00798679    & 0.0103126 & 0.0102547 & 0.00565001     & 0.00867664 & 0.0086368 &0.00461295 \\
$40$ & 0.0115877& 0.011496 & 0.00797503    & 0.00961541& 0.00956142 & 0.00564678    & 0.00828104 & 0.00824302 &0.00461148 \\
$50$ & 0.00967883& 0.00960238 & 0.00796185   & 0.00878783& 0.00873852  & 0.00564287   & 0.00779879 & 0.007763 &0.00460965 \\
\hline
\end{tabular}
}
\caption{
For $\lambda=\mu=1$,  $\alpha=2$ and  $\beta=\gamma=0$ the quantities $w(k\epsilon)\,\epsilon$, $\rho(k)$ are $\Delta(k)$ 
are shown for $\epsilon =0.1$ and for various choices of  $k$ and $N, \sigma^2$ such that $ \sigma^2= \alpha N\epsilon ^2$.} 
\end{center}
\end{table}
\par
Moreover, the agreement is also confirmed by comparing the mean and the variance of 
${\cal N}$ and $X/\epsilon$. Indeed, if $\varrho =1$, and $\beta=0$, making use of (\ref{eq:stirl}) 
from Remark \ref{rem:EVCV} one has
\begin{equation}
 E\left[{\cal N}\right] \sim \frac{\sqrt N}{\sqrt\pi},
 \qquad 
 Var\left[{\cal N}\right] \sim  N \left(\frac{1}{2}-\frac{1}{ \pi}\right)
 - \frac{1}{2} \frac{ \sqrt N}{ \sqrt \pi}
 \qquad \hbox{as $N\to \infty$,}
  \label{eq:limEVN}
\end{equation}
so that recalling (\ref{eq:meanvarb0}) under the considered scaling one immediately has 
$$
 E\left[{\cal N}\right] \sim E\left[ \frac{X}{\epsilon}\right], 
 \qquad 
  Var\left[{\cal N}\right] \sim  Var\left[\frac{X}{\epsilon}\right], 
  \qquad \hbox{as $N\to \infty$.} 
$$ 
% ==============================================================
\section{Concluding remarks}\label{section:cremarks}
% ==============================================================
Nowadays many researchers are interested in the analysis of random motions on star graphs and related structures. 
Up to now various efforts have been devoted mainly to the cases of birth-death processes and Brownian diffusion on 
such domains. This contribution is among the first studies concerning birth-death processes with state-dependent 
rates and the approximating Ornstein-Uhlenbeck process over a spider. It is noteworthy that the present investigation 
leads to closed-form results for the transient analysis, at least in the case $\lambda=\mu$, and to the complete 
asymptotic analysis of the multi-type Ehrenfest model, as well as to a detailed study of the asymptotic behavior 
of the approximating Ornstein-Uhlenbeck process. 
\par
Possible future developments  can be oriented to the analysis of 
\\
(i) the first-passage-time problem for the 
considered processes through the origin of the spider or other fixed states, 
\\ 
(ii) suitable modifications of the stochastic system, such as after the inclusion of the possibility of instantaneous transitions as due to the 
effect of catastrophes occurring randomly in time, 
\\
(iii) extension to the multidimensional version, in which the various branches of the state-space 
can be occupied at the same time, 
\\
(iv) modification in the transition rates leading to a birth-death process with quadratic birth and death rates, 
similar as in Section 5 of Di Crescenzo et al.\ \cite{DiCrescenzo2021}, leading to a diffusion approximation expressed 
by a lognormal diffusion process. 
\par
Finally, we remark that the multi-type Ehrenfest model introduced in Section \ref{Section:model} 
can be modeled as a finite non homogeneous quasi-birth-death (QBD) process (see, for instance,  
the book  by Latouche and Ramaswami \cite{LatoucheRamaswami}). 
Such QBD process has a two-dimensional state space $\bigcup_{k=0}^N l(k)$, where $l(0)=\left\{(0,1),\,(0,2),\ldots,\,(0,d)\right\}$ and $l(k)=\left\{(k,1),\,(k,2),\ldots,\,(k,d)\right\}$ ($k=1,\,2,\ldots,N; d \in D$); the subset of the states $l(k)$ is called level $k$. In our context, the states $(0,j)$ ($j=1,\,2,\ldots,\,d$) of $l(0)$ correspond to the state $0$ (the origin of the graph), whereas the second element of the couple $j$ represents the last visited line. 
Hence, numerical techniques from matrix-analytic methods could therefore be applied to obtain e.g.\ the stationary distribution of the model. 
This approach allows also to construct suitable generalizations of the process. 
This can be the object of a further prosecution of the present investigation. 
%
%%%%%%%%%%%%%%%%%%%%%%%%%%%%%%%%%
\subsection*{Acknowledgements}
%%%%%%%%%%%%%%%%%%%%%%%%%%%%%%%%%
%
The  authors are members of the research group GNCS of INdAM (Istituto Nazionale di Alta Matematica). 
This research is partially supported by MIUR - PRIN 2017, project `Stochastic Models for Complex Systems', 
no.\ 2017JFFHSH.

%%%%%%%%%%%%%%%%%%%%%%%%%%%%%%%%%
\subsection*{Conflict of interest}
%%%%%%%%%%%%%%%%%%%%%%%%%%%%%%%%%
This work does not have any conflicts of interest.
%%%%%%%%%%%%%%%%%%%%%%%%%%%%%%%%%
%

%--------------------------------------------------------------------
%\section*{References}
%--------------------------------------------------------------------
%--------------------------------------------------------------------

%
%%%%%%%%%%%%%%%%%%%%%%%%%%
\appendix 
%%%%%%%%%%%%%%%%%%%%%%%%%%

\section{}

\bigskip\noindent
{\bf Proof of Lemma \ref{lemma_pol}}

\smallskip\noindent
Before providing the proof of Lemma \ref{lemma_pol}, we recall the following useful conditions 
about the  Gamma function, for $i \in \mathbb{N}$:
%%%%%%%%%%%%%
\begin{eqnarray}
&&\Gamma\left(\frac{1}{4}-2i\right)>0,\qquad \Gamma\left(\frac{3}{4}-2i\right)>0,\qquad\Gamma\left(\frac{9}{4}-2i\right)>0,\qquad\Gamma\left(\frac{11}{4}-2i\right)>0\label{Gamma_pos}\\
&&\Gamma\left(\frac{5}{4}-2i\right)<0,\qquad\Gamma\left(\frac{7}{4}-2i\right)<0.\label{Gamma_neg}
\end{eqnarray}
%%%%%%%%%%%%%%
With reference to (\ref{pol_P}), to show that $P(x)$ has $N$ distinct negative roots in addition to $0$, 
we deal with two cases: $N$ even and $N$ odd.
\\
%For the other $N$ solutions, we will apply the Intermediate Zero Theorem in $N$ different intervals, with negative extremes, to obtain the results of Lemma \ref{lemma_pol}.\\ \\
%%%%%%%%%%%%%%%%%%%%%%%%%%
%N even
%%%%%%%%%%%%%%%%%%%%%%%%%%
(i) Let $N$ be even, i.e.\ $N=2n$, with $n \in \mathbb{N}$. In this case we apply the Intermediate Zero Theorem 
to  the following intervals of negative numbers:
\begin{equation}
\big(-4(n+k)\mu-\mu,  -4(n+k) \big),\qquad k=1,2,\ldots,n,
\label{N_even_A}
\end{equation}
\begin{equation}
\big(-4(n-k+1)\mu-\mu,-4(n-k)\mu -\mu\big),\qquad k=1,2,\ldots,n.
\label{N_even_B}
\end{equation}
%
%
%\begin{equation}
%\left(-(4n+1)\mu-4 k\mu;-(4n+1)\mu-4 (k-1)\mu\;\right),\qquad k=1,2,\ldots,n,
%%\label{N_even_A}
%\end{equation}
%%
%\begin{equation}
%\left(-(4n+1)\mu+4  (k-1)\mu;-(4n+1)\mu+4  k\mu\;\right),\qquad k=1,2,\ldots,n.
%%\label{N_even_B}
%\end{equation}
%%
Evaluating $P(x)$ in the left-hand extreme of the interval (\ref{N_even_A}) we obtain:
$$
P\left(-4(n+k)\mu- \mu\right)
=-16^n \mu^{2n+1}(4n+4k+1)\left\{\frac{\Gamma\left(\frac{1}{4}-k+n\right)}{\Gamma\left(\frac{1}{4}-k-n\right)}+\frac{\Gamma\left(\frac{3}{4}-k+n\right)}{\Gamma\left(\frac{3}{4}-k-n\right)}\right\}.
$$
Note that $-16^n \mu^{2n+1}(4n+4k+1)<0$, with $\Gamma\left(\frac{1}{4}-k+n\right)>0$ and $\Gamma\left(\frac{3}{4}-k+n\right)>0$.
Moreover, discussing various cases on the basis of the parity of $n$ and $k$   
it can be shown that  $\Gamma\left(\frac{1}{4}-k-n\right)\Gamma\left(\frac{3}{4}-k-n\right)>0$. 
Consequently, the polynomial $P(x)$ takes opposite signs in the interval's extremes, 
so that it has at least one root in each interval (\ref{N_even_A}).
Similarly, the same result can be shown for the interval (\ref{N_even_B}) since 
$$
P\left(-4(n-k)\mu-\mu\right)=-16^n \mu^{2n+1}(4n-4k+1)\left\{\frac{\Gamma\left(\frac{1}{4}+k+n\right)}{\Gamma\left(\frac{1}{4}+k-n\right)}+\frac{\Gamma\left(\frac{3}{4}+k+n\right)}{\Gamma\left(\frac{3}{4}+k-n\right)}\right\}.
$$
In conclusion, for $N$ even, the polynomial $P(x)$ defined in (\ref{pol_P}) has $N$ distinct (negative) roots.
\\
%
%N odd
%
(ii)  Let $N$ be odd, with $N=2n-1$, $n \in \mathbb{N}$. The polynomial $P(x)$ has a root given by
\begin{equation}
-(2N+1) \mu\equiv -(4n-1)\mu,
\label{sol_N_odd}
\end{equation}
since
$$
P\left(-(4n-1) \mu\right)=-4^n \mu^{2n-1} (4n-1) \Gamma\left(2n-\frac{1}{2}\right)\sin(n \pi) \, \pi^{-1/2}=0.
$$
So, for $n=1$ (i.e.\ $N=1$) the unique root of $P(x)$ is (\ref{sol_N_odd}). We now focus on the case $n=2,3,\ldots,M$. In addition to the solution (\ref{sol_N_odd}), the remaining $2n-2$ roots can be obtained by applying the Intermediate Zero Theorem 
to the following intervals, having negative extremes:
$$
\left(-(2N+1)\mu-2  (2k+1)\mu, -(2N+1)\mu-2 (2(k-1)+1)\mu\;\right),\qquad k=1,2,\ldots,n-1,
$$
% and also $n-1$ intervals:
$$
\left(-(2N+1)\mu+2 (2k+1)\mu, -(2N+1)\mu+2 (2(k-1)+1)\mu\;\right),\qquad k=1,2,\ldots,n-1.
$$
Following the same procedure adopted for $N$ even, we can conclude that the polynomial $P(x)$ 
%assumes opposite signs in the interval extremes, so it has at least one solution in each interval of type (\ref{N_even_B}).
%\end{itemize}
% the polynomial $P(x)$ defined in (\ref{pol_P}) 
has $N$ distinct (negative) solutions also when $N$ is odd. This concludes the proof of Lemma \ref{lemma_pol}.
\hfill $\Box$

%%%%%%%%%%%%%%%%%%%%%%%%%%%
%%%%%%%%%%%%%%%%%%%%%%%%%
% \end{itemize}
%fine dimostrazione sul polinomio
%%%%%%%%%%%%%%%%%%%
%%%%%%%%%%%%%%%%%%%%%%%%%%%

\bigskip\noindent
{\bf Proof of Proposition \ref{asymptotic}}

\smallskip\noindent
Recall that the Laplace transform of $p(0,t)$, denoted by $H(\eta)$, is given in (\ref{p_Ltransf}). 
Hence, due to (\ref{solFgenerale_lambda_neq_mu}), the Laplace transform of $F(z,t)$ can be expressed as 
{\small
\begin{eqnarray}\label{F_Ltransf_gen}
{\cal L}_\eta\left[F(z,t)\right] 
\!\!\!\!
&=& \!\!\!\! \int_0^{\infty} e^{-\eta t} F(z,t) dt
=\frac{1}{(\lambda+\mu)^{2N}z^N}{\cal L}_\eta\left[\left((\mu z-\mu)e^{-t(\lambda+\mu)}+(z \lambda+\mu)\right)^N\left((\lambda-\lambda z)e^{-t(\lambda+\mu)}+(z \lambda+\mu)\right)^N\right]\nonumber\\
&- & \!\!\!\! \frac{\mu N(1-z)}{z^N (\lambda+\mu)^{2N-1}} \,H(\eta)\,
 {\cal L}_\eta\left[e^{-2Nt(\lambda+\mu)}\left((\mu z-\mu)+(\lambda z+\mu)e^{t(\lambda+\mu)}\right)^{N-1}\left((\lambda-\lambda z)+(\lambda z+\mu)e^{t(\lambda+\mu)}\right)^{N} \right] \nonumber \\
&=& \!\!\!\! \frac{1}{(\lambda+\mu)^{2N}z^N}{\cal L}_\eta\left[\left((\mu z-\mu)e^{-t(\lambda+\mu)}+(\lambda z+\mu)\right)^N\left((\lambda-\lambda z)e^{-t(\lambda+\mu)}+(\lambda z+\mu)\right)^N\right]\nonumber\\
&-& \!\!\!\! \frac{\mu N(1-z)}{z^N (\lambda+\mu)^{2N-1}} \,H(\eta) \left\{(\lambda z+\mu){\cal L}_{\eta+\lambda+\mu}\left[\left((\mu z-\mu)e^{-t(\lambda+\mu)}+(\lambda z+\mu)\right)^{N-1}\left((\lambda-\lambda z)e^{-t(\lambda+\mu)}+(\lambda z+\mu)\right)^{N-1}\right]\right.\nonumber \\
&-& \!\!\!\! \left.(\lambda z-\lambda){\cal L}_{\eta+2(\lambda+\mu)}\left[\left((\mu z-\mu)e^{-t(\lambda+\mu)}+(\lambda z+\mu)\right)^{N-1}\left((\lambda-\lambda z)e^{-t(\lambda+\mu)}+(\lambda z+\mu)\right)^{N-1}\right]
\right\}. 
\end{eqnarray}
}
We need to compute the following Laplace transform, for $n \in \mathbb{N}$:
\begin{eqnarray}
&&\hspace{-1cm}
{\cal L}_\eta\left[\left((\mu z-\mu)e^{-t(\lambda+\mu)}+(\lambda z+\mu)\right)^n\left((\lambda-\lambda z)e^{-t(\lambda+\mu)}+(z \lambda+\mu)\right)^n\right]\nonumber\\
&&=(\lambda \mu)^n(z-1)^{2n}\left(\frac{\lambda z+\mu}{\mu z-\mu}\right)^n\left(\frac{\lambda z+\mu}{\lambda z-\lambda}\right)^n {\cal L}_\eta\left[\left(1+\frac{e^{-t(\lambda+\mu)}}{\frac{\lambda z+\mu}{\mu z-\mu}}\right)^n\left(1-\frac{e^{-t(\lambda+\mu)}}{\frac{\lambda z+\mu}{\lambda z-\lambda}}\right)^n\right]\nonumber\\
&&=(\lambda z+\mu)^{2n}\sum_{j=0}^n\sum_{k=0}^n {n \choose j}{n \choose k}\left(\frac{\mu z-\mu}{\lambda z+\mu}\right)^j\left(-\frac{\lambda z-\lambda}{\lambda z+\mu}\right)^k{\cal L}_\eta\left[e^{-t(\lambda+\mu)(k+j)}\right]\nonumber\\
&&=(\lambda z+\mu)^{2n}\sum_{j=0}^n {n \choose j}\left(\frac{\mu z-\mu}{\lambda z+\mu}\right)^j\frac{1}{\eta+j(\lambda+\mu)}{}_{2}F_{1}\left(-n,j+\frac{\eta}{\lambda+\mu},j+1+\frac{\eta}{\lambda+\mu},\frac{\lambda z-\lambda}{\lambda z+\mu}\right);\nonumber
\end{eqnarray}
so the expression (\ref{F_Ltransf_gen}), for (\ref{p_Ltransf}), becomes: 
%
%\small
\begin{eqnarray}\label{F_Ltransf_gen_bis}
&&\hspace{-1cm}
{\cal L}_\eta\left[F(z,t)\right]=\frac{(\lambda z+\mu)^{2N}}{(\lambda+\mu)^{2N}z^N} \left[\frac{1}{\eta}{}_{2}F_{1}\left(-N,\frac{\eta}{\lambda+\mu},1+\frac{\eta}{\lambda+\mu},\frac{\lambda z-\lambda}{\lambda z+\mu}\right)\right. \nonumber\\
&&\left. +\sum_{j=1}^N {N \choose j}\mu^j\left(\frac{ z-1}{\lambda z+\mu}\right)^{j}\frac{1}{\eta+j(\lambda+\mu)}{}_{2}F_{1}\left(-N,j+\frac{\eta}{\lambda+\mu},j+1+\frac{\eta}{\lambda+\mu},\frac{\lambda z-\lambda}{\lambda z+\mu}\right)\right]
\nonumber\\
&&+\frac{\mu N (\lambda z+\mu)^{2N}}{z^N (\lambda+\mu)^{2N-1}}  \\
&&\times \frac{1}{\eta}\frac{\mu{}_{2}F_{1}\left(-N,\frac{\eta}{\lambda+\mu},1+N+\frac{\eta}{\lambda+\mu},-\frac{\lambda}{\mu}+\right)}{\mu {}_{2}F_{1}\left(1-N,1+\frac{\eta}{\lambda+\mu},1+N+\frac{\eta}{\lambda+\mu},-\frac{\lambda}{\mu}+\right)+\frac{\lambda(\lambda+\mu+\eta)}{(\lambda+\mu)(N+1)+\eta}{}_{2}F_{1}\left(1-N,2+\frac{\eta}{\lambda+\mu},2+N+\frac{\eta}{\lambda+\mu},-\frac{\lambda}{\mu}+\right) }\nonumber\\
&&\times \left[\sum_{j=0}^{N-1} {N-1 \choose j}\mu^j\left(\frac{ z-1}{\lambda z+\mu}\right)^{j+1}\frac{1}{\eta+(1+j)(\lambda+\mu)}{}_{2}F_{1}\left(-N+1,j+1+\frac{\eta}{\lambda+\mu},j+2+\frac{\eta}{\lambda+\mu},\frac{\lambda z-\lambda}{\lambda z+\mu}\right)\right.\nonumber\\
&&\left.-\lambda \sum_{j=0}^{N-1} {N-1 \choose j}\mu^j\left(\frac{ z-1}{\lambda z+\mu}\right)^{j+2}\frac{1}{\eta+(2+j)(\lambda+\mu)}{}_{2}F_{1}\left(-N+1,j+2+\frac{\eta}{\lambda+\mu},j+3+\frac{\eta}{\lambda+\mu},\frac{\lambda z-\lambda}{\lambda z+\mu}\right)\right].\nonumber
\end{eqnarray}
Hence, recalling that  $F(z)= \lim_{\eta\to 0} \eta\,{\cal L}_\eta[F(z,t)]$ by the Tauberian theorem (see Chapter VIII of 
Bhattacharya and Waymire \cite{Bhattacharya}), and making use of (\ref{F_Ltransf_gen_bis}), we have
\begin{eqnarray}
F(z) & = & \frac{(\lambda z+\mu)^{2N}}{(\lambda+\mu)^{2N}z^N}+\frac{\mu N (\lambda z+\mu)^{2N}}{z^N (\lambda+\mu)^{2N}}
\nonumber\\
&&\times \frac{\mu(N+1)}{\mu(N+1){}_{2}F_{1}\left(1-N,1,1+N,-\frac{\lambda}{\mu}\right)+\lambda {}_{2}F_{1}\left(1-N,2,2+N,-\frac{\lambda}{\mu}\right)}\nonumber\\
&&\times  \left[\sum_{j=0}^{N-1} {N-1 \choose j}\mu^j\left(\frac{ z-1}{\lambda z+\mu}\right)^{j+1}\frac{1}{1+j}\,{}_{2}F_{1}\left(-N+1,j+1,j+2,\frac{\lambda z-\lambda}{\lambda z+\mu}\right)\right.\nonumber\\
&&\left.-\lambda \sum_{j=0}^{N-1} {N-1 \choose j}\mu^j\left(\frac{ z-1}{\lambda z+\mu}\right)^{j+2}\frac{1}{2+j}\,{}_{2}F_{1}\left(-N+1,j+2,j+3,\frac{\lambda z-\lambda}{\lambda z+\mu}\right)\right].
\label{F_Ltransf_asympt_gen_1}
\end{eqnarray}
Due to the definition of the Hypergeometric function, after some calculation it is possible to simplify (\ref{F_Ltransf_asympt_gen_1}) to obtain (\ref{asymp_Fz_expr}). 
\hfill $\Box$
%%%%%%%%%%%%%%%%%%%%%%%%%%%%
%%%%%%%%%%%%%%%%%%%
%%%%%%%%%%%%%%%%%%%%%%%%%%%

\bigskip\noindent
{\bf Proof of Proposition \ref{stationaryprob}}

\smallskip\noindent
In Eq.\ (\ref{asymp_Fz_expr}) we make use of the following series expansions:
$$
\frac{(1+\varrho z)^{2N}}{z^N}=\sum_{r=1}^N {2N \choose N-r}\varrho^{N-r}z^{-r}+\varrho^{N}{2N \choose N}+\sum_{r=1}^N {2N \choose N+r}\varrho^{N+r} z^r,
$$
and 
\begin{eqnarray}
&&\frac{(1+\varrho z)^{2N}}{z^N}\left(\frac{ z-1}{1+\varrho z}\right)^{j+1}\frac{1}{1+j}\;{}_{2}F_{1}\left(-N+1,j+1,j+2,\frac{\varrho (z-1)}{1+\varrho z}\right)\nonumber\\
&&=(j+1)\sum_{h=0}^{N-1}\sum_{r=0}^{h+j+1}\sum_{s=0}^{2N-r}\frac{(-1)^{h+r}}{j+h+1}{N-1 \choose h}{h+j+1 \choose r}{2N-r \choose s}\frac{(1+\varrho)^r}{\varrho^{j+1-s}}z^{s-N}.\nonumber
\end{eqnarray}
Hence, after some calculations, the series expansion of (\ref{asymp_Fz_expr}) becomes
\begin{eqnarray}
F(z)&=&
\rho(0)+\sum_{s=1}^N z^{s} \varrho(s)
\nonumber
\\
&=&\frac{\varrho^{N} }{(1+\varrho)^{2N}}{2N \choose N}+\frac{N \varrho^{N-1} }{(1+\varrho)^{2N}}g(\varrho,N)\nonumber\\
&&\times\sum_{j=0}^{N-1}\sum_{h=0}^{N}\sum_{r=0}^{h+j+1}\frac{(-1)^{h+r}}{j+h+1}{N-1 \choose j}{N \choose h}{h+j+1 \choose r}{2N-r \choose N}\frac{\left(1+\varrho\right)^r}{\varrho^j}\nonumber\\
%&&+\sum_{s=1}^N z^{-s}\left\{\frac{(\lambda \mu)^{N} }{(\lambda+\mu)^{2N}}\left(\frac{\mu}{\lambda}\right)^s {2N \choose N-s}+\frac{N (\lambda \mu)^{N} }{(\lambda+\mu)^{2N}}\frac{\mu}{\lambda}g(\lambda,\mu,N)\right.\nonumber\\
%&&\left.\sum_{j=0}^{N-1}\sum_{k=0}^{N}\sum_{r=0}^{k+j+1}\frac{(-1)^{k+r}}{j+k+1}{N-1 \choose j}{N \choose k}{k+j+1 \choose r}{2N-r \choose N-s}(\frac{\mu}{\lambda}+1)^r\left(\frac{\mu}{\lambda}\right)^(j+s)\right\}\nonumber\\
&&+\sum_{s=1}^N z^{s}\varrho^s\left\{\frac{\varrho^{N} }{(1+\varrho)^{2N}}{2N \choose N+s}+\frac{N \varrho^{N-1} }{(1+\varrho)^{2N}}g(\varrho,N)\right.\nonumber\\
&&\times\left.\sum_{j=0}^{N-1}\sum_{h=0}^{N}\sum_{r=0}^{N-s}\frac{(-1)^{h+r}}{j+h+1}{N-1 \choose j}{N \choose h}{h+j+1 \choose r}{2N-r \choose N+s}\frac{(1+\varrho)^r}{\varrho^j}\right\},\nonumber
\end{eqnarray}
where the function $g$ is defined in (\ref{g_1}),   so   that 
\begin{eqnarray}
\rho(k)&=&\frac{\varrho^{N+k} }{(1+\varrho)^{2N}} {2N \choose N+k}+\frac{N \varrho^{N+k-1} }{(1+\varrho)^{2N}}g(\varrho,N)
\nonumber\\
&&\times\sum_{j=0}^{N-1}\sum_{h=0}^{N}\sum_{r=0}^{N-k}\frac{(-1)^{h+r}}{j+h+1}{N-1 \choose j}{N \choose h}{h+j+1 \choose r}{2N-r \choose N+k}\frac{\left(1+\varrho\right)^r}{\varrho^j}.
\label{plimit_lambda_neq_mu}
\end{eqnarray}
Noting that
$$
\sum_{r=0}^{N-k}(-1)^r{h+j+1 \choose r}{2N-r \choose N+k}\left(1+\varrho\right)^r={2N \choose k+N}{}_{2}F_{1}(-1-h-j,k-N;-2N;1+\varrho),
$$
the expression (\ref{plimit_lambda_neq_mu}) becomes
\begin{eqnarray}
\rho(k)&=&\frac{\varrho^{N+k} }{(1+\varrho)^{2N}}{2N \choose N+k}+\frac{N \varrho^{N+k-1} }{(1+\varrho)^{2N}}g(\varrho,N)\nonumber\\
&&\times\sum_{j=0}^{N-1}\sum_{h=0}^{N}\frac{(-1)^{h}}{j+h+1}{N-1 \choose j}{N \choose h}{2N \choose k+N}\frac{1}{\varrho^j}{}_{2}F_{1}\left(-1-h-j,k-N;-2N;1+\varrho\right).
\label{plimit_lambda_neq_mu_2}
\end{eqnarray}
Moreover, due to Eq. (5.92.12) of Prudnikov et al.\ \cite{Prudnikov}, after some manipulations one has 
$$
\sum_{h=0}^{N}\frac{(-1)^{h}}{j+h+1}{N \choose h}{}_{2}F_{1}\left(-1-h-j,k-N;-2N;1+\varrho\right)=\frac{N! j!}{(j+1+N)!},
$$
and thus the expression (\ref{plimit_lambda_neq_mu_2}) becomes
\begin{eqnarray}
\rho(k)&=&\frac{\varrho^{N+k} }{(1+\varrho)^{2N}} {2N \choose N+k}+\frac{N \varrho^{N+k-1} }{(1+\varrho)^{2N}}g(\varrho,N)
{2N \choose N+k}\sum_{j=0}^{N-1}{N-1 \choose j}\frac{N! j!}{(j+1+N)!}\frac{1}{\varrho^j}\nonumber\\
&=&\frac{\varrho^{N+k} }{(1+\varrho)^{2N}} {2N \choose N+k}+\frac{N \varrho^{N+k-1} }{(1+\varrho)^{2N}}g(\varrho,N)
{2N \choose N+k}\frac{{}_{2}F_{1}\left(1,1-N;N+2;-\frac{1}{\varrho}\right)}{N+1}.
\label{plimit_lambda_neq_mu_3}
\end{eqnarray}
Finally, recalling Eq.\ (\ref{g_1}) we obtain
\begin{eqnarray}
\rho(k)&=&\frac{\varrho^{N+k} }{(1+\varrho)^{2N}}{2N \choose N+k}
\left[1+\frac{N}{N+1}\cdot \frac{1}{\varrho}
\cdot \frac{{}_{2}F_{1}\left(1,1-N;N+2;-\frac{1}{\varrho}\right)}{{}_{2}F_{1}\left(1,-N;N+1;-\varrho\right)}\right].
\label{plimit_lambda_neq_mu_sempl2}
\end{eqnarray}
Hereafter we  show that, if $k=0$, then the stationary probability (\ref{eq:rhok}) is given by
\begin{equation}
\rho(0)=g(\varrho,N).
\label{remark}
\end{equation}
For Equation 7.3.1.143 of Prudnikov et al.\cite{Prudnikov}, one has  
% 07.23.03.0121.01 of Wolfram
%
$$
{}_{2}F_{1}\left(1,1-N;2+N;-\frac{1}{\varrho}\right)=\frac{(N-1)!(N+1)!}{(2N)!}\left(1+\frac{1}{\varrho}\right)^{N-1}P_{N-1}^{(N+1,-N)}\left(\frac{\varrho-1}{1+\varrho}\right),
$$
where $P_n^{(\alpha,\beta)}(z)$ is the Jacobi Polynomial. Therefore, from this last equality and (\ref{plimit_lambda_neq_mu_sempl2}) with $k=0$, it results:
\begin{equation}
\rho(0)=g(\varrho,N)\left[\frac{1}{g(\varrho,N)}\frac{\varrho^N}{(1+\varrho)^{2N}}{2N \choose N}+\frac{1}{(1+\varrho)^{N+1}}P_{N-1}^{(N+1,-N)}\left(\frac{\varrho-1}{1+\varrho}\right)\right].
\label{rho0_3}
\end{equation}
The thesis (\ref{remark}) thus follows by proving that the quantity in square brackets in (\ref{rho0_3}) is equal to $1$. 
Indeed, by using Equation 7.3.1.143 of Prudnikov et al.\cite{Prudnikov}  
for $g(\varrho,N)=1 /{}_{2}F_{1}\left(-N,1,1+N,-\varrho\right)$, one has
\begin{eqnarray}
 && \frac{1}{g(\varrho,N)}\frac{\varrho^N}{(1+\varrho)^{2N}}{2N \choose N}
+\frac{1}{(1+\varrho)^{N+1}}P_{N-1}^{(N+1,-N)}\left(\frac{\varrho-1}{1+\varrho}\right)
 \nonumber \\
 && \hspace{5cm} 
 = \frac{\varrho^N}{(1+\varrho)^N}P_{N}^{(N,-N-1)}\left(\frac{1-\varrho}{1+\varrho}\right)+
\frac{1}{(1+\varrho)^{N+1}}P_{N-1}^{(N+1,-N)}\left(\frac{\varrho-1}{1+\varrho}\right)
 \nonumber \\
 && \hspace{5cm} 
 = \frac{4^N\Gamma\left(N+\frac{1}{2}\right)}{\sqrt{\pi}\Gamma(N)}\left[B_{\frac{\varrho}{1+\varrho}}(N,N+1)
 +B_{\frac{1}{1+\varrho}}(N+1,N)\right]
 \nonumber\\
 && \hspace{5cm}
 =\frac{4^N\Gamma\left(N+\frac{1}{2}\right)}{\sqrt{\pi}\Gamma(N)}B(N,N+1) 
 =1,
 \nonumber
\end{eqnarray}
where $B(a,b)$ is the Beta function and $B_z(a,b)$ is the Incomplete Beta function, 
and where use of the formula in Section C of Chapter 1 of Gupta and Nadarajah\cite{GuNa} has been made. 
Finally, by comparing  (\ref{remark}) with (\ref{plimit_lambda_neq_mu_sempl2}) evaluated at $k=0$, 
Eq.\ (\ref{plimit_lambda_neq_mu_new}) immediately follows. 
\hfill $\Box$

\bigskip\noindent
{\bf Proof of Lemma \ref{lemma_appr_g}}

\smallskip\noindent
For Theorem 1.1 of Daalhuis \cite{Daalhuis}, if $\varrho<1$, for $N$ large one has
\begin{eqnarray}\label{approx_Hyp1}
{}_{2}F_{1}\left(1,-N,N+1,-\varrho\right) 
&\approx & \frac{2^N\left(1+\varrho\right)^{N-1}(N!)^2}{\varrho^{N/2}(2N)!\sqrt{2\pi}}\left\{ D_{-1}\left(\sqrt{2N\log\left[\frac{(1+\varrho)^2}{4\varrho}\right]}\right)\left(1+\varrho\right)\right.\nonumber\\
&&\left.+N^{-\frac{1}{2}}D_{0}\left(\sqrt{2N\log\left[\frac{(1+\varrho)^2}{4\varrho}\right]}\right)\left[\frac{\left(\varrho+1\right)+\frac{2(1+\varrho)}{\varrho-1} \sqrt{\log\left[\frac{(1+\varrho)^2}{4\varrho}\right]}}{-\sqrt{2}\;\sqrt{\log\left[\frac{(1+\varrho)^2}{4\varrho}\right]}}\right.\right.\nonumber\\
&&\left.\left.+\frac{1}{N}\frac{(1+\varrho)\left[2(\varrho-1)^3+(3\varrho+1)^2\left(\log\left[\frac{(1+\varrho)^2}{4\varrho}\right]\right)^{3/2}\right]}{4\sqrt{2}(\varrho-1)^3\left(\log\left[\frac{(1+\varrho)^2}{4\varrho}\right]\right)^{3/2}}\right]
\right\},
\end{eqnarray}
where $D_r(z)$ is the Parabolic Cylinder function. Since $D_{0}(z)=e^{-\frac{z^2}{4}}$ and 
$D_r(z)\approx z^r e^{-\frac{z^2}{4}}\left[1-\frac{(r-1)r}{2 z^2}+\frac{(r-3)(r-2)(r-1)r}{8z^4}\right]$ for $z\rightarrow \infty$, in our case it results 
$$
D_{0}\left(\sqrt{2N\log\left[\frac{(1+\varrho)^2}{4\varrho}\right]}\right)=2^N\left[\frac{(1+\varrho)^2}{\varrho}\right]^{-N/2},
$$
and
\begin{eqnarray}
D_{-1}\left(\sqrt{2N\log\left[\frac{(1+\varrho)^2}{4\varrho}\right]}\right)&\approx& 2^{N-\frac{1}{2}}N^{-\frac{1}{2}}\left(\log\left[\frac{(1+\varrho)^2}{4\varrho}\right]\right)^{-\frac{1}{2}}\left[\frac{(1+\varrho)^2}{\varrho}\right]^{-\frac{N}{2}}\nonumber\\
&&\times \left[\frac{4 N^2 \left(\log\left[\frac{(1+\varrho)^2}{4\varrho}\right]\right)^{2}-2N \log\left[\frac{(1+\varrho)^2}{4\varrho}\right]+3}{4 N^2 \left(\log\left[\frac{(1+\varrho)^2}{4\varrho}\right]\right)^{2}}\right].\nonumber
\end{eqnarray}
Hence, the expression (\ref{approx_Hyp1}) becomes 
\begin{eqnarray}\label{approx_Hyp2}
{}_{2}F_{1}\left(1,-N,N+1,-\varrho\right) &\approx &  \frac{2^{2N}(N!)^2 }{(2N)!\sqrt{2\pi N}(1+\varrho)}\nonumber\\
&&\times
\left\{ \left(2\log\left[\frac{(1+\varrho)^2}{4\varrho}\right]\right)^{-\frac{1}{2}}\left[\frac{4N^2 \left(\log\left[\frac{(1+\varrho)^2}{4\varrho}\right]\right)^2-2N\log\left[\frac{(1+\varrho)^2}{4\varrho}\right]+3}{4 N^2 \left(\log\left[\frac{(1+\varrho)^2}{4\varrho}\right]\right)^2}\right]\left(\varrho+1\right)\right.\nonumber\\
&&\left.-\frac{\varrho+1+\frac{2(1+\varrho)}{\varrho-1}\sqrt{\log\left[\frac{(1+\varrho)^2}{4\varrho}\right]}}{\sqrt{2\log\left[\frac{(1+\varrho)^2}{4\varrho}\right]}}+\frac{1}{N}\frac{(1+\varrho)\left[2(\varrho-1)^3+(3\varrho+1)^2\left(\log\left[\frac{(1+\varrho)^2}{4\varrho}\right]\right)^{3/2}\right]}{4\sqrt{2}(\varrho-1)^3\left(\log\left[\frac{(1+\varrho)^2}{4\varrho}\right]\right)^{3/2}}\right\}\nonumber\\
&&=\frac{2^{2N-3}(N!)^2}{(2N)!\sqrt{\pi}N^{\frac{5}{2}}}\;\frac{3(\varrho-1)^3+N \left(\log\left[\frac{(1+\varrho)^2}{4\varrho}\right]\right)^{\frac{5}{2}}\left[(3\varrho+1)^2-8N(\varrho-1)^2\right]}{(\varrho-1)^3\left(\log\left[\frac{(1+\varrho)^2}{4\varrho}\right]\right)^{\frac{5}{2}}},\nonumber
\end{eqnarray}
so that  the result (\ref{approx_g_1}) finally holds.
\hfill $\Box$
%%%%%%%%%%%%%%%%%%%%%%%%%%%%%
%%%%%%%%%%%%%%%%%%%%%%%%%%%


\begin{thebibliography}{99}
%--------------------------------------------------------------------
%
\bibitem{Abram1994}
Abramowitz M, Stegun IA.  
Handbook of Mathematical Functions with Formulas, Graph, and Mathematical Tables. 
New York: Dover; 1992.
%
\bibitem{Allen2020}
Allen B, Sample C, Jencks R, Withers J, Steinhagen P, Brizuela L, et al. 
Transient amplifiers of selection and reducers of fixation for death-birth updating on graphs. 
PLoS Comput Biol 2020;16(1): e1007529. 
%
\bibitem{Ascione} 
Ascione G, Mishura Y, Pirozzi E.  
Fractional Ornstein-Uhlenbeck process with stochastic forcing, and its applications. 
Methodol Comput Appl Probab 2021;23:53--84.
%
\bibitem{Balaji2010}
Balaji S, Mahmoud H, Tong Z. 
Phases in the diffusion of gases via the Ehrenfest urn model. 
J Appl Probab 2010;47:841--855.
%
\bibitem{Bhattacharya}
Bhattacharya R, Waymire EC. 
An explicit representation of the Luria--Delbr\"uck distribution. 
J Math Biol 2007;42:145--174.
%
\bibitem{Benichou}
Benichou O, Desbois J. 
Exit and occupation times for Brownian motion on graphs with general drift and diffusion constant.
J Phys A Math Theor 2009;42:015004.
%
%\bibitem{BrRy2008}
%\red{
%secondo me si puo' eliminare}
%??? Broom M, Rycht\'a$\check{\rm r}$ J () 
%An analysis of the fixation probability of a mutant on special classes of non-directed graphs.
%Proc R Soc A 2008;464:2609--2627,
%with addendum in: Proc R Soc A  2010;466:2795--2798.
%
\bibitem{BuCaNoPi2015}
Buonocore A, Caputo L, Nobile AG, Pirozzi E. 
Restricted Ornstein-Uhlenbeck process and applications in neuronal models with periodic input signals.
J Comput Appl Math 2015;285: 59--71.
%
\bibitem{ChPoZhCa2005}
Chen A, Pollett P, Zhang H, Cairns B.
Uniqueness criteria for continuous-time Markov chains with general transition structures. 
Adv Appl Probab  2005;37:1056--1074. 
%
\bibitem{CrSu2012}
Crawford FW, Suchard MA. 
Transition probabilities for general birth-death processes with applications in ecology, 
genetics, and evolution. 
J Math Biol 2012;65:553--580. 
%%
\bibitem{Csaki2016}
Cs\'aki E, Cs\"org\H o M, F\"oldes A, R\'ev\'esz P.
Some limit theorems for heights of random walks on a spider. 
J Theor Probab 2016;29:1685--1709. 
%
\bibitem{Daalhuis}
Daalhuis ABO. 
Uniform asymptotic expansions for hypergeometric functions with large parameters I.
Analysis and Applications 2003;1(1):111--120.
%
\bibitem{Dassios2020}
Dassios A, Zhang J.
Parisian time of reflected Brownian motion with drift on rays and its application in banking.
Risks 2020:8,127; doi:10.3390/risks8040127
%
\bibitem{Dharmaraja2015}
Dharmaraja S, Di Crescenzo A, Giorno V, Nobile AG. 
A continuous-time Ehrenfest model with catastrophes and its jump-diffusion approximation.
J Stat Phys  2015;161(2):326--345. 
%
\bibitem{DiCrescenzo2016}
Di Crescenzo A, Martinucci B, Rhandi A. 
A multispecies birth-death-immigration process and its diffusion approximation. 
J  Math  Anal  Appl  2016;442(1):291--316.
%
\bibitem{DiCrescenzo2021}
Di Crescenzo A, Paraggio P, Rom\'an-Rom\'an P, Torres-Ruiz F. 
Applications of the multi-sigmoidal deterministic and stochastic logistic models for plant dynamics. 
Appl  Math Modelling 2021;92:884--904.
%
\bibitem{Flegg2008}
Flegg MB, Pollett PK, Gramotnev DK. 
Ehrenfest model for condensation and evaporation processes
in degrading aggregates with multiple bonds. 
Phys Rev E 2008;78:031117 
%
\bibitem{FrWe93}
Freidlin MI, Wentzell AD. 
Diffusion processes on graphs and the averaging principle.
Ann Probab 1993;21:2215--2245.
%
%\bibitem{Giorno1986}
%Giorno V, Nobile AG, Ricciardi LM (1986)
%On some diffusion approximations to queueing systems. 
%Adv. in Appl. Probab. 18, no. 4, 991--1014
%
\bibitem{GiornoNN1985}
Giorno V, Negri C, Nobile AG.  
A solvable model for a finite-capacity queueing system
J Appl Prob 1985;22:903--911.
%
\bibitem{Giorno2012}
Giorno V, Nobile AG, di Cesare R. 
On the reflected Ornstein-Uhlenbeck process with catastrophes.
Appl Math Comput 2012;218(23):11570--11582. 
%
\bibitem{Giorno2014}
Giorno V, Spina S. 
On the return process with refractoriness for a non-homogeneous Ornstein-Uhlenbeck neuronal model. 
Math Biosci Eng 2014;11(2):285--302. 
%
\bibitem{Granovsky1997}
Granovsky BL, Zeifman AI.  
The decay function of nonhomogeneous birth-death processes, with application to mean-field models. 
Stoch Proc Appl 1997;72:105--120. 
%
\bibitem{GuNa}
Gupta AK,  Nadarajah S.  
Handbook of Beta Distribution and Its Applications; 1st edition.
CRC Press, 2004.
%
\bibitem{Hauert2004}
Hauert Ch, Nagler J, Schuster HG. 
Of dogs and fleas: the dynamics of $N$ uncoupled two-state systems.
J Stat Phys 2004;116:1453--1469. 
%
\bibitem{Hongler2019}
Hongler MO,  Filliger R. 
On jump-diffusive driving noise sources. 
Methodol Comput Appl Probab  2019;21:753--764. 
%
\bibitem{Huang2014}
Huang G, Mandjes M, Spreij P. 
Limit theorems for reflected Ornstein--Uhlenbeck processes. 
Stat Neerlandica 2014;68:25--42. 
%
\bibitem{Kaveh2015}
Kaveh K, Komarova NL, Kohandel M. 
The duality of spatial death--birth and birth--death processes and limitations of the isothermal theorem. 
R Soc Open Sci 2015;2:140465. 
%
\bibitem{Kostrykin2012}
Kostrykin V, Potthoff J, Schrader R. 
Construction of the paths of Brownian motions on star graphs II.
Commun Stoch  Analysis 2012;6(2):Article 5. 
%
\bibitem{Lansky2007}
Lansky P, Sacerdote L, Zucca C. 
Optimum signal in a diffusion leaky integrate-and-fire neuronal model. 
Math Biosci  2007;207(2):261--274. 
%
\bibitem{LatoucheRamaswami}
Latouche G, Ramaswami V. 
Introduction to Matrix Analytic Methods in Stochastic Modeling. 
ASA-SIAM Series on Statistics and Applied Probability. 
SIAM, Philadelphia, PA; American Statistical Association, Alexandria, VA, 1999.  
%
\bibitem{PaPaLe12}
Papanicolaou VG, Papageorgiou EG, Lepipas DC. 
Random motion on simple graphs.
Method Comput Appl Prob 2012;14:285--297,  
with addendum in Method Comput Appl Prob 2013;15:713. 
%
\bibitem{Peliti1985}
Peliti L.
Path integral approach to birth-death processes on a lattice. 
J Physique 1985;46:1469--1483. 
%
\bibitem{Prudnikov}  
Prudnikov AP, Brychkov YA,  Marichev OI. 
Integrals and Series: More Special Functions 3. Gordon \& Breach Science Publishers, 1990. 
%
\bibitem{PrudnikovVol4}  
Prudnikov AP, Brychkov YA,  Marichev OI. 
Integrals and Series. Vol.\ 4: Direct Laplace transforms.
London: Routledge,  1992.
%
\bibitem{Prudnikov5}
Prudnikov AP, Brychkov YuA, Marichev OI.
Integrals and Series: Inverse Laplace Transforms.  Vol 5. 
New York: Gordon \& Breach Science Publishers, 1992.
%
\bibitem{Ratanov}
Ratanov N.
Ornstein-Uhlenbeck processes of bounded variation.
Methodol Comput Appl Probab  2021;23:925-946.
%
\bibitem{Ri79}
Ricciardi LM, Sacerdote L.  
The Ornstein-Uhlenbeck process as a model for neuronal activity.
Biol Cybern 1979;35:1--9. 
%
\bibitem{Sui2015}
Sui X, Wu B, Wang L. 
Speed of evolution on graphs. 
Phys Rev E 2015;92:062124. 
%
\bibitem{Takahashi2004}
Takahashi H. 
Ehrenfest model with large jumps in finance. 
Phys. D 2004;189:61--69: 
%
\bibitem{We01}
Weber M. 
On occupation time functionals for diffusion processes and birth-and-death processes on graphs. 
Ann Appl Probab 2001;11:544--567. 
With correction note in:  Ann Appl Probab  2001;11:1003. 
%
\end{thebibliography}
\end{document}